%% file: ArXiv_QBIRKA.tex
\documentclass{siamart0516}

\input{ex_shared}

\ifpdf
\hypersetup{
  pdftitle={\TheTitle},
  pdfauthor={\TheAuthors}
}
\fi

\usepackage{soul}

\begin{document}

\maketitle

\begin{abstract}
  \input{abstract.tex}
\end{abstract}


\
 \input{introduction.tex}
 \input{TensorTheory.tex}
\input{h2_norm_QB_sys.tex}

\input{optimality_conditions.tex}

\input{numerical_results.tex}
\section{Conclusions}\label{sec:conclusions}
\input{conclusion.tex}

\section*{Acknowledgments}
\input{Acknowledgements.tex}

\bibliographystyle{siamplain}
\bibliography{mor}


\input{Appendix1.tex}
\input{Appendix.tex}
\input{proof_red_optimality_rel.tex}

\end{document}

%% file: ex_shared.tex
\usepackage{lipsum}
\usepackage{amsfonts}
\usepackage{graphicx}
\usepackage{epstopdf}
\usepackage{algorithm} 
\usepackage{algpseudocode}
\usepackage{tikz}
\usepackage{pgfplots}

\numberwithin{equation}{section}
\numberwithin{theorem}{section}
\numberwithin{figure}{section}
\numberwithin{table}{section}
\numberwithin{algorithm}{section}
\usetikzlibrary{external}
\newcommand{\includetikz}[1]{%
	\tikzsetnextfilename{Figures/#1}%
	\input{Figures/#1.tikz}%
}

\ifpdf
  \DeclareGraphicsExtensions{.eps,.pdf,.png,.jpg}
\else
  \DeclareGraphicsExtensions{.eps}
\fi

\newcommand{\TheTitle}{ \texorpdfstring{\boldmath{$\cH_2$}}{H2}-Quasi-Optimal Model Order Reduction for Quadratic-Bilinear Control Systems}
\newcommand{\TheShortTitle}{ {$\mathbf{\mathcal{H}_2}$}-Quasi-Optimal MOR for QB Control Systems}
\newcommand{\TheAuthors}{Peter Benner, Pawan Goyal, and Serkan Gugercin}

\headers{\TheShortTitle}{\TheAuthors}

\title{{\TheTitle}\thanks{Submitted to the editors \today.
}}

\author{
  Peter Benner\thanks{Max Planck Institute for Dynamics of Complex Technical Systems, Sandtorstra\ss e 1, 39106 Magdeburg, Germany; (\email{benner@mpi-magdeburg.mpg.de}).}
  \and
  Pawan Goyal\thanks{Corresponding author. Max Planck Institute for Dynamics of Complex Technical Systems, Sandtorstra\ss e 1, 39106 Magdeburg, Germany; (\email{goyalp@mpi-magdeburg.mpg.de}).}
  \and
  Serkan Gugercin\thanks{Department of Mathematics, Virginia Tech, Blacksburg, VA 24061-0123, United States;  (\email{gugercin@math.vt.edu}).}
}

\usepackage{amsopn}
\usepackage[software,hardware]{mymacros}
\newcommand{\0}{\mathbf 0}
\newcommand{\bbm}{\begin{bmatrix}}
	\newcommand{\ebm}{\end{bmatrix}}
\newcommand{\bpm}{\begin{pmatrix}}
	\newcommand{\epm}{\end{pmatrix}}

\newenvironment{customlegend}[1][]{%
	\begingroup
	\csname pgfplots@init@cleared@structures\endcsname
	\pgfplotsset{#1}%
}{%
\csname pgfplots@createlegend\endcsname
\endgroup
}%

\def\addlegendimage{\csname pgfplots@addlegendimage\endcsname}

\newlength\fheight
\newlength\fwidth

\usepackage{adjustbox}
\usepackage{pdflscape}

\usepackage[ntheorem]{empheq}
\usepackage[american]{babel}
\usepackage[title,titletoc,toc]{appendix}

\renewcommand{\hat}[1]{\widehat#1}
\renewcommand{\tilde}[1]{\widetilde#1}

\newtheorem{example}[theorem]{Example}
\newtheorem{remark}[theorem]{Remark}
\usepackage{enumitem}
\usepackage{amsmath}

\usepackage{array,multirow}

%% file: abstract.tex
We investigate the optimal model reduction problem for large-scale quadratic-bilinear (QB) control systems. Our contributions are  threefold. First, we discuss the variational analysis and the Volterra series formulation for QB systems. We then define the  $\mathcal H_2$-norm  for a QB system based on the  \emph{kernels} of the underlying Volterra series and  also propose  a truncated $\cH_2$-norm.  Next, we derive  first-order necessary conditions  for an optimal approximation, where optimality is measured in term of the truncated $\cH_2$-norm of the error system.  We then propose an iterative model reduction algorithm, which upon convergence yields a reduced-order system that \emph{approximately} satisfies the newly derived optimality conditions. We also discuss an efficient computation of the  reduced Hessian, using the special Kronecker structure of the Hessian of the system.  We illustrate the efficiency of the proposed method by means of several numerical examples resulting from semi-discretized nonlinear partial differential equations and show its competitiveness with the existing model reduction schemes for QB systems such as moment-matching methods and balanced truncation.

%% file: introduction.tex
\section{Introduction}  \label{sec:intro}

Numerical simulation is a fundamental tool in the analysis of dynamical systems, and is required repeatedly in the computational control design, optimization and prediction thereof. Dynamical systems are generally governed by partial differential equations (PDEs), or ordinary differential equations (ODEs), or a combination of both. A high-fidelity approximation of the underlying physical phenomena requires a finely discretized mesh over the interesting spatial domain,  leading to  complex dynamical systems with a high-dimensional  state-space. The simulations of such large-scale systems, however, impose a huge computational  burden.  This inspires \emph{model order reduction} (MOR), which aims at constructing simple and reliable surrogate models such that their input-output behaviors  approximate that of the original large-scale system accurately. These surrogate models can then be used in engineering studies, which make numerical simulations faster and efficient.

In recent decades, numerous theoretical and computational aspects for MOR of linear systems have been developed; e.g., see~\cite{morAnt05,baur2014model,morBenMS05,morSchVR08}. These methods have been successfully applied in various fields, e.g., optimal control and PDE constrained optimization  and uncertainty quantification; for example, see~\cite{BenSV14,morHinV08}. In recent years, however, MOR of nonlinear systems has gained significant interest with a goal of extending the input-independent, optimal MOR techniques from linear systems to nonlinear ones. For example, MOR techniques for linear systems such as balanced truncation~\cite{morAnt05,morMoo81}, or the iterative rational Krylov method (IRKA)~\cite{morGugAB08}, have been extended to a special class of nonlinear systems, so-called \emph{bilinear systems}, in which nonlinearity arises from the product of the state and  input~\cite{morBenB12b,morBenD11,flagg2015multipoint,morZhaL02}. However, there are numerous open and difficult problems for other important classes of nonlinear systems. In this
article, we address another vital class of nonlinear systems,  called \emph{quadratic-bilinear} (QB) systems. These are of the form:
\begin{equation}\label{sys:QBsys}
\Sigma :
\begin{cases}
\begin{aligned}
 \dot x(t) &= Ax(t) + H\left(x(t)\otimes x(t)\right) +\sum_{k=1}^m N_kx(t)u_k(t) + Bu(t),\\
 y(t)& = Cx(t), \quad x(0) = 0,
\end{aligned}
\end{cases}
\end{equation}
where $x(t) \in \Rn$, $u(t) \in \Rm$ and $y(t)\in\Rp$  are the states, inputs, and outputs of the systems, respectively, $u_k$ is the $k$th component of $u$, and $n$ is the state dimension which is generally very large. Furthermore, $A,N_k \in \Rnn$ for $k \in \{1,\ldots,m\}$, $H \in \R^{n\times n^2}$, $B \in \Rnm$, and $C \in \Rpn$.

There is a variety of applications where the system inherently contains a quadratic nonlinearity, which can be modeled in the QB form~\eqref{sys:QBsys}, e.g., Burgers' equation. Moreover, a large class of smooth nonlinear systems, involving combinations of elementary functions like exponential, trigonometric, and polynomial functions, etc.\ can be equivalently rewritten as  QB systems~\eqref{sys:QBsys} as shown in~\cite{morBenB15,morGu09}. This is achieved by  introducing some new appropriate state variables to simplify the  nonlinearities present in the underlying control system and deriving differential equations, corresponding to the newly introduced variables, or by using appropriate algebraic constraints. When algebraic constraints are introduced in terms of the state and the newly defined variables, the system contains algebraic equations  along with differential equations. Such systems are called \emph{differential-algebraic equations} (DAEs) or \emph{descriptor systems} \cite{kunkel2006differential}.
 MOR procedures for DAEs become inevitably more complicated, even in the linear and bilinear settings; e.g., see~\cite{MPIMD15-16,morGugSW13}.  In this article, we restrict ourselves to QB ODE systems and leave MOR for QB descriptor systems as a future research topic.

For a given QB system~\eqref{sys:QBsys} $\Sigma$ of dimension $n$, our aim  is to construct a reduced-order system
\begin{equation}\label{sys:QBsysRed}
\hat\Sigma :
\begin{cases}
\begin{aligned}
 \dot  \hx(t) &= \hA\hx(t) + \hH\left(\hx(t)\otimes \hx(t)\right) +\sum_{k=1}^m \hN_k\hx(t)u_k(t) + \hB u(t),\\
 \hy(t)& = \hC\hx(t),\quad \hx(0) = 0,
\end{aligned}
\end{cases}
\end{equation}
where $\hA,\hN_k \in \Rrr$ for $k \in \{1,\ldots,m\}$, $\hH \in \R^{r\times r^2}$, $\hB \in \R^{r\times m}$, and $\hC \in \R^{p\times r}$ with $r\ll n$ such that the outputs of the system~\eqref{sys:QBsys} and~\eqref{sys:QBsysRed}, $y$ and $\hy$, are very well approximated in a proper norm for all admissible inputs $u_k \in L^2[0,\infty[$. 

Similar to the linear and bilinear cases, we construct the reduced-order system~\eqref{sys:QBsysRed} via projection. Towards this goal, we construct two model reduction basis matrices $V,W\in \Rnr$  such that $W^TV$ is invertible. Then, the reduced matrices in~\eqref{sys:QBsysRed} are given by
\begin{equation*}
 \begin{aligned}
  \hA &= (W^TV)^{-1}W^TAV,& \hN_k &= (W^TV)^{-1}W^TN_kV,~~\text{for} ~~k \in \{1,\ldots,m\},\\
  \hH &= (W^TV)^{-1}W^TH(V\otimes V),\quad &\hB &= (W^TV)^{-1}W^TB,\quad\text{and}\quad \hC = CV.
 \end{aligned}
\end{equation*}
It can   be easily seen that the quality of the reduced-order system  depends on the choice of the reduction subspaces $\Ima V$ and $\Ima W$. There exist various MOR approaches in the literature to determine these subspaces. One of the earlier and popular methods for nonlinear systems is trajectory-based MOR such as proper orthogonal decomposition (POD); e.g., see \cite{morAstWWetal08,morChaS10,morHinV05,morKunV08}. This relies on the Galerkin projection $\cP = \cV\cV^T$, where $\cV$ is determined based on the dominate modes of the system dynamics. For efficient computations of nonlinear parts in the system, some advanced methodologies can be employed such as the empirical interpolation method (EIM), the best points interpolation method, the discrete empirical interpolation method (DEIM), see, e.g.~\cite{morAstWWetal08,morBarMNetal04,morGreMNetal07,morChaS10,morNguyenPP08}.  Another widely used method for nonlinear system is the trajectory piecewise linear (TPWL) method, e.g., see~\cite{morRew03}. For this method, the nonlinear system is replaced by a weighted sum of linear systems;  these linear systems can then be reduced by using well-known methods for linear systems such as balanced truncation, or interpolation methods, e.g., see~\cite{morAnt05}. However, the above mentioned methods require some snapshots or solution trajectories of the original systems for particular inputs. This indicates that the resulting reduced-order system depends on the choice of  inputs, which may make the reduced-order system inadequate in many applications such as control and optimization, where the variation of the input is  inherent to the problem.

MOR methods, based on interpolation or moment-matching, have been extended from linear systems to QB systems, with the aim of  capturing the input-output behavior of the underlying system independent of a training input. One-sided interpolatory projection for QB systems is studied in, e.g.,~\cite{bai2002krylov,morGu09,phillips2000projection,morPhi03}, and has been recently extended  to a two-sided interpolatory projection in~\cite{morBenB12a,morBenB15}. These methods result in reduced systems that do not rely on the training data for a control input; see also the survey  \cite{baur2014model} for some related approaches. Thus, the determined reduced systems can be used in input-varying applications. In the aforementioned  references, the authors have shown how to construct an interpolating reduced-order system for a given set of interpolation points. But it is still an open problem how to choose these interpolation points optimally with respect to an appropriate norm. Furthermore, the two-sided interpolatory projection method~\cite{morBenB15} is only applicable to single-input single-output systems, which is very restrictive, and additionally, the stability of the resulting reduced-order systems also remains another major issue.

Very recently, balanced truncation  has been extended from linear/bilinear systems to QB systems~\cite{morBTQBgoyal}.  This method first determines the states which are hard to control and observe and constructs the reduced model by truncating those states.  Importantly, it results in locally Lyapunov stable reduced-order systems, and an appropriate order of the reduced system can be determined based on the singular values of the Gramians of the system. But  unlike in the linear case the resulting reduced systems do not guarantee other desirable properties such as  an a priori  error bound.   Moreover, in order to apply balanced truncation to QB systems, we require the solutions of four conventional Lyapunov equations, which could be computationally cumbersome in large-scale settings; though there have been many advancements in recent times related to computing the low-rank solutions of Lyapunov equations~\cite{benner2013numerical,simoncini2016computational}.

In this paper, we study the $\cH_2$-optimal approximation problem for QB systems.  
Precisely, we show how to choose the model reduction bases in a  two-sided projection  interpolation framework for QB systems so that the reduced model is a local minimizer in an appropriate norm.
Our main contributions are threefold. In \Cref{sec:preliminarywork}, we derive various expressions and formulas related to Kronecker products, which are later heavily utilized in deriving  optimality conditions.  In \Cref{sec:optimalitycondition}, we first define the  $\cH_2$-norm of the QB system~\eqref{sys:QBsys} based on the  \emph{kernels} of its Volterra series (input/output mapping), and then derive an expression for a truncated $\cH_2$-norm for QB systems as well. Subsequently, based on a truncated $\cH_2$-norm  of the error system, we derive  first-order necessary conditions for optimal model reduction of QB systems.  We then propose an iterative algorithm to construct reduced models that \emph{approximately} satisfy the newly derived optimality conditions.  Furthermore, we discuss an  efficient alternative way to compute reduced Hessians as compared to the one proposed in~\cite{morBenB15}.
 In \Cref{sec:numerics}, we illustrate the efficiency of the proposed method for various semi-discretized nonlinear PDEs and compare it with existing  methods such as balanced truncation~\cite{morBTQBgoyal} as well as the one-sided and two-sided interpolatory projection methods for QB systems~\cite{morBenB15,morGu09}. We conclude the paper with a short summary and potential future directions in \Cref{sec:conclusions}.

\vspace{2ex}
\noindent{\bf Notation:} Throughout the paper, we make use of the following notation:
\begin{itemize}[leftmargin=2em,topsep=1ex,itemsep=-1ex,partopsep=1ex,parsep=1ex]
 \item $I_q$ denotes the identity matrix of size $q\times q$, and its $p$th column is denoted by $e_p^q$.
 \item $\vecop{\cdot}$ denotes vectorization of  a matrix, and $\cI_m$ denotes $\vecop{I_m}$.
 \item $\trace{\cdot}$ refers to the trace of a matrix.
 \item Using  \matlab ~notation, we denote the $j$th column of the matrix $A$  by $A(:,j)$.
 \item $\0$ is a zero matrix of appropriate size.
 \item We denote the full-order system~\eqref{sys:QBsys} and reduced-order system~\eqref{sys:QBsysRed} by $\Sigma$ and $\hat{\Sigma}$, respectively.
\end{itemize}

%% file: TensorTheory.tex
\section{Tensor Matricizations and their Properties} \label{sec:preliminarywork}
We first review some basic concepts from tensor algebra. 
First, we note the following important properties of the $\vecop{\cdot}$ operator: 
\begin{subequations}\label{eq:tracepro}
\begin{align}
 \trace{X^TY} &= \vecop{X}^T\vecop{Y} = \vecop{Y}^T\vecop{X},~\text{and}\label{eq:trace1}\\
 \vecop{XYZ} &= (Z^T\otimes X)\vecop{Y}. \label{eq:trace2}
 \end{align}
\end{subequations}
Next, we review the concepts of matricization of a tensor. Since the Hessian $H$ of the QB system in~\eqref{sys:QBsys} is closely related to  a $3$rd order tensor,  we focus on  $3$rd order tensors  $ \cX^{n\times n\times n}$. However, most of the concepts can be extended to a general $k$th order tensor.   Similar to how rows and columns are defined for a matrix, one can define a fiber of $\cX$ by fixing all indices except for one, e.g., $\cX(:,i,j),\cX(i,:,j)$ and  $\cX(i,j,:)$.  Mathematical operations involving tensors are  easier to perform with its corresponding matrices. Therefore, there exists a very well-known process of unfolding a tensor into a matrix, called \emph{matricization} of a tensor.  For a 3-dimensional tensor, there are $3$ different ways to unfold the tensor, depending on the mode-$\mu$ fibers that are used for the unfolding. If the tensor is unfolded using mode-$\mu$ fibers,  it is called the mode-$\mu$ matricization of $\cX$. We refer to~\cite{horn37topics,koldatensor09} for more details on these basic concepts of tensor theory.

In the following example, we illustrate how a $3$rd order tensor $\cX \in \R^{n\times n\times n}$ can be unfolded into different matrices.
 \begin{example}
 Consider a $3$rd order tensor $\cX^{n\times n\times n}$ whose frontal slices are given by matrices $X_i \in \Rnn$ as shown in \Cref{fig:tensor_pic}.

 \begin{figure}[h]
 \centering
 \includegraphics[scale =.5]{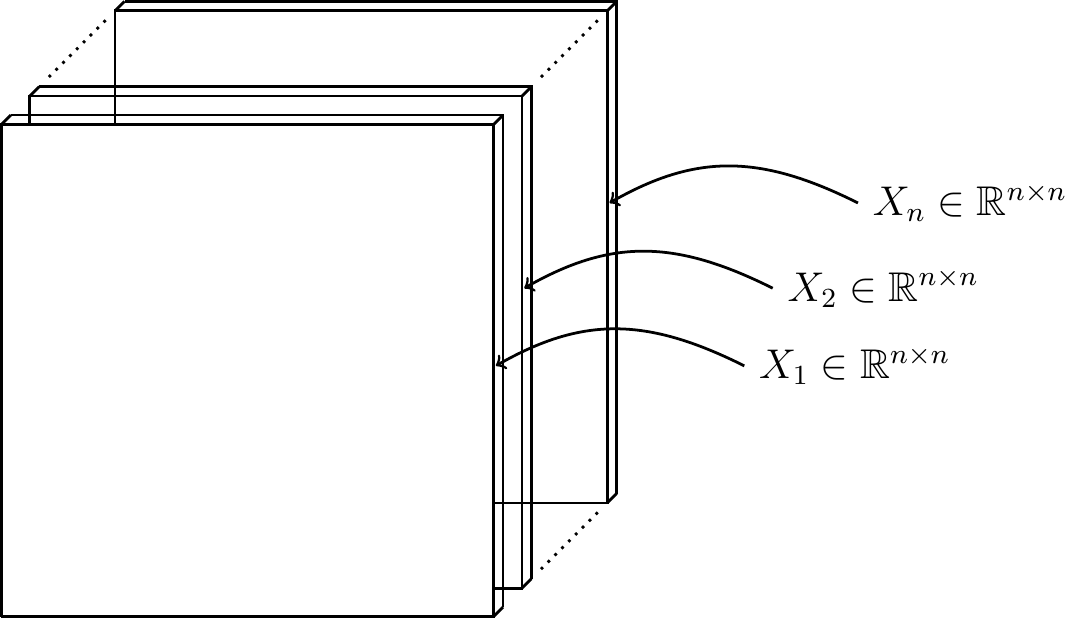}
 \caption{Representation of a tensor using the frontal slices~\cite{koldatensor09}.}
 \label{fig:tensor_pic}
 \end{figure}
Then, its mode-$\mu$ matricizations, $\mu \in \{1,2,3\}$, are given by
\begin{align*}
\cX^{(1)} & = [ X_1 ,X_2,\ldots,X_n],~~~
\cX^{(2)}  = [ X^T_1 ,X^T_2,\ldots,X^T_n],~\mbox{and}\\
\cX^{(3)} & = [ \vecop{X_1} ,\vecop{X_2},\ldots,\vecop{X_n}]^T.
\end{align*}
 \end{example}
 
 Similar to the matrix-matrix product, one can also perform a  tensor-matrix  or tensor-tensor multiplication. Of particular interest, we here only consider  tensor-matrix multiplications, which can be performed by means of matricizations; see, e.g.,~\cite{koldatensor09}. For a given tensor $\cX\in\R^{n\times n\times n}$ and  a matrix $\cA\in  \R^{n_1\times n} $, the $\mu$-mode matrix product is denoted by $\cX\times _\mu \cA =: \cY$, i.e., $\cY \in \R^{n_1\times n\times n}$ for $\mu = 1$. In the case of the $\mu$-mode matrix multiplication,  the mode-$\mu$ fiber is multiplied with the matrix $\cA$, which can be written as
\begin{equation*}
 \cY = \cX\times_\mu \cA \Leftrightarrow \cY^{(\mu)} = \cA\cX^{(\mu)}.
\end{equation*}
Furthermore, if a tensor is given as
\begin{equation}\label{eq:def_Z}
 \cZ = \cX \times_1 \cA \times_2 \cB \times_3 \cC,
\end{equation}
where $\cA\in \R^{n_1\times n}$, $\cB\in \R^{n_2\times n}$ and $\cC\in \R^{n_3\times n}$, then the mode-$\mu$ matriciziations  of $\cZ$ satisfy:
\begin{equation}\label{eq:matricization_relations}
 \begin{aligned}
  \cZ^{(1)} &= \cA\cX^{(1)}(\cC\otimes \cB)^T,~
  \cZ^{(2)} = \cB\cX^{(2)}(\cC\otimes \cA)^T,~
  \cZ^{(3)} = \cC\cX^{(3)}(\cB\otimes \cA)^T.
 \end{aligned}
\end{equation}
Using these properties of the tensor products, we now introduce our first result on tensor matricizations.
\begin{lemma}\label{lemma:trace_property}
 Consider tensors $\cX,\cY \in \R^{n\times n\times n}$  and let $\cX^{(i)}$ and $\cY^{(i)}$ denote, respectively, their mode-$i$ matricizations. Then,  
 \begin{equation*}
  \trace{\cX^{(1)} (\cY^{(1)})^T}=\trace{\cX^{(2)} (\cY^{(2)})^T}=\trace{\cX^{(3)} (\cY^{(3)})^T}.
 \end{equation*}
\end{lemma}
\begin{proof}
 We begin by denoting the $i$th frontal slice of $\cX$ and $\cY$ by $X_i$ and $Y_i$, respectively; see \Cref{fig:tensor_pic}. Thus,
   \begin{align*}
   \trace{\cX^{(1)}(\cY^{(1)})^T} &= \trace{\bbm X_1,X_2, \ldots,X_n\ebm\bbm Y_1,Y_2, \ldots,Y_n\ebm^T}\nonumber\\
   &= \sum_{i=1}^n\trace{ X_iY^T_i}  = \sum_{i=1}^n  \trace{X^T_iY_i}  \\
   &=\trace{\bbm X_1^T,X_2^T, \ldots,X_n^T\ebm\bbm Y_1^T,Y_2^T, \ldots,Y_n^T\ebm^T} 
   =\trace{\cX^{(2)} (\cY^{(2)})^T}.\nonumber
  \end{align*} 
 Furthermore, since $\trace{X^TY} = \vecop{X}^T\vecop{Y}$, this allows us to write
 \begin{align*}
   \trace{\cX^{(1)}(\cY^{(1)})^T} &=  \sum_{i=1}^n\trace{ X^T_iY_i}  = \sum_{i=1}^n \vecop{X_i}^T \vecop{Y_i}.
 \end{align*}
Since the $i$th rows of $\cX^{(3)}$ and $\cY^{(3)}$ are given by $\vecop{X_i}^T$ and $\vecop{Y_i}^T$, respectively, this means $\sum_{i=1}^n \vecop{X_i}^T \vecop{Y_i} = \trace{\cX^{(3)} (\cY^{(3)})^T}$. This concludes the proof.
\end{proof}

Recall that the Hessian $H$ in the QB system~\eqref{sys:QBsys} is of size $n\times n^2$; thus, it can  be interpreted as an unfolding of a tensor $\cH^{n\times n\times n}$. Without loss of generality, we  assume  the Hessian $H$ to be the mode-$1$ matricization of $\cH$, i.e., $H = \cH^{(1)}$. Also,  we assume $\cH$ to be symmetric. This means that for given vectors $u$ and $v$,
\begin{equation}\label{eq:H_symm}
\begin{aligned}
 H(u\otimes v) =  \cH^{(1)}(u\otimes v) &= \cH^{(1)}(v\otimes u) = H(v\otimes u).
 \end{aligned}
\end{equation}
This provides us additional information that the other two matricization modes of $\cH$ are the same, i.e., 
\begin{equation}\label{eq:H2H3_rel}
\cH^{(2)} = \cH^{(3)}.
\end{equation}
 In general, it is not necessary that the Hessian $H$ (mode-1 matrizication of the tensor $\cH$), obtained from the discretization of the governing PDEs satisfies~\eqref{eq:H_symm}. However, as shown in~\cite{morBenB15}, the Hessian $H$ can be modified in such a way that the modified Hessian $\tH$ satisfies~\eqref{eq:H_symm} without any change in the dynamics of the system; thus, for the rest of the paper,    without loss of generality, we assume that the tensor $\cH$ is symmetric.


The additional property that the Hessian is symmetric will allow us to derive some new relationships between matricizations and matrices, that will prove to be crucial ingredients in simplifying the expressions arising in the derivation of  optimality conditions in \Cref{sec:optimalitycondition}. 

\begin{lemma}\label{lemma:RelH1H1ABC}
 Let $\cH \in \R^{n\times n\times n}$  be a $3$rd order tensor, satisfying~\eqref{eq:H_symm} and \eqref{eq:H2H3_rel}, and consider matrices $\cA,\cB,\cC \in \Rnn $. Then,
 \begin{equation}\label{eq:H1AB_property}
  \cH^{(1)}(\cB\otimes \cC)\left(\cH^{(1)}\right)^T =  \cH^{(1)}(\cC\otimes \cB)\left(\cH^{(1)}\right)^T,
 \end{equation}
 and
 \begin{align*}
  \left( \vecop{\cB}\right)^T\vecop{\cH^{(2)}(\cC\otimes \cA)(\cH^{(2)})^T} &= \left(\vecop{ \cC}\right)^T\vecop{\cH^{(2)}(\cB\otimes \cA)(\cH^{(2)})^T} \\
 & = \left(\vecop{ \cA}\right)^T\vecop{\cH^{(1)}(\cC\otimes \cB)(\cH^{(1)})^T}.
 \end{align*}
\end{lemma}
\begin{proof}
 We begin by proving the  relation in~\eqref{eq:H1AB_property}. The order in the Kronecker product can be changed  via pre- and post-multiplication by  appropriate permutation matrices; see~\cite[Sec. 3]{henderson1981vec}. Thus,
 \begin{equation*}
  \cB\otimes \cC = S(\cC\otimes \cB)S^T,
 \end{equation*}
where $S$ is a permutation matrix, given as $S = \sum_{i=1}^n((e^n_i)^T\otimes I_n\otimes e^n_i)$. We can then write
\begin{equation}\label{eq:HAB}
  \cH^{(1)}(\cB\otimes \cC)\left(\cH^{(1)}\right)^T =  \cH^{(1)}S(\cC\otimes \cB)\left(\cH^{(1)}S\right)^T.
\end{equation}
We now  manipulate the term $\cH^{(1)}S$: 
	\begin{align}\label{eq:H1s_relation}
		\cH^{(1)}S &= \sum_{i=1}^n\cH^{(1)}((e^n_i)^T\otimes I_n\otimes e^n_i).
		\end{align}
Furthermore, we can write  $I_n$ in the Kronecker form as
\begin{equation}\label{eq:Identity_Kron}
 I_n = \sum_{j=1}^n(e^n_j)^T\otimes e^n_j ,
\end{equation}
 and since for a vector $f\in \R^q$, $f^T\otimes f = ff^T $, we can write \eqref{eq:Identity_Kron} in another form as
 \begin{equation}\label{eq:Identity_Kron2}
 I_n =  \sum_{j=1}^ne^n_j(e^n_j)^T. 
 \end{equation}
    Substituting  these relations in \eqref{eq:H1s_relation} leads to
		\begin{align}
		\cH^{(1)}S &= \sum_{i=1}^n\sum_{j=1}^n\cH^{(1)}((e^n_i)^T\otimes (e^n_j)^T\otimes e^n_j\otimes e^n_i)\nonumber\\
		&= \sum_{i=1}^n\sum_{j=1}^n\cH^{(1)}\left(e^n_j\otimes e^n_i\right)\left((e^n_i)^T\otimes (e^n_j)^T\right) && \left(\because  \text{for}~ f\in \R^q,f^T\otimes f = ff^T\right)\nonumber \\
		&= \sum_{i=1}^n\sum_{j=1}^n\cH^{(1)}(e^n_i\otimes e^n_j)((e^n_i)^T\otimes (e^n_j)^T). && \left(\because \text{the relation~}\eqref{eq:H_symm}\right) \label{eq:HS1}
		\end{align}
		Next, we use a tensor-multiplication property in the above equation, that is,
		\begin{equation}\label{eq:kronprodpro}
		(\cP\otimes \cQ)(\cR\otimes \cS) = (\cP\cR\otimes \cQ\cS),
		\end{equation}
where $\cP,\cQ,\cR$ and $\cS$ are of compatible dimensions. Using the Kronecker product property~\eqref{eq:kronprodpro} in \eqref{eq:HS1}, we obtain
		\begin{align*}
		\cH^{(1)}S 	 	&= \cH^{(1)}  \left(\sum_{i=1}^n  e^n_i (e^n_i)^T\otimes \sum_{j=1}^n  e^n_j (e^n_j)^T \right) \\
	& = \cH^{(1)}(I_n\otimes I_n) = \cH^{(1)}. \qquad &&\left(\text{from}~\eqref{eq:Identity_Kron2}\right) 
	\end{align*}
Substituting the above relation in~\eqref{eq:HAB} proves~\eqref{eq:H1AB_property}. For the second part, we utilize the trace property~\eqref{eq:trace1} to obtain
\begin{align*}
 \left(\vecop{\cB}\right)^T\vecop{\cH^{(2)}(\cC\otimes \cA)(\cH^{(2)})^T} & = \trace{\underbrace{\cB^T\cH^{(2)}(\cC\otimes \cA)}_{\cL^{(2)}}(\cH^{(2)})^T},
\end{align*}
where $\cL^{(2)} \in \R^{n\times n^2}$ can be considered as a mode-$2$ matricization of a tensor $\cL^{n\times n\times n}$. Using \Cref{lemma:trace_property} and the relations~\eqref{eq:matricization_relations}, we obtain
\begin{align*}
 \trace{\cL^{(2)}\left(\cH^{(2)}\right)^T} &= \trace{\cL^{(3)}\left(\cH^{(3)}\right)^T} = \trace{\cC^T\cH^{(3)}(\cB\otimes \cA) \left( \cH^{(3)}\right)^T} \\
 &=\trace{\cC^T\cH^{(2)}(\cB\otimes \cA) \left(\cH^{(2)}\right)^T}\qquad\qquad \qquad\left(\text{using}~\eqref{eq:H2H3_rel} \right)\\
 &= \left(\vecop{\cC}\right)^T\vecop{\cH^{(2)}(\cB\otimes \cA) \left(\cH^{(2)}\right)^T}.
\end{align*}
Furthermore, we also have
\begin{align*}
 \trace{\cL^{(2)}\left(\cH^{(2)}\right)^T} &= \trace{\cL^{(1)}\left(\cH^{(1)}\right)^T} = \trace{\cA^T\cH^{(1)}(\cC\otimes \cB) \left( \cH^{(1)}\right)^T}\\
  &= \left(\vecop{\cA}\right)^T\vecop{\cH^{(1)}(\cC\otimes \cB) \left( \cH^{(1)}\right)^T} ,
 \end{align*}
which completes the proof.
\end{proof}

Next, we prove  the connection between a permutation matrix and the Kronecker product.
\begin{lemma}\label{lemma:change_kron}
 Consider  two matrices $X,Y\in\Rnm$. Define the permutation matrix $T_{(n,m)} \in \{0,1\}^{n^2m^2\times n^2m^2}$ as
\begin{equation}\label{eq:perm_kronvec}
\begin{aligned}
T_{(n,m)} &= I_m\otimes  \bbm I_m \otimes e^n_1,\ldots,I_m \otimes e^n_n \ebm \otimes I_n.
\end{aligned}
\end{equation}
Then,
\begin{equation*}
 \vecop{X\otimes Y} = T_{(n,m)}\left(\vecop{X}\otimes \vecop{Y}\right).
\end{equation*}
\end{lemma}
\begin{proof}
Let us denote the $i$th columns of $X$ and $Y$ by $x_i$ and $y_i$, respectively.   We can then write
 \begin{align}\label{eq:kronXY1}
  \vecop{X\otimes Y} & = \bbm \vecop{x_1\otimes Y}\\ \vdots \\ \vecop{x_m\otimes Y} \ebm.
 \end{align}
Now we concentrate on the $i$th block row of $\vecop{X\otimes Y}$, which, using~\eqref{eq:trace2} and \eqref{eq:kronprodpro}, can be written as
\begin{align}
 \vecop{x_i\otimes Y} &= \vecop{(x_i\otimes I_n)Y} = \left(I_m\otimes x_i\otimes I_n\right)\vecop{Y}\nonumber\\
 &= \left(I_m\otimes \left[ x_i^{(1)}e_1^n + \cdots + x_i^{(n)}e_n^n \right]\otimes I_n\right)\vecop{Y}, \label{eq:kronPro1}
\end{align}
where $x_i^{(j)}$ is the $(j,i)$th entry of the matrix $X$. An alternative way to write~\eqref{eq:kronPro1} is
\begin{align*}
 \vecop{x_i\otimes Y} &=  \left[I_m\otimes e_1^n \otimes I_n, \ldots, I_m\otimes e_n^n\otimes I_n \right] (x_i\otimes I_{nm}) \vecop{Y}\\
  &= \left(\left[I_m\otimes e_1^n , \ldots, I_m\otimes e_n^n \right]\otimes I_n \right) (x_i\otimes \vecop{Y}).
\end{align*}
This yields
\begin{align*}
  \vecop{X\otimes Y} & = \bbm \left(\left[I_m\otimes e_1^n , \ldots, I_m\otimes e_n^n \right]\otimes I_n \right) (x_1\otimes \vecop{Y}) \\ \vdots \\ \left(\left[I_m\otimes e_1^n , \ldots, I_m\otimes e_n^n \right]\otimes I_n \right) (x_m\otimes \vecop{Y}) \ebm \\
  & =  \left(I_m\otimes \left[I_m\otimes e_1^n , \ldots, I_m\otimes e_n^n \right]\otimes I_n \right)\bbm  x_1\otimes \vecop{Y} \\ \vdots \\  x_m\otimes \vecop{Y} \ebm\\
   & =  \left(I_m\otimes \left[I_m\otimes e_1^n , \ldots, I_m\otimes e_n^n \right]\otimes I_n \right) \left(\bbm  x_1 \\ \vdots \\  x_m\ebm\otimes \vecop{Y} \right) \\
   & =  \left(I_m\otimes \left[I_m\otimes e_1^n , \ldots, I_m\otimes e_n^n \right]\otimes I_n \right) \left( \vecop{X}\otimes \vecop{Y} \right),
 \end{align*}
which proves the assertion.
\end{proof}
\Cref{lemma:change_kron} will be utilized  in simplifying the error expressions in the next section.

%% file: h2_norm_QB_sys.tex
\section{\texorpdfstring{\boldmath{$\cH_2$}}{H2}-Norm for QB Systems and Optimality Conditions}\label{sec:optimalitycondition}
In this section, we first define the  $\cH_2$-norm for the QB systems~\eqref{sys:QBsys} and its truncated version.   Then, based on a truncated  $\cH_2$ measure, we derive  first-order necessary conditions for optimal model reduction. These  optimality conditions will naturally lead to a numerical  algorithm to construct quasi-optimal reduced models for QB systems those are independent of any training data.  The proposed model reduction framework extends the optimal $\cH_2$ methodology from  linear \cite{morGugAB08} and bilinear  systems \cite{morBenB12b,flagg2015multipoint}  to  QB nonlinear systems.
\subsection{ \texorpdfstring{\boldmath{$\cH_2$}}{H2}-norm of QB systems}
  In order to define the $\cH_2$-norm for QB systems and its truncated version, we first require the input/output representation for QB systems. In other words, we aim at obtaining the solution of QB systems with the help of \emph{Volterra} series as done for bilinear systems,  e.g. as  in~\cite[Section 3.1]{rugh1981nonlinear}. For this, one can utilize the \emph{variational analysis}~\cite[Section 3.4]{rugh1981nonlinear}. As a first step, we consider the external force $u(t)$ in \eqref{sys:QBsys} that is multiplied by a scaler factor $\alpha$. And since the QB system falls under the class of  linear-analytic systems, we can write the solution $x(t)$ of~\eqref{sys:QBsys} as 
  \begin{equation*}
  x(t) = \sum_{s=1}^\infty \alpha^sx_s(t),
  \end{equation*}
  where $x_s(t) \in \Rn$. Substituting the expression of $x(t)$ in terms of $x_s(t)$ and replacing input $u(t)$ by $\alpha u(t)$ in~\eqref{sys:QBsys}, we get
  \begin{equation}
  \begin{aligned}
  \left(\sum_{s=1}^\infty \alpha^s\dot x_s(t)\right) &= A(\sum_{s=1}^\infty \alpha^sx_s(t)) + H\left(\sum_{s=1}^\infty \alpha^sx_s(t)\right)\otimes \left(\sum_{s=1}^\infty \alpha^sx_s(t)\right) \\
  &\quad + \sum_{k=1}^m\alpha N_k\sum_{s=1}^\infty \alpha^sx_s(t)u_k(t) + \alpha B u(t).
  \end{aligned}
  \end{equation}
By comparing the coefficients of $\alpha^i, i \in \{1,2,\ldots\},$ leads to   an infinite cascade of coupled linear systems as follows:
  \begin{equation}\label{eq:x1x2}
  \begin{aligned}
  \dot x_1(t) &= Ax_1(t) + Bu(t),\\
    \dot x_2(t) &= Ax_2(t) + Hx_1(t)\otimes x_1(t) + \sum_{k=1}^mN_kx_1u_k(t),\\
      \dot x_s(t) &= Ax_s(t) + \sum_{ \substack{ i,j\geq 1 \\  i+j = s}}Hx_i(t)\otimes x_j(t) + \sum_{k=1}^mN_kx_{s-1}(t)u(t),\quad s\geq 3. 
  \end{aligned}
  \end{equation}
  Then, let $\alpha =1$ so that $x(t)  = \sum_{s=1}^\infty x_s(t)$, where $x_s(t)$ solves a coupled linear differential equation~\eqref{eq:x1x2}.   The equation  for $x_1(t)$ corresponds to a linear system, thus allowing us to write the expression for $x_1(t)$ as a  convolution: 
    \begin{subequations}
    \begin{align}
    x_1(t) &= \int_0^te^{At_1}Bu(t-t_1)dt_1. \label{eq:x1_sol}
        \end{align}
    \end{subequations}
    Using the expression for $x_1(t)$, we can obtain an explicit expression for $x_2(t)$:
     \begin{subequations}
     	\begin{align*}   
    x_2(t) &= \int\limits_0^t\int\limits_0^{t-t_3}\int\limits_0^{t-t_3} e^{At_3}He^{At_2}B\otimes e^{At_1}B (u(t-t_2-t_3))\otimes u(t-t_1-t_3)dt_1dt_2dt_3 \label{eq:x2_sol}\\
    &~~+ \sum_{k=1}^m\int\limits_0^t\int\limits_0^{t-t_2}e^{At_2}N_ke^{At_1}Bu(t-t_1-t_2)u_k(t-t_2)dt_1dt_2.
    \end{align*}
    \end{subequations}
Similarly, one can write down  explicit expressions for $x_s(t), s\geq 3$ as well, but the notation and expression become tedious, and we skip them for brevity. Then, we can write the output $y(t)$ of the QB system as $y(t) = \sum_{s=1}^\infty Cx_s(t)$, leading to the input/output mapping as:
  \begin{equation}\label{eq:inputoutputmapping}
  \begin{aligned}
  y(t) &= \int\limits_0^tCe^{At_1}Bu(t-t_1)dt_1 + \\
  &   \int\limits_0^t\int\limits_0^{t-t_3}\int\limits_0^{t-t_3} e^{At_3}H\left(e^{At_2}B\otimes e^{At_1}B\right) (u(t-t_2-t_3))\otimes u(t-t_1-t_3)dt_1dt_2dt_3 + \\
  &  \int\limits_0^t\int\limits_0^{t-t_2}e^{At_2}\bbm N_1,\ldots, N_m\ebm \left(I_m\otimes e^{At_1}B\right)\left(u(t-t_2)\otimes u(t-t_1-t_2)\right)dt_1dt_2 + \cdots.
  \end{aligned}
  \end{equation}
Examining the structure of \eqref{eq:inputoutputmapping} reveals that the \emph{kernels} $f_i(t_1,\ldots,t_i)$  of \eqref{eq:inputoutputmapping} is given by the recurrence formula 
  \begin{equation}\label{eq:kernels_QB}
  f_i(t_1,\ldots,t_i) = Cg_i(t_1,\ldots,t_i),
  \end{equation}
  where
  \begin{equation}\label{eq:P_mapping}
  \begin{aligned}
  g_1(t_1) & = e^{At_1}B,\\
  g_2(t_1,t_2) & = e^{At_2}\bbm N_1,\ldots,N_m\ebm \left(I_m\otimes e^{At_1}B\right),\\
  g_i(t_1,\ldots,t_i) & = e^{At_i}\left[H\left[g_1(t_1)\otimes g_{i-2}(t_2,\ldots,t_{i-1}),\ldots,g_{i-2}(t_1,\ldots,t_{i-2})\otimes g_1(t_{i-1})\right],\right. \\
  &\qquad\quad \left. \bbm N_1,\ldots N_m\ebm \left(I_m\otimes g_{i-1}\right)\right],\quad i\geq 3.
  \end{aligned}
  \end{equation}
 As shown in~\cite{morZhaL02}, the $\cH_2$-norm of a bilinear system can be defined in terms of an infinite series of kernels, corresponding to the input/output mapping. Similarly, we, in the following, define the $\cH_2$-norm of a QB system based on these kernels.
  \begin{definition}\label{def:h2norm}
  	Consider the QB system~\eqref{sys:QBsys} whose kernels of Volterra series are defined in~\eqref{eq:kernels_QB}.  Then, the $\cH_2$-norm of the QB system is given by
  	\begin{equation}\label{eq:h2norm}
  	\|\Sigma\|_{\cH_2} :=\sqrt{\trace{ \sum_{i=1}^\infty  \int_0^\infty\cdots \int_0^\infty f_i(t_1,\ldots,t_i)f_i^T(t_1,\ldots,t_i) dt_1\ldots dt_i} }.
  	\end{equation}  
  \end{definition}
  
  This definition of the $\cH_2$-norm is not suitable for computation. Fortunately, we can find an 
  alternative way, to compute the norm in a numerically efficient  way using matrix equations.  We know from the case of linear and bilinear systems that the $\cH_2$-norms of these systems can  be computed in terms of the certain system Gramians. This is also the case for QB systems.  The algebraic Gramians for QB systems is recently studied in~\cite{morBTQBgoyal}.  So, in the following, we extend such relations between the $\cH_2$-norm, see~\Cref{def:h2norm} and the systems Gramians  to QB systems. 
  \begin{lemma}\label{lemma:Gramian_PQ}
Consider a QB system with a stable matrix $A$, and  let $P$ and $Q$, respectively, be the controllability and observability Gramians of the system, which are the unique  positive semidefinite solutions of the following quadratic-type Lyapunov equations:
  	\begin{align}
AP + PA^T + H(P\otimes P)H^T + \sum_{k=1}^mN_kPN^T_k + BB^T &= 0, ~~\text{and}\label{eq:True_Cont}\\
  	  	A^TQ + QA + \cH^{(2)}(P\otimes Q)\left(\cH^{(2)}\right)^T + \sum_{k=1}^mN^T_kQN_k + C^TC &= 0. \label{eq:True_Obs}
  	  	\end{align}
  	  	Assuming the $\cH_2$-norm of a QB system exists, i.e., the infinite series in~\eqref{eq:h2norm} converges, then the $\cH_2$-norm of a QB system can be computed as 
  	  	\begin{equation}\label{eq:PQ_trace_h2}
  	  	\|\Sigma\|_{\cH_2}:= \sqrt{\trace{CPC^T}} = \sqrt{\trace{B^TQB}}.
  	  	\end{equation}
  \end{lemma}
  \begin{proof}
  	We begin with the definition of the $\cH_2$-norm of a QB system, that is,
  	  	\begin{align*}
  	  	\|\Sigma_{QB}\|_{\cH_2} &= \sqrt{\trace{\sum_{i=1}^\infty  \int_0^\infty\cdots \int_0^\infty f_i(t_1,\ldots,t_i)f_i^T(t_1,\ldots,t_i) dt_1\ldots dt_i}}\\
  	  	&= \sqrt{ \trace{C \left(\sum_{i=1}^\infty\int_0^\infty\cdots \int_0^\infty g_i(t_1,\ldots,t_i)g_i^T(t_1,\ldots,t_i) dt_1\ldots dt_i \right)C^T} },
  	  	\end{align*}  
where $f_i(t_1,\ldots,t_i)$ and $g_i(t_1,\ldots,t_i)$ are defined in \eqref{eq:kernels_QB} and\eqref{eq:P_mapping}, respectively. It is shown in~\cite{morBTQBgoyal} that 
\begin{equation}
\left(\sum_{i=1}^\infty\int_0^\infty\cdots \int_0^\infty g_i(t_1,\ldots,t_i)g_i^T(t_1,\ldots,t_i) dt_1\ldots dt_i \right) = P,
\end{equation}
 where $P$ solves~\eqref{eq:True_Cont} if  the  sequence with infinite terms converges. Thus,
\begin{equation*}
	\|\Sigma\|_{\cH_2}  =  \sqrt{\trace{CPC^T}}.
\end{equation*}
Next, we prove that $\trace{CPC^T} = \trace{B^TQB}$, where $Q$ solves~\eqref{eq:True_Obs}. Making use of Kronecker properties~\eqref{eq:tracepro}, we can write $\trace{CPC^T}$ as
\begin{align*}
\trace{CPC^T} & = \cI_p^T (C\otimes C)\vecop{P}.
\end{align*}
Taking $\vecop{\cdot}$ both sides of \eqref{eq:True_Cont}, we obtain
\begin{equation}
\left(A\otimes I_n + I_n\otimes A + \sum_{k=1}^mN_k\otimes N_k \right)\vecop{P} + (H\otimes H) \vecop{P\otimes P} + (B\otimes B)\cI_m = 0.
\end{equation}
Using~\Cref{lemma:change_kron} in the above equation and performing some simple manipulations yield an expression for $\vecop{P}$ as
\begin{equation*}
\vecop{P}  =  \cG^{-1}(B\otimes B)\cI_m  =: P_v,
\end{equation*}
where 
\begin{equation*}
\cG = -\left(A\otimes I_n + I_n\otimes A + \sum_{k=1}^mN_k\otimes N_k + (H\otimes H) T_{(n,n)}(I_{n^2}\otimes \vecop{P})\right).
\end{equation*}
Thus,  
\begin{equation}\label{eq:H2_P_Q_rel}
\trace{CPC^T} = \cI_p^T(C\otimes C)\cG^{-1}(B\otimes B)\cI_m = \cI_m^T(B^T\otimes B^T)\cG^{-T}(C^T\otimes C^T)\cI_p.
\end{equation}
Now, let $Q_v =   \cG^{-T}(C^T\otimes C^T)\cI_p^T $. As a result, we obtain
\begin{equation*}
\begin{aligned}
(C^T\otimes C^T)\cI_p^T & = \vecop{C^TC} = \cG^{T}Q_v\\
&= -\left(A^T\otimes I_n + I_n\otimes A^T + \sum_{k=1}^mN^T_k\otimes N^T_k \right)Q_v \\
&\qquad\qquad\qquad\qquad\qquad+  \left(\left(H\otimes H\right) T_{(n,n)}\left(I_{n^2}\otimes P_v\right)\right)^TQ_v.
\end{aligned}
\end{equation*}
Next, we consider a matrix $\tQ$ such that $\vecop{\tQ} = Q_v$, which further simplifies the above equation as
\begin{align}\label{eq:deriveQ_dual1}
\vecop{C^TC} &= -\vecop{A^T\tQ + \tQ A + \sum_{k=1}^mN_k^T\tQ N_k} 
-  \left((H\otimes H) T_{(n,n)}(I_{n^2}\otimes P_v)\right)^TQ_v. 
\end{align}
Now, we focus on the transpose of the second part of \eqref{eq:deriveQ_dual1}, that is,
\begin{align*}
&Q_v^T(H\otimes H) T_{(n,n)}\left(I_{n^2}\otimes \vecop{P}\right) \\
&  \qquad = Q_v^T(H\otimes H) T_{(n,n)}\bbm e_1^{n^2} \otimes\vecop{P},\ldots, e_{n^2}^{n^2}\otimes \vecop{P}  \ebm \\
&  \qquad = Q_v^T(H\otimes H) \bbm \vecop{\Psi_1 \otimes P},\ldots, \vecop{\Psi_{n^2}\otimes P}  \ebm  =: \Xi, ~~~~(\text{using~\cref{lemma:change_kron}})
\end{align*}
where $\Psi_i \in \Rnn$ is such that $e_i^{n^2} = \vecop{\Psi_i}$.  Using~\eqref{eq:tracepro} and \Cref{lemma:RelH1H1ABC}, we further analyze the above equation:
\begin{align*}
\Xi   &  = \vecop{\tQ}^T\bbm \vecop{H\left(\Psi_1 \otimes P\right)H^T},\ldots, \vecop{H\left(\Psi_{n^2}\otimes P\right)H^T}  \ebm \\
&  = \vecop{\tQ}^T\bbm \vecop{H\left(P \otimes \Psi_1\right)H^T},\ldots, \vecop{H\left(P\otimes \Psi_{n^2}\right)H^T}  \ebm \\
&  = \Big[ \vecop{\Psi_1}^T\vecop{\cH^{(2)}\left(P \otimes \tQ\right)(\cH^{(2)})^T},\ldots, \\ &
\hspace{6cm} \vecop{\Psi_{n^2}}^T\vecop{\cH^{(2)}\left(P\otimes \tQ\right)(\cH^{(2)})^T}  \Big]\\
&  = \bbm \left(e^{n^2}_1\right)^T\vecop{\cH^{(2)}\left(P \otimes \tQ\right)(\cH^{(2)})^T},\ldots, \left(e^{n^2}_{n^2}\right)^T\vecop{\cH^{(2)}\left(P\otimes \tQ\right)(\cH^{(2)})^T}  \ebm \\
& = \left(\vecop{\cH^{(2)}\left(P \otimes \tQ\right)(\cH^{(2)})^T}\right)^T. 
\end{align*}
Substituting this relation in~\eqref{eq:deriveQ_dual1} yields
\begin{align*}
\vecop{C^TC} &= -\vecop{A^T\tQ + \tQ A + \sum_{k=1}^mN_k^T\tQ N_k + \cH^{(2)}\left(P \otimes \tQ\right)(\cH^{(2)})^T},
\end{align*}
which shows that $\tQ$ solves \eqref{eq:True_Obs} as well. Since it is assumed that Eq.\ \cref{eq:True_Obs} has a unique solution, thus $\tQ = Q$.  Replacing  $\cG^{-T}(C^T\otimes C^T)\cI^T_p$ by $\vecop{Q}$ in~\eqref{eq:H2_P_Q_rel} and using \eqref{eq:tracepro} result in 
\begin{align*}
\trace{CPC^T} &= \cI_m^T(B^T\otimes B^T)\vecop{Q} = \trace{B^TQB}.
\end{align*}
This concludes the proof.
\end{proof}

It can be  seen that if  $H$ is zero,  the expression \eqref{eq:PQ_trace_h2} boils down to the $\cH_2$-norm of bilinear systems, and if $N_k$ are also set to zero then it provides us the $\cH_2$-norm of stable  linear systems as one would expect.
\begin{remark}\label{rescalling_discussions}
In \Cref{lemma:Gramian_PQ}, we have assumed that the solutions of \eqref{eq:True_Cont} and \eqref{eq:True_Obs} exist, and are unique and positive semidefinite. Equivalently, the series appearing in the definition of the $\cH_2$-norm is finite (see \cref{def:h2norm}); hence, the $\cH_2$-norm exits. Naturally, the stability of the matrix $A$ is necessary for the existence of Gramians, and a detailed study of  such solutions of \eqref{eq:True_Cont} and \eqref{eq:True_Obs} has been carried out in \cite{morBTQBgoyal}.  However, as for bilinear systems, these Gramians may not have these desired properties when $\|N_k\|$ and $\|H\|$ are  large. 

Nonetheless, a solution of these problems can be obtained via rescaling of the system as done in the bilinear case~\cite{morcondon2005}. For this, we need to rescale the input variable $u(t)$ as well as the state vector $x(t)$. More precisely, we can replace $x(t)$ and $u(t)$ by $\gamma x(t) =: \tx(t)$ and $\gamma u(t)=: \tu(t)$ in~\eqref{sys:QBsys}. This leads to 
\begin{equation}\label{sys:QBsys_scale}
\begin{aligned}
\gamma\dot \tx(t) &= \gamma A\tx(t) + \gamma^2 H\left(\tx(t)\otimes \tx(t)\right) + \gamma^2 \sum_{k=1}^m N_k\tx(t)\tu_k(t) + \gamma B\tu(t),\\
y(t)& = \gamma C\tx(t), \quad \tx(0) = 0.
\end{aligned}
\end{equation}
For $\gamma \neq 0$, we get a scaled system as follows:
\begin{equation}\label{sys:QBsys_scale1}
\begin{aligned}
\dot \tx(t) &=  A\tx(t) + (\gamma H)\left(\tx(t)\otimes \tx(t)\right) +  \sum_{k=1}^m (\gamma N_k)\tx(t)\tu_k(t) +  B\tu(t),\\
\ty(t)& = C\tx(t), \quad \tx(0) = 0,
\end{aligned}
\end{equation}
where $\ty(t) =  y(t)/\gamma$.  Now, if you compare the systems \eqref{sys:QBsys} and \eqref{sys:QBsys_scale1}, the input/output mapping is the same, and the outputs of the systems  differ by the factor $\gamma$. 
Since interpolation-like model reduction methods aim at capturing the input/output mapping, we can also use  \eqref{sys:QBsys_scale1} for constructing  the subspaces $V$ and $W$. 
\end{remark}

Our primary aim is to determine a reduced-order system which minimizes  the $\cH_2$-norm of the  error system. From the derived $\cH_2$-norm expression for the QB system, it is clear that the true $\cH_2$-norm  has a quite complicated structure as defined in~\eqref{eq:h2norm} and does not lend itself well to 
deriving necessary conditions for optimality. Therefore, to simplify the problem, we focus only on the first three leading terms of the series \eqref{eq:inputoutputmapping}.  The main reason for considering the first three terms is that it is the minimum number of terms containing contributions from all the system matrices $(A,H,N_k,B,C)$; in other words, linear, bilinear and quadratic terms are already contained in these first three kernels.  Our approach is also inspired by \cite{flagg2015multipoint} where a truncated 
$\cH_2$ norm is defined for bilinear systems and used to construct high-fidelity reduced-order models minimizing corresponding error measures. Therefore, based on these three leading kernels,  we define a truncated $\cH_2$-norm for QB systems, denoted  by $\|\Sigma\|_{\cH_2^{(\cT)}}$. Precisely, in terms of kernels, a truncated norm can be defined as follows:
\begin{equation}\label{eq:H2_QB}
 \|\Sigma\|_{\cH^{(\cT)}_2} := \sqrt{\trace{ \sum_{i=1}^3\int_0^\infty\cdots \int_0^\infty \tf_i{(t_1,\ldots,t_i)} \left(\tf_i{(t_1,\ldots,t_i)} \right)^T dt_1\cdots dt_i }},
\end{equation}
where 
 \begin{equation}\label{eq:def_f}
 \tf_i (t_1,\ldots, t_i) = C \tg_i (t_1,\ldots, t_i) ,\qquad i \in \{1,2,3\},
 \end{equation}
 and
 \begin{equation*}
 \begin{aligned}
 \tg_1(t_1)&= e^{At_1}B ,\quad \quad \tg_2(t_1,t_2)= e^{At_2}\bbm N_1,\ldots,N_m\ebm \left(I_m\otimes e^{At_1}B\right),\\
 \tg_3 (t_1,t_2, t_3)&= e^{At_3}H (e^{At_2}B\otimes  e^{At_1}B\big).
 \end{aligned}
 \end{equation*}

Analogous to the $\cH_2$-norm of the QB system, a truncated $\cH_2$-norm of  QB systems can be determined by  truncated controllability  and observability  Gramians associated with the QB system, denoted by  $P_\cT $ and  $Q_\cT $, respectively~\cite{morBTQBgoyal}. These truncated Gramians  are the unique solutions of the  following  Lyapunov equations, provided $A$ is stable:
\begin{subequations}\label{eq:Gram_TQB}
 \begin{align}
  A P_\cT + P_\cT A^T + \sum_{k=1}^m N_kP_lN_k^T + H(P_l\otimes P_l) H^T + BB^T &= 0,\label{eq:CGram}\\
  A^TQ_\cT  + Q_\cT A + \sum_{k=1}^m N_k^TQ_lN_k + \cH^{(2)}(P_l\otimes Q_l) \left(\cH^{(2)}\right)^T + C^TC &= 0\label{eq:OGram},
 \end{align}
\end{subequations}
where $\cH^{(2)}$ is the mode-2 matricization of the QB Hessian, and $P_l$ and $Q_l$ are the unique solutions of  the following Lyapunov equations:
\begin{subequations}\label{eq:Linear_Lyap}
 \begin{align}
  AP_l + P_lA^T +BB^T &= 0, \label{eq:LinearC}\\
  A^TQ_l + Q_lA   + C^TC &= 0.\label{eq:LinearO}
 \end{align}
\end{subequations}
Note that the stability of the matrix  $A$ is a necessary and sufficient condition for the existence of truncated Gramians~\eqref{eq:Gram_TQB}.   In what follows, we show the connection between the truncated $\cH_2$-norm and  the defined truncated Gramians for  QB systems.
\begin{lemma}\label{lemma:H2_norm}
Let $\Sigma$ be the QB system~\eqref{sys:QBsys}  with a stable $A$ matrix. Then, the  truncated $\cH_2$-norm based on the first three terms of the Volterra series is given by
 \begin{equation*}
  \|\Sigma\|_{\cH^{(\cT)}_2} = \sqrt{\trace{CP_\cT C^T}} = \sqrt{\trace{B^TQ_\cT B}},
 \end{equation*}
where $P_\cT $ and $Q_\cT $ are  truncated controllability and observability Gramians of the system, satisfying \eqref{eq:Gram_TQB}.
\end{lemma}
\begin{proof}
First, we note that \eqref{eq:Gram_TQB} and \eqref{eq:Linear_Lyap} are standard Lyapunov equations. As $A$ is assumed to be stable, these equations have unique solutions~\cite{bartels1972solution}.   Next, let $\cR_i$ be
 \begin{equation*}
 \cR_i = \int_0^\infty\cdots \int_0^\infty \tf_i(t_1,\ldots,t_i) \left(\tf_i(t_1,\ldots,t_i)\right)^T dt_1\cdots dt_i,
 \end{equation*}
 where  $\tf_i(t_1,\ldots, t_i)$ are defined in~\eqref{eq:def_f}.  Thus, $\|\Sigma\|^2_{\cH^{(\cT)}_2} = \trace{C\left(\sum_{i=1}^3 \cR_i\right) C^T}$. It is shown in~\cite{morBTQBgoyal} that  $\sum_{i=1}^3 \cR_i = P_\cT $ solves  the Lyapunov equation \eqref{eq:CGram}. Hence,
 $$\|\Sigma\|^2_{\cH^{(\cT)}_2} = \trace{CP_\cT C^T}.$$
 Next, we show that $\trace{CP_\cT C^T} = \trace{B^TQ_\cT B}$. For this, we use the trace property~\eqref{eq:trace2} to obtain:
 \begin{align*}
  \trace{CP_\cT C^T} &= \left(\cI_p\right)^T (C\otimes C)\vecop{P_\cT }~\text{and}~
    \trace{B^TQ_\cT B} = \left(\vecop{Q_\cT }\right)^T  (B\otimes B) \cI_m.
 \end{align*}
 Applying $\vecop{\cdot}$ to both sides of~\eqref{eq:Gram_TQB} results in
 \begin{equation*}
 \begin{aligned}
  \vecop{P_\cT } &= \cL^{-1}\left( \left(B\otimes B\right)\cI_m + \sum_{k=1}^m\left(N_k\otimes N_k\right) \cL^{-1}\left(B\otimes B\right)\cI_m  \right. \\
  & \left.\qquad\qquad\qquad+ \vecop{H\left(P_l\otimes P_l\right)H^T } \vphantom{\sum_{k=1}^m} \right), ~\text{and}\\
  \vecop{Q_\cT } &= \cL^{-T}\left( (C\otimes C)^T\cI_p + \sum_{k=1}^m (N_k\otimes N_k)^T \cL^{-T}(C\otimes C)^T\cI_p \right.\\
  &\left.\qquad\qquad\qquad+ \vecop{\cH^{(2)}(P_l\otimes Q_l)\left(\cH^{(2)}\right)^T } \vphantom{\sum_{k=1}^m}\right),
  \end{aligned}
 \end{equation*}
where $\cL = -(A\otimes I_n + I_n\otimes A)$, and $P_l$ and $Q_l$ solve \eqref{eq:Linear_Lyap}. Thus,
\begin{equation}
\begin{aligned}
  \trace{B^TQ_\cT B}   &=  \left( \left(\cI_p\right)^T(C\otimes C) +\left(\cI_p\right)^T(C\otimes C) \cL^{-1}\sum_{k=1}^m(N_k\otimes N_k) \right.\\
  & \left.\qquad+ \left(\vecop{\cH^{(2)}(P_l\otimes Q_l)\left(\cH^{(2)}\right)^T }\right)^T \right) \cL^{-1}(B\otimes B)\cI_m.\label{eq:trace_change}
\end{aligned}
\end{equation}
Since $P_l$ and $Q_l$ are the unique solutions of~\eqref{eq:LinearC} and~\eqref{eq:LinearO}, this gives $\vecop{P_l} = \cL^{-1}(B\otimes B)\cI_m$ and $\vecop{Q_l} = \cL^{-T}(C\otimes C)^T\cI_p$. This implies that
\begin{align*}
&\left(\vecop{\cH^{(2)}(P_l\otimes Q_l)\left(\cH^{(2)}\right)^T }\right)^T  \vecop{P_l} \\
&\qquad\qquad=   \vecop{P_l}^T\vecop{\cH^{(2)}(P_l\otimes Q_l)\left(\cH^{(2)}\right)^T } \\
&\qquad\qquad= \vecop{Q_l}^T\vecop{ H (P_l\otimes P_l)H^T } \hspace{3cm} \left(\text{using \Cref{lemma:trace_property}}\right)\\
&\qquad\qquad=  \left(\cI_p\right)^T (C\otimes C)\cL^{-1}\vecop{ H (P_l\otimes P_l)H ^T }.
\end{align*}
Substituting the above relation in~\eqref{eq:trace_change} yields
\begin{align*}
 \trace{B^TQ_\cT B} &=  \left(\cI_p\right)^T(C\otimes C) \cL^{-1}\Big((B\otimes B)\cI_m + \sum_{k=1}^m(N_k\otimes N_k)\cL^{-1}(B\otimes B)\cI_m\\
  & \qquad+ \vecop{H(P_l\otimes P_l)H^T } \Big)\\
  &=\left(\cI_p\right)^T(C\otimes C)\vecop{P_\cT }=\trace{CP_\cT C^T}.
\end{align*}
This concludes the proof.
\end{proof}

\begin{remark}
One can consider  the first $\cM$ number of terms of the corresponding Volterra series and, based on these $\cM$ kernels, another truncated $\cH_2$-norm can be defined. However, this significantly increases the complexity of the problem. In this paper, we stick to the truncated $\cH_2$-norm for the QB system that depends on the first three terms of the input/output mapping. 	And we intend to construct reduced-order systems~\eqref{sys:QBsysRed} such that this truncated $\cH_2$-norm of the error system is minimized. Another motivation for the derived truncated $\cH_2$-norm for QB systems is that for bilinear systems, the authors in~\cite{flagg2015multipoint}  showed that the $\cH_2$-optimal model reduction based on a truncated $\cH_2$-norm (with only two terms of the Volterra series of a bilinear system) also mimics the accuracy of the true $\cH_2$-optimal approximation very closely.
\end{remark}

%% file: optimality_conditions.tex
\subsection{Optimality conditions based on a truncated \texorpdfstring{\boldmath{$\cH_2$}}{H2}-norm}
We now  derive  necessary conditions for optimal model reduction based on the truncated $\cH_2$-norm of the error system. First, we define the QB error system.  For the full QB model $\Sigma$ in \eqref{sys:QBsys} and the reduced QB model $\hat \Sigma$ in \eqref{sys:QBsysRed}, we  can write the error system as
\begin{equation}\label{eq:ErrorSys}
 \begin{aligned}
  \bbm \dot x(t) \\ \dot \hx(t) \ebm &{=} \underbrace{\bbm A & \mathbf{0} \\ \mathbf{0} & \hA \ebm}_{A^e} \underbrace{\bbm  x(t) \\  \hx(t) \ebm}_{x^e(t)}  {+}\bbm H\left(x(t)\otimes x(t)\right) \\ \hH\left(\hx(t)\otimes \hx(t)\right)\ebm {+} \sum_{k=1}^m\underbrace{\bbm N_k & \mathbf{0} \\ \mathbf{0} & \hN_k \ebm }_{N^e_k} \bbm  x(t) \\  \hx(t) \ebm u_k(t)  {+} \underbrace{\bbm B \\ \hB\ebm}_{B^e} u(t),\\
  y^{e}(t) &=y(t) - \hy(t)=  \underbrace{\bbm C & -\hC \ebm}_{C^e} \bbm x^T(t) & \hx^T(t) \ebm^T,\quad x^e(0)=0.
 \end{aligned}
\end{equation}
It can be seen that the error system~\eqref{eq:ErrorSys} is not in the conventional  QB form due to the absence of the quadratic term  $x^e(t)\otimes x^e(t)$. However, we can rewrite the system~\eqref{eq:ErrorSys} into a regular QB form by using an appropriate Hessian of the error system~\eqref{eq:ErrorSys} as follows:
\begin{equation}\label{eq:ErrorSys1}
\Sigma^e :=\begin{cases}
 \begin{aligned}
  \dot x^{e}(t) &= A^{e} x^{e}(t) +  H^e\left(x^e(t)\otimes x^e(t)\right) +\sum_{k=1}^m N_k^{e}x^{e}(t)u_k(t)  + B^e u(t),\\
  y^e(t) &= C^e x^e(t) ,\quad x^e(0)=0,
 \end{aligned}
\end{cases}
\end{equation}
where $H^e = \bbm H\cF \\ \hH\hat{\cF} \ebm $ with $\cF = \bbm I_n & \mathbf{0} \ebm \otimes \bbm I_n & \mathbf{0} \ebm $ and $\hat{\cF} = \bbm \mathbf{0} & I_r\ebm \otimes \bbm \mathbf{0} & I_r\ebm  $.
Next, we write the truncated $\cH_2$-norm, as defined in~\cref{lemma:H2_norm},  for the error system~\eqref{eq:ErrorSys1}. For the existence of this norm for the system~\eqref{eq:ErrorSys1}, it is necessary to assume  that the matrix $A^e$ is stable, i.e., the matrices $A$ and $\hA$ are stable.  Further, we assume that the matrix $\hA$ is diagonalizable. Then, by performing basic algebraic manipulations and making use of \Cref{lemma:change_kron}, we obtain the  expression for the error functional $\cE$ based on the truncated $\cH_2$-norm of the error system~\eqref{eq:ErrorSys1} as shown next.
\begin{corollary}\label{coro:error}
 Let $\Sigma$  be the original system, having a stable matrix $A$, and  let $\hat\Sigma$  be the reduced-order original system, having a stable and diagonalizable matrix $\hA$. Then,
  \begin{align}
  \cE^2 := \|\Sigma^e\|^2_{\cH^{(\cT)}_2} &= (\cI_p)^T ( C^e \otimes C^e) (-A^e\otimes I_{n+r} - I_{n+r}\otimes A^e  )^{-1}\Big( (B^e\otimes B^e)\cI_m \nonumber \\
  \label{eq:errexp1}
  &\quad + \sum_{k=1}^m (N^e_k\otimes N^e_k) \vecop{P^e_l} + \vecop{H^e(P^e_l\otimes P^e_l)\left(H^e\right)^T } \Big),
 \end{align}
 where $P_l^e$ solves
\begin{equation*}
 A^eP_l^e + P_l^e ( A^e)^T + B^e(B^e)^T = 0.
\end{equation*}
Furthermore, let $\hA = R\Lambda R^{-1}$ be the spectral decomposition of $\hA$, and define $\tB = R^{-1}\hB$, $\tC = \hC R$, $\tN_k = R^{-1}\hN_k R$ and $\tH = R^{-1}\hH (R\otimes R)$. Then, the error  can be rewritten as
\begin{equation}\label{eq:errexp2}
 \begin{aligned}
  \cE^2  &= (\cI_p)^T \left( \tC^e   \otimes \tC^e \right)  \left( -\tA^e \otimes I_{n+r}   - I_{n+r}\otimes \tA^e \right)^{-1}\left( \left(\tB^e \otimes \tB^e \right)\cI_m \right.\\
&\left.  \quad  + \sum_{k=1}^m \left(\tN^e_k \otimes \tN^e_k \right)\cP_l+ \left(\tH^e \otimes \tH^e \right)T_{(n+r,n+r)} ( \cP_l\otimes  \cP_l) \right),
 \end{aligned}
 \end{equation}
 where
  \begin{align}
  \tA^e &= \bbm A & \mathbf{0} \\ \mathbf{0} & \Lambda \ebm , \tN^e_k =  \bbm N_k & \mathbf{0} \\ \mathbf{0} &\tN_k \ebm,\tH^e = \bbm H\cF \\ \tH\hat{\cF} \ebm, \tB^e = \bbm B \\ \tB \ebm,  \tC^e = \bbm C^T \\ -\tC^T \ebm^T,\nonumber\\
  \cP_l &= \bbm \cP_l^{(1)} \\ \cP_l^{(2)} \ebm = \bbm  \left( - A  \otimes I_{n+r}   - I_n\otimes\tA^e \right)^{-1}  \left( B  \otimes \tB^e \right)\cI_m  \\ \left( -\Lambda  \otimes I_{n+r}   - I_{r}\otimes\tA^e \right)^{-1}  \left(\tB  \otimes \tB^e \right)\cI_m    \ebm,  \text{and}\label{eq:p1_exp}\\
 T_{(n+r,n+r)}&= I_{n+r} \otimes \bbm I_{n+r}\otimes e_1^{n+r},\ldots ,I_{n+r}\otimes e_{n+r}^{n+r} \ebm \otimes I_{n+r}.\nonumber
  \end{align}
\end{corollary}

The above spectral decomposition for $\hA$ is computationally  useful in simplifying the expressions as we will see later. It reduces the optimization variables by $r(r{-}1)$ since  $\Lambda$ becomes a diagonal matrix without changing the value of the cost function (this is a state-space transformation on the reduced model, which does not change the input-output map). Even though it limits the reduced-order systems to those only having diagonalizable $\hA$, as observed in the linear \cite{morGugAB08} and bilinear cases \cite{morBenB12b,flagg2015multipoint}. It is extremely rare in practice that the optimal $\mathcal{H}_2$ models will have a non-diagonalizable $\hA$; therefore, this diagonalizability assumption does not incur any restriction from a practical perspective.

Our aim is to choose the optimization variables $\Lambda$, $\tB$, $\tC$, $\tN_k$ and $\tH$ such that the $\|\Sigma-\hat\Sigma\|_{\cH^{(\cT)}_2}$, i.e., equivalently the error expression~\eqref{eq:errexp2}, is minimized. Before we proceed further, we introduce a particular  permutation matrix
\begin{equation}\label{eq:per_M}
 M_{pqr} = \bbm I_{p}\otimes \bbm I_q \\ \mathbf{0} \ebm & I_p\otimes \bbm \mathbf{0} \\ I_r \ebm \ebm,
 \end{equation}
 which will prove helpful in simplifying the expressions related to the Kronecker product of block matrices. For example, consider matrices $\cA\in \Rpp$, $\cB \in \Rqq$ and $\cC \in \Rrr$. Then, the following relation holds:
\begin{equation*}
 M_{pqr}^T \begin{pmatrix} \cA\otimes \bbm \cB & \mathbf{0} \\ \mathbf{0} & \cC \ebm \end{pmatrix} M_{pqr} = \bbm  \cA\otimes \cB & \mathbf{0}\\ \mathbf{0} &  \cA\otimes \cC \ebm.
\end{equation*}
Similar block structures can be found in the error expression $\cE$ in \Cref{coro:error}, which can be simplified analogously. Moreover, due to the presence of  many Kronecker products, it would be convenient to  derive  necessary conditions for optimality in the Kronecker formulation itself. Furthermore,  these conditions can  be easily translated into a theoretically equivalent framework of Sylvester equations, which  are more concise, are more easily interpretable, and, more importantly, automatically lead to an effective numerical algorithm for model reduction. To this end, let $V_i\in \Rnr$ and $W_i\in \Rnr $, $i \in \{1,2\}$,  be the solutions of the following standard Sylvester equations:
\begin{subequations}\label{eq:VW_proj}
 \begin{align}
  V_1(-\Lambda) -AV_1 &= B\tB^T,  \label{eq:solveV1} \\
       W_1(-\Lambda) -A^TW_1 &= C^T\tC,  \label{eq:solveW1} \\
   V_2(-\Lambda) -AV_2 &= \sum_{k=1}^mN_kV_1\tN_k^T + H(V_1\otimes V_1)\tH^T , \label{eq:solveV2}\\
   W_2(-\Lambda) -A^TW_2 &= \sum_{k=1}^mN_k^TW_1\tN_k +  2\cdot \cH^{(2)}(V_1\otimes W_1)(\tilde{\cH}^{(2)})^T,\label{eq:solveW2}
 \end{align}
\end{subequations}
where $\Lambda, \tN_k,\tB$ and $\tC$ are as defined in \Cref{coro:error}. Furthermore, we define trial and test bases $V\in \Rnr$ and $W\in \Rnr$ as
\begin{equation}\label{eq:def_VW}
 V = V_1 + V_2 \quad\text{and}\quad W = W_1+W_2.
\end{equation}
We also define $\hV\in \Rrr$ and $\hW\in \Rrr$ (that will appear in optimality conditions as we see later)   as follows:
\begin{equation}\label{eq:def_VWhat}
\hV =\hV_1 + \hV_2 \quad\text{and}\quad \hW = \hW_1+\hW_2,
\end{equation}
where  $\hV_i\in \Rrr$, $\hW_i\in \Rrr$, $i \in \{1,2\}$  are the solutions of the set of equations in \eqref{eq:VW_proj} but with the original system's state-space matrices being replaced with the reduced-order system ones,  for example,  $A$ with $\hA$ and $B$ with $\hB$, etc.  Next, we present first order necessary conditions for optimality, which aim at  minimizing  the error expression~\eqref{eq:errexp2}. The following theorem extends the truncated  $\cH_2$ optimal conditions from the bilinear case to the much more general quadratic-bilinear nonlinearities.
\begin{theorem}\label{thm:optimality_conditions}
 Let $\Sigma$ and $\hat \Sigma$  be the original and reduced-order systems as defined in~\eqref{sys:QBsys} and~\eqref{sys:QBsysRed}, respectively.  Let $\Lambda = R^{-1}\hA R$ be the spectral decomposition of $\hA$, and define  $\tH = R^{-1}\hH(R\otimes R)$, $\tN_k = R^{-1}\hN_kR$, $\tC = \hC R$, $\tB = R^{-1}\hB$.   If $\hat\Sigma$ is a reduced-order system that minimize the truncated $\cH_2$-norm of  the error system~\eqref{eq:ErrorSys1} subject to $\hA$ being diagonalizable, then $\hat\Sigma$ satisfies the following conditions:
 \begin{subequations}\label{eq:optimalityconditions}
\begin{align}
 & \trace{CVe^r_i\left(e^p_j\right)^T } = \trace{\hC \hV e^r_i\left(e^p_j\right)^T }, \qquad\quad\qquad i \in \{1,\ldots,r\},~~ j \in \{1,\ldots,p\},\label{eq:cond_C}\\
  &\trace{B^TWe^r_i\left(e^m_j\right)^T } = \trace{\hB^T\hW e^r_i\left(e^m_j\right)^T },\qquad~~ i \in \{1,\ldots,r\},~~ j \in \{1,\ldots,m\},\label{eq:cond_B}\displaybreak\\
  &(W_1(:,i))^T N_k V_1(:,j)  = (\hW_1(:,i))^T \hN_k \hV_1(:,j),  ~~~~  i,j \in \{1,\ldots,r\},~~ k \in \{1,\ldots,m\}, \label{eq:cond_N}  \\
  &  (W_1(:,i))^T  H (V_1(:,j) \otimes V_1(:,l))  = (\hW_1(:,i))^T  \hH (\hV_1(:,j) \otimes \hV_1(:,l)), \label{eq:cond_H} \\ &\hspace{7.5cm} i,j,l \in \{1,\ldots,r\},\nonumber\\
  &  (W_1(:,i))^T  V(:,i) +  \left(W_2(:,i)\right)^TV_1(:,i)      = (\hW_1(:,i))^T  \hV(:,i) +  \left(\hW_2(:,i)\right)^T \hV_1(:,i),\label{eq:cond_lambda}\\& \hspace{7.5cm} i \in \{1,\dots,r\}. \nonumber
 \end{align}
\end{subequations}
\end{theorem}
\begin{proof}
 The proof is given in \Cref{appen:proof}.
\end{proof}

\subsection{Truncated quadratic-bilinear iterative rational Krylov algorithm}
The remaining challenge is now to develop a numerically efficient model reduction algorithm to construct a reduced QB system satisfying  first-order optimality conditions in~\cref{thm:optimality_conditions}. However, as in the linear \cite{morGugAB08} and bilinear~\cite{morBenB12b,flagg2015multipoint} cases,  since the optimality conditions involve the matrices $V,W,\hV,\hW$, which depend on the reduced-order system matrices we are trying to construct, it is not a straightforward task to determine a reduced-order system directly that satisfies all the necessary conditions for optimality, i.e.,~\eqref{eq:cond_C}--\eqref{eq:cond_lambda}.  We propose \Cref{algo:TQB-IRKA}, which  upon convergence  leads to reduced-order systems that \emph{approximately} satisfy  first-order necessary conditions for optimality of \Cref{thm:optimality_conditions}. Throughout the paper, we denote the algorithm by {truncated QB-IRKA}, or {TQB-IRKA}.

\begin{algorithm}[!htb]
 \caption{ TQB-IRKA for quadratic-bilinear systems.}
 \label{algo:TQB-IRKA}
 \begin{algorithmic}[1]
    \Statex {\bf Input:} The system matrices: $ A, H,N_1,\ldots,N_m, B,C$.
    \Statex {\bf Output:} The reduced matrices: $\hA,\hH,\hN_1,\ldots,\hN_m,\hB,\hC$.
    \State Symmetrize the Hessian $H$ and determine its mode-2 matricization $\cH^{(2)}$.
    \State Make an initial (random) guess of the reduced matrices $\hA,\hH,\hN_1,\ldots,\hN_m,\hB,\hC$.
        \While {not converged}
        \State Perform the spectral decomposition of $\hA$ and define:
    \Statex\quad\qquad $\Lambda = R^{-1}\hA R,~\tN_k =R^{-1}\hN_{k}R,~\tH = R^{-1}\hH \left(R\otimes R\right),~\tB = R^{-1}\hB, ~\tC = \hC R. $
    \State Compute mode-2 matricization $\tilde{\cH}^{(2)}$.
        \State Solve for $V_1$ and $V_2$: \label{step:v}
	        \Statex \quad\qquad$ -V_1\Lambda  -  AV_1 = B\tB^T$,
        \Statex \quad\qquad$ -V_2\Lambda  - AV_2 = H(V_1\otimes V_1)\tH^T + \sum\limits_{k=1}^m N_kV_1\tN_k^T $.
        \State  Solve for $W_1$ and $W_2$:\label{step:w}
       \Statex \quad\qquad$ -W_1\Lambda  -  A^TW_1=C^T\tC $,
        \Statex \quad\qquad$ -W_2\Lambda  -  A^TW_2 = 2\cdot \cH^{(2)}(V_1\otimes W_1)(\tilde {\cH}^{(2)})^T + \sum\limits_{k=1}^m N_k^TW_1\tN_k $.
         \State Compute $V$ and $W$:
         \Statex \quad\qquad $V:= V_1 +V_2$,\qquad $W := W_1 + W_2$.
         \State $V = \ortho{V}$, $W = \ortho{W}$.
    \State Determine the reduced matrices:
    \Statex \quad\qquad $\hA = (W^T V)^{-1}W^TAV,\qquad~~\hH =(W^T V)^{-1}W^T H(V\otimes V),$
    \Statex \quad\qquad $\hN_{k} = (W^T V)^{-1}W^TN_kV,\qquad \hB = (W^T V)^{-1}W^TB,\qquad\hC = CV $.
    \EndWhile
\end{algorithmic}
\end{algorithm}

\begin{remark}
	Ideally, the word \emph{upon convergence} means that the reduced-order models quantities $\hA$, $\hH$, $\hN_k$, $\hB$, $\hC$ in  \Cref{algo:TQB-IRKA} stagnate.
In a numerical implementation, one can check the stagnation based on the change of eigenvalues of the reduced matrix $\hA$ and terminate the algorithm once the relative change in the eigenvalues of $\hA$ is less than machine precision. However, in our all numerical experiments, we run TQB-IRKA until the relative change in the eigenvalues of $\hA$ is less than $10^{-5}$. We observe that the quality of reduced-order systems does not change significantly, thereafter; as in the cases of IRKA, B-IRKA, and TB-IRKA.
 
\end{remark}
Our next goal is to show how the reduced-order  system resulting from TQB-IRKA upon convergence relates to  first-order optimality conditions~\eqref{eq:optimalityconditions}. As a first step, we  provide explicit expressions showing how farther away the resulting reduced-order system is from satisfying the optimality conditions. Later, based on these expressions, we discuss how the reduced-order systems, obtained from TQB-IRKA for weakly nonlinear QB systems, satisfy the optimality condition with small perturbations. We also illustrate  using  our numerical examples that the reduced-order system satisfies optimality conditions quite accurately.

\begin{theorem}\label{thm:opt_rom}
	Let $\Sigma$ be a QB system~\eqref{sys:QBsys}, and let $\hat\Sigma$ be the reduced-order QB system~\eqref{sys:QBsysRed}, computed by TQB-IRKA upon convergence. Let $V_{\{1,2\}}, W_{\{1,2\}}$, $V$ and $W$ be the projection matrices that solve~\eqref{eq:VW_proj} using the converged reduced-order system. Similarly, let $\hV_{\{1,2\}}, \hW_{\{1,2\}}$, $\hV$ and $\hW$ that solve \eqref{eq:def_VWhat}.  Also, assume that $\sigma(\hA)\cap \sigma(-\Pi A) = \emptyset$ and $\sigma(\hA)\cap \sigma(-\Pi^T A^T) = \emptyset$,  where $\Pi = V(W^TV)^{-1}W^T$ and $\sigma(\cdot)$ denotes the eigenvalue spectrum of a matrix.
Furthermore, assume $\Pi_v = V_1(W^TV_1)^{-1}W^T$ and $\Pi_w = W_1(V^TW_1)^{-1}V^T$ exist. 	Then, the reduced-order system $\hat\Sigma$ satisfies the following relations:
	\begin{subequations}\label{eq:per_optimality}
		\begin{align}
		& \trace{CVe^r_i\left(e^p_j\right)^T } = \trace{\hC \hV e^r_i\left(e^p_j\right)^T } + \epsilon_C^{(i,j)}, \quad\quad~~~~ i \in \{1,\ldots,r\},~~ j \in \{1,\ldots,p\},\label{eq:cond_C1}  \\
		&\trace{B^TWe^r_i\left(e^m_j\right)^T } = \trace{\hB^T\hW e^r_i\left(e^m_j\right)^T} + \epsilon_B^{(i,j)},\quad i \in \{1,\ldots,r\},~~ j \in \{1,\ldots,m\},\label{eq:cond_B1}  \\
		&(W_1(:,i))^T N_k V_1(:,j)  = (\hW_1(:,i))^T \hN_k \hV_1(:,j)+ \epsilon_{N}^{(i,j,k)},  \label{eq:cond_N1} \\ &\hspace{7.9cm} \hfill i,j \in \{1,\ldots,r\},~~ k \in \{1,\ldots,m\}, \nonumber \\
		&  (W_1(:,i))^T  H (V_1(:,j) \otimes V_1(:,l))  = (\hW_1(:,i))^T  \hH (\hV_1(:,j) \otimes \hV_1(:,l))  + \epsilon_H^{(i,j,l)} , \label{eq:cond_H1} \\ &\hspace{8cm} i,j,l \in \{1,\ldots,r\}, \nonumber\\
		&  (W_1(:,i))^T  V(:,i) +  \left(W_2(:,i)\right)^TV_1(:,i)      = (\hW_1(:,i))^T  \hV(:,i) +  \left(\hW_2(:,i)\right)^T \hV_1(:,i) + \epsilon_\lambda^{(i)}, \label{eq:cond_lambda1}  \\& \hspace{8cm} i \in \{1,\dots,r\}. \nonumber
		\end{align}
	\end{subequations}
where
\begin{align*}
\epsilon_C^{(i,j)} & = -\trace{CV \Gamma_v e^r_i\left(e^p_j\right)^T }, \\
\epsilon_B^{(i,j)} & = - \trace{B^TW(W^TV)^{-T} \Gamma_w e^r_i\left(e^m_j\right)^T }  ,\\
\epsilon_{N}^{(i,j,k)} &=  \left(\epsilon_w(:,i)\right)^TN_k(  V_1(:,j)-\epsilon_v(:,j)) + \left(W_1(:,i)\right)^TN_k(\epsilon_v(:,j))  , \displaybreak\\
\epsilon_{H}^{(i,j,l)} &=  \left(W_1(:,i) - \epsilon_w(:,i)\right)^TH(\epsilon_v(:,j)\otimes (V_1(:,l) - \epsilon_{v}(:,l)) + V_1( :,j)\otimes \epsilon_{v}(:,l)) \\
&\qquad + \left(\epsilon_w(:,i)\right)^TH((V(:,j)- \epsilon_v(:,j))\otimes (V_1(:,l) - \epsilon_{v}(:,l))), ~\mbox{and}\\
\epsilon_\lambda^{(i)} &= -\left( \hW(:,i) \right)^T\Gamma_v(:,i) -  \left(\Gamma_{w}(:.i)\right)^T \left( \hV(:,i)- \Gamma_{v}(:,i)\right) \\
&\qquad  -(W_2(:,i))^T  V_2(:,i)  +  (\hW_2(:,i))^T  \hV_2(:,i),
\end{align*}
in which
 $\epsilon_v $, $	\epsilon_w $, $\Gamma_v$ and $\Gamma_w$, respectively, solve
 \begin{subequations}
\begin{align}
	\epsilon_v \Lambda + \Pi A \epsilon_w &= (\Pi -\Pi_{v})(AV_1 + B\tB^T),\label{eq:epsilonv_thm}\\
	\epsilon_w \Lambda +  (A\Pi)^T \epsilon_w &= (\Pi^T-\Pi_{w})(A^TW_1 + C^T\tC),\label{eq:epsilonw_thm}\\
	\Gamma_v \Lambda + \hA \Gamma_v &= -(W^TV)^{-1}W^T \left( \sum_{k=1}^mN_k\epsilon_v\tN_k^T + H(\epsilon_v\otimes (V_1+\epsilon_v)  + V_1\otimes \epsilon_v)\tH^T \right)\\
	\Gamma_w \Lambda + \hA^T \Gamma_w &= V^T \left(\sum_{k=1}^mN^T_k\epsilon_w\tN_k + \cH^{(2)}(\epsilon_v\otimes (W_1+\epsilon_w) + V_1\otimes \epsilon_w)\left(\cH^{(2)}\right)^T \right).
\end{align}
\end{subequations}
\end{theorem}

\begin{proof}
	The proof is given in \Cref{appen:proof_opt}.
	\end{proof}
\begin{remark}  \label{rem:epsilon}
In \Cref{thm:opt_rom}, we have presented measures,  e.g., the distance  between $\trace{CVe^r_i(e^p_j)^T}$ and $\trace{\hC\hV e^r_i(e^p_j)^T}$, denoted $\epsilon_C^{(i,j)}$ with which the reduced-order system, via TQB-IRKA, satisfies the optimality conditions \eqref{eq:optimalityconditions}. Even though \Cref{thm:opt_rom} does not provide a guarantee for smallness of these distances, we indeed observe in all the numerical examples in \Cref{sec:numerics} that these deviations from the exact optimality conditions are rather small. Here, we provide an intuition for these numerical results.  Recall that  $V_1$ and $V_2$ solve the Sylvester equations \eqref{eq:solveV1} and \eqref{eq:solveV2}, respectively, and the right-hand side for $V_2$ is quadratic in $H$ and $N_k$. So, for a weakly nonlinear QB system, meaning  $\|H\|$ and $\|N_k\|$ are small w.r.t.\ $\|B\|$, the matrix $V_2$ will be relatively small  compared to the matrix $V_1$. Hence, the  subspace  $V $ is expected to be close to $V_1$. Thus, one could anticipate that the projectors $\Pi = V(W^TV)W^T$ and $\Pi_{v} = V_1(W^TV_1)W^T$ will be close to each other.  As a result, the right-hand side of the Sylvester equation~\eqref{eq:epsilonv_thm} will be small, and hence thus so is $\epsilon_v$. In a similar way, one can argue that $\epsilon_{w}$ in \eqref{eq:epsilonw_thm} will be small. Therefore, it can be shown that in the case of weakly nonlinear QB systems~\eqref{sys:QBsys},  all $\epsilon$'s in \eqref{eq:per_optimality} such as $\epsilon_C^{(i,j)}$ will be very small.

Indeed, the situation in practice proves much better. 
We observe in our numerical results (see~\Cref{sec:numerics}) that even for strongly nonlinear QB systems, i.e., $\|H\|$ and $\|N_k\|$ are  comparable or even much larger than $\|B\|$, \Cref{algo:TQB-IRKA} yields  reduced-order systems which satisfy the optimality conditions~\eqref{eq:optimalityconditions} almost exactly with negligible
perturbations.
\end{remark}

		\begin{remark}  \label{rem:truncindex}
 \Cref{algo:TQB-IRKA} can be seen as an extension of the \emph{truncated} B-IRKA with the truncation index $2$~\cite[Algo. 2]{flagg2015multipoint} from bilinear systems to QB systems.
	In~\cite{flagg2015multipoint}, the truncation index $\cN$, which denotes the number of terms in the underlying Volterra series for bilinear systems, is free, and  as $\cN \rightarrow \infty$, all the perturbations go to zero. However, it is shown in~\cite{flagg2015multipoint}  that in most cases, a small $\cN$, for example, $2$ or $3$ is enough to satisfy all optimality conditions very closely. In our case, a similar convergence will occur if we let the number of terms in the underlying Volterra series of the QB system grow; however, this is not numerically feasible since the subsystems in the QB case become rather complicated after the first three terms. Indeed, because of this, \cite{morGu09} and \cite{morBenB15} have considered the interpolation of multivariate transfer functions corresponding to only the  first two subsystems. Moreover, even in the case of balanced truncation for QB systems~\cite{morBTQBgoyal}, it is  shown by means of numerical examples that
the truncated Gramians for QB systems based on the first three terms of the underlying Volterra series produce quantitatively very accurate reduced-order systems.  Our numerical examples show that this is the case here as well.
\end{remark}

\begin{remark}
So far in all of our discussions, we have assumed that the reduced matrix $\hA$ is diagonalizable. This is a reasonable assumption since non-diagonalizable	matrices lie in a set of the Lebesgue measure zero. The probability of entering this set by any numerical algorithm including TQB-IRKA is zero with respect to the Lebesgue measure.  Thus, TQB-IRKA can be considered safe in this regard. 

Furthermore, throughout the analysis, it has been assumed that the reduced matrix $\hA$ is Hurwitz. However, in case $\hA$ is not Hurwitz, then the truncated $\cH_2$-norm of the error system will be unbounded; thus the reduced-order systems indeed cannot be (locally) optimal. Nonetheless, a mathematical study to ensure the stability from $\cH_2$ iterative schemes are still under investigation even for linear systems. However, a simple fix to this problem is to reflect the unstable eigenvalues of $\hA$ in every step back to the left-half plane. See also \cite{kohler2014closest}
for a more involved approach to stabilize a reduced order system.
\end{remark}

\begin{remark}
	Thus far, we have used $E = I$ for QB systems; however, in the case of $E \neq I$, but  nonetheless a nonsingular matrix, we can still employ \Cref{algo:TQB-IRKA}. One obvious way is to invert $E$, but this is inadmissible in the large-scale setting. Moreover, the resulting matrices may be dense, making the algorithm computationally expensive.  Nevertheless, \Cref{algo:TQB-IRKA} can be employed without inverting $E$. For this, we need to modify  steps 6 and 7 in \Cref{algo:TQB-IRKA} as follows:
	\begin{align*}
	-EV_1\Lambda  -  AV_1 &= B\tB^T,  \\
	-EV_2\Lambda  - AV_2 &= H(V_1\otimes V_1)\tH^T + \sum\nolimits_{k=1}^m N_kV_1\tN_k^T ,\\
	-E^TW_1\Lambda  -  A^TW_1&=C^T\tC ,  \\
	-E^TW_2\Lambda  -  A^TW_2 &= 2\cdot\cH^{(2)}(V_1\otimes W_1)(\tilde {\cH}^{(2)})^T + \sum\nolimits_{k=1}^m N_k^TW_1\tN_k,
	\end{align*}
	and replace $(W^TV)^{-1}$ with $(W^TEV)^{-1}$, assuming $W^TEV$ is invertible while determining the reduced system matrices in step 9 of \Cref{algo:TQB-IRKA}. Then, the modified iterative algorithm with the matrix $E$ also provides a reduced-order system, \emph{approximately} satisfying optimality conditions.  We skip the rigorous proof for the $E\neq I$  case, but it can be proven along the lines of $E = I$. Indeed, one does not even need to invert $W^T E V$ by letting the reduced QB have a reduced E term as $W^TEV$, and the spectral decomposition of $\hA$ is replaced by a generalized eigenvalue decomposition of $\hE$ and $\hA$. However, to keep the notation of \Cref{algo:TQB-IRKA} simple, we omit these details.
\end{remark}

\subsection{Computational issues}
The main bottleneck in applying {TQB-IRKA} is the computations of the reduced matrices, especially the computational cost related to $\hH := W^TH(V\otimes V)$ that needs to be evaluated at each iteration. Regarding this, there is an efficient method, proposed in~\cite{morBenB15}, utilizing the properties of tensor matricizations, which we summarize in  \Cref{algo:Hessiancomputation}.
\begin{algorithm}[!tb]
 \caption{Computation of the Hessian of the reduced QB system~\cite{morBenB15}.}
 \begin{algorithmic}[1]
    \State Determine $\cY\in\R^{r\times n\times n}$, such that $\cY^{(1)} = W^TH$.
 \State Determine $\cZ \in \R^{r\times r\times n}$, such that $\cZ^{(2)} = V^T\cY^{(2)}$.
 \State Determine $\cX \in \R^{r\times r\times r}$, such that $\cX^{(3)} = V^T \cZ^{(3)}.$
 \State Then, the reduced Hessian is $ \hH = \cX^{(1)}.$
\end{algorithmic}\label{algo:Hessiancomputation}
\end{algorithm}

 \Cref{algo:Hessiancomputation}  avoids the highly undesirable explicit formulation of $V\otimes V$ for large-scale systems to compute the reduced Hessian,  and the algorithm does not rely on  any particular structure of the Hessian. However, a QB system obtained using semi-discretization of a PDEs usually leads to a Hessian which has a particular structure related to that particular PDE and the choice of the discretization method. 

 Therefore, we propose another efficient way to compute $\hH$ that utilizes a particular structure of the Hessian, arising from the governing PDEs or ODEs. Generally, the term  $H(x\otimes x)$ in the QB system~\eqref{sys:QBsys} can be written as  $H(x\otimes x) = \sum_{j =1}^p (\cA^{(j)}x)\circ (\cB^{(j)}x)$, where $\circ$ denotes the Hadamard product, and $\cA^{(j)}$ and $\cB^{(j)}$ are  sparse matrices, depending on the nonlinear operators in PDEs and discretization scheme, and $p$ is generally a very small integer; for instance, it is equal to $1$ in case of Burgers' equations. Furthermore, using the $i$th rows of $\cA^{(j)}$ and $\cB^{(j)
}$, we can construct the $i$th row of the Hessian:
\begin{equation*}
 \begin{aligned}
  H{(i,:)} = \sum\nolimits_{j =1}^p \cA^{(j)}(i,:)\otimes\cB^{(j)}(i,:),
 \end{aligned}
\end{equation*}
where $H{(i,:)}$, $\cA^{(j)}(i,:)$ and $\cB^{(j)}(i,:)$ represent $i$th row of the matrix $H$, $\cA^{(j)}$ and $\cB^{(j)}$, respectively.  This clearly shows that there is  a particular Kronecker structure of the Hessian $H$, which can be used in order to determine $\hH$. Using the Chafee-Infante equation as an example, we illustrate how the structure of the Hessian (Kronecker structure) can be exploited to determine $\hH$ efficiently.
\begin{example}
 Here, we consider the Chafee-Infante equation, which is discretiz- ed over the spatial domain via a finite difference scheme. The MOR problem for this example is considered in~\Cref{exp:chafee}, where one can also find the governing equations and boundary conditions.  For this particular example, the Hessian (after having rewritten the system into the QB form) is given by
 \begin{equation*}
  \begin{aligned}
   H{(i,:)} &= - \dfrac{1}{2}e^n_{i}\otimes e^n_{k+i} - \dfrac{1}{2}e^n_{k+i}\otimes e^n_{i}, ~~i \in \{1,\ldots,k\},\\
     H{(i,:)} &= -2 (e^n_{i}\otimes e^n_{i}) +  e^n_{i{-}k}\otimes \bbm X{(i{-}k,:)} & \0 \ebm  +  \bbm X{(i{-}k,:)} & \0 \ebm \otimes e^n_{i{-}k} ,\\
     &\hspace{8cm} i\in \{k{+}1,\ldots,n\},
  \end{aligned}
 \end{equation*}
where $k$ is the number of grid points, $n = 2 k$, and  $H{(i,:)}$  is the $i$th row vector of the matrix $H$.  $X{(i,:)}$ also denotes the $i$th row vector of the matrix $X \in \R^{k\times k}$
\begin{equation}\label{eq:X_expression}
X = \bbm 0 & 1 &  &  \\ 1 &0 & \ddots&  \\  & \ddots &\ddots & 1 \\ & & 1 &0  \ebm.
\end{equation}
The Kronecker representation of each row of the matrix $H$ allows us to compute each row of  $H_v : = H(V\otimes V)$ by selecting only the required rows of $V$. This way, we can determine $H_v$ efficiently for  large-scale settings, and then multiply with $W^T$ to obtain the desired reduced Hessian. We note down the step in \Cref{algo:chefee_hessian}  that shows how one can determine the reduced Hessian for the Chafee-Infante example.

\begin{algorithm}[!tb]
	\caption{Computation of the reduced Hessian for Chafee-Infante example.}
	\begin{algorithmic}[1]
		\State \textbf{Input:} $V,W\in \R^{2k\times r},X \in \R^{k\times k} (\text{as defined in}~ \eqref{eq:X_expression})$
		\State Compute $V_x := XV(1{:}k,:)$, where $V(1{:}k,:)$ denotes the first $k$ row vectors of $V$.
		\For{$i = 1:k$}
		\State $H_v{(i,:)} =  -\dfrac{1}{2}V{(i,:)}\otimes V{(k+i,:)} -\dfrac{1}{2}V{(i,:)}\otimes V{(k+i,:)}$,
		\Statex ~$H_v{(k+i,:)} =  -2\left(V{(i,:)}\otimes V{(i,:)}\right) + V{(i,:)}\otimes V_x{(i,:)} + V_x{(i,:)}\otimes V{(i,:)} $,
		\Statex ~~~where $H_v{(q,:)}$ is the $q$th row vector of $H_v$ and the same holds for other matrices.
		\EndFor
		\State Then, the reduced Hessian is $ \hH = W^TH_v.$
	\end{algorithmic}
	\label{algo:chefee_hessian}
\end{algorithm}

In order to show the effectiveness of the proposed methodology that uses the special Kronecker structure of the Hessian $H$, we compute $\hH = W^TH(V\otimes V)$ for different orders of original and reduced-order systems and show the required CPU-time to compute it in \Cref{fig:kronprod_CPU}.

\begin{figure}[!tb]
  \centering
  \begin{tikzpicture}
    \begin{customlegend}[legend columns=-1, legend style={/tikz/every even column/.append style={column sep=0.5cm}} , legend entries={ Utilizing Kronecker structure , Using Algorithm~3.2 }, ]
          \addlegendimage{blue,mark = diamond,line width = 1.0pt}            \addlegendimage{black!50!green,mark = o,line width = 1.0pt}
    \end{customlegend}
  \end{tikzpicture}
  \centering
  \setlength\fheight{3cm}  \setlength\fwidth{5.25cm}
  \includetikz{new_kron_scheme_fullorrder}  \includetikz{new_kron_scheme_reducedorder}
  \caption{The left figure shows the computational time for $\hH:= W^TH(V\otimes V)$ by varying the number of grid points in the spatial domain by fixing the order of the reduced system to $r = 20$. In the right figure, we show the computational time for different orders of the reduced system using a fix number of grid points, $k = 1000$.}
  \label{fig:kronprod_CPU}
 \end{figure}
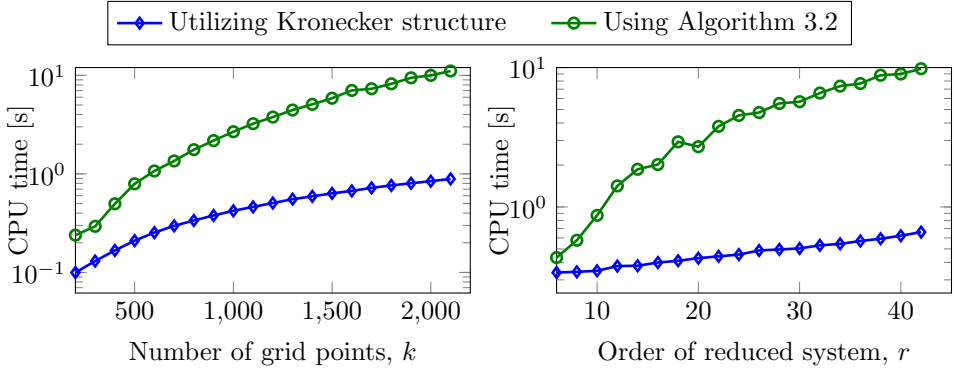

 These figures show that if one utilizes the  present Kronecker structure of the Hessian $H$, then the computational cost does not increase as rapidly as one would use~\Cref{algo:Hessiancomputation} to compute the reduced Hessian. So, our conclusion here is that it is worth to exploit the Kronecker structure  of the Hessian of the system for an efficient computation of $\hH$ in the large-scale settings.
\end{example}

%% file: numerical_results.tex
\definecolor{mycolor1}{rgb}{1.00000,0.00000,1.00000}%
\definecolor{mycolor2}{rgb}{0.00000,1.00000,1.00000}%


\section{Numerical Results}\label{sec:numerics}
In this section, we illustrate the efficiency of the proposed model reduction method  TQB-IRKA for QB systems by means of several semi-discretized nonlinear PDEs, and compare it with the existing MOR techniques, such as one-sided and two-sided subsystem-based interpolatory projection methods~\cite{morBenB15,morGu09,morPhi03},  balanced truncation (BT) for QB systems~\cite{morBTQBgoyal} and POD, e.g., see~\cite{morHinV05,morKunV08}.  We iterate \Cref{algo:TQB-IRKA} until the relative change in the eigenvalues of $\hA$ becomes smaller than a tolerance level. We set this tolerance to $1e{-}5$. Moreover, we determine interpolation points for the one-sided and two-sided interpolatory projection methods using {IRKA}~\cite{morGugAB08} on the corresponding linear part, which appear to be a good set of interpolation points as shown in~\cite{morBenB15}. All the simulations were done on 
 a board with 4 \intel ~\xeon~E7-8837 CPUs with a 2.67-GHz clock speed using \matlab ~8.0.0.783 (R2012b). Some more details related to the numerical examples are as follows:

\begin{enumerate}
 \item All original and reduced-order systems are integrated by the routine {\tt ode15s} in \matlab~with relative error tolerance of $10^{-8}$  and absolute error tolerance of $10^{-10}$. 
 \item We measure the output at 500 equidistant points within the time interval $[0,T]$, where $T$ is defined in each numerical example.
\item In order to employ BT, we need to solve four standard Lyapunov equations. 
For this, we use {\tt mess\_lyap.m} from \mess-1.0.1 which is based on one of  the latest ADI methods proposed in~\cite{BenKS14b}.

\item We initialize TQB-IRKA (\Cref{algo:TQB-IRKA}) by choosing an arbitrary reduced system while ensuring $\hA$ is Hurwitz and diagonalizable. One can also first compute the projection matrices $V$ and $W$ by employing IRKA on the corresponding  linear system and, using it, we can compute the initial reduced-order system.  But we observe that random reduced systems  give almost the same convergence rate or even better sometimes. Therefore, we stick to a random initial guess selection.
\item Since POD can be applied to a general nonlinear system, we apply POD to the original nonlinear system, without transforming it into a QB system as we observe that this way, POD yields better reduced systems.


\item One of the aims of the numerical examples is to determine the residuals in \Cref{thm:opt_rom}. For this, we first define $\Phi^e_C\in \R^{r\times p}$,  $\Phi^e_B\in \R^{r\times m}$, $\Phi^e_N \in \R^{r\times  r\times m}$,  $\Phi^e_H\in \R^{r\times r\times r}$ and  $\Phi^e_\Lambda\in \R^{r}$  such  that $\epsilon_C^{(i,j)}$ is  the $(i,j)$th entry of $\Phi^e_C$, $\epsilon_B^{(i,j)}$ is the $(i,j)$th entry of $\Phi^e_B$, $\epsilon_N^{(i,j,k)}$ is the $(i,j,k)$th entry of $\Phi^e_N$, $\epsilon_H^{(i,j,k)}$ is the $(i,j,k)$th entry of $\Phi^e_H$ and $\epsilon_\Lambda^{(i)}$is $i$th entry of $\Phi^e_\Lambda$. 

Furthermore, we define $\Phi_C$, $\Phi_B$, $\Phi_N$, $\Phi_H$ and  $\Phi_\Lambda$ to be the terms on the left hand side  of equations \eqref{eq:cond_C1} -- \eqref{eq:cond_lambda1}  in \Cref{thm:opt_rom}, e.g., the $(i,j)$th entry of $\Phi_C$ is $\trace{CVe_i^r\left(e_j^p\right)^T}$. As a result, we define a  relative perturbation measures as follows:
\begin{equation}\label{eq:perturb_def}
\cE_C = \dfrac{\|{\Phi^e_C}\|_2}{\|{\Phi_C}\|_2},~\cE_B = \dfrac{\|{\Phi^e_B}\|_2}{\|{\Phi_B}\|_2},~
\cE_N = \dfrac{\|{\Phi^{e1}_N}\|_2}{\|{\Phi^{(1)}_N}\|_2}, ~ \cE_H = \dfrac{\|{\Phi^{e1}_H}\|_2}{\|{\Phi^{(1)}_H}\|_2},  ~\cE_\Lambda = \dfrac{\|{\Phi^e_\Lambda}\|_2}{\|{\Phi_\Lambda}\|_2},
\end{equation}
where $\Phi^{(1)}_{\{N,H\}}$ and $\Phi^{e1}_{\{N,H\}}$, respectively are mode-1 matricizations of  the tensors $\Phi_{\{N,H\}}$ and $\Phi^{e1}_{\{N,H\}}$.
\item We also address a numerical issue which one might face while employing \Cref{algo:TQB-IRKA}. In step 8 of \Cref{algo:TQB-IRKA}, we need to take a sum  of two matrices $V_1$ and $V_2$, but if $H$ and $N_k$ are too large, then the norm of $V_2$ can be much larger than that of $V_1$. Thus, a  direct sum might reduce the effect of $V_1$. As a remedy we propose to use a scaling factor $\gamma$ for $H$ and $N_k$, thus resulting in the matrices $V_1$ and $V_2$ which have almost the same order of norm. We have already noted in \Cref{rescalling_discussions} that this scaling does not change the input-output behavior.
\end{enumerate}

\subsection{One dimensional Chafee-Infante equation}\label{exp:chafee}
Here, we consider the one-dimensional Chafee-Infante (Allen-Cahn) equation whose governing equation, initial condition and boundary controls are given by
\begin{equation}\label{sys:Chafee_infante}
 \begin{aligned}
  \dot{v} +  v^3 &= v_{xx} + v, & ~~ &(0,L)\times (0,T),&\qquad
  v(0,\cdot)  &= u(t), & ~~ & (0,T),\\
  v_x(L,\cdot) &= 0, & ~~ & (0,T),&
  v(x,0) &= 0, & ~~ & (0,L).
 \end{aligned}
\end{equation}
MOR for this system has been considered in various articles; see, e.g.,~\cite{morBenB15,morBTQBgoyal}. The governing equation \eqref{sys:Chafee_infante} contains a cubic nonlinearity, which  can then be rewritten into  QB form as shown in~\cite{morBenB15}. For more details on the system, we refer to~\cite{chafee1974bifurcation,hansen2012second}.  Next, we  utilize a finite difference scheme by using $k$ equidistant points over the length, resulting in a semi-discretized QB system of order $2 k$. The output of our interest  is the response at the right boundary, i.e., $v(L,t)$, and we  set the number of grid points to $k = 500$, leading to an order   $n = 1000$  QB system.

We construct reduced-order systems of order $r = 10$ using TQB-IRKA, BT, one-sided and two-sided interpolatory projection methods, and POD.  Having initialized TQB-IRKA randomly, it takes $9$ iterations to converge, and for this example, we choose the scaling factor $\gamma = 0.01$. We compute the reduced Hessian as shown in \Cref{algo:chefee_hessian}.   For the POD based approximation, we collect $500$ snapshots of the true solution for the training input $u_1(t) =\left(1+\sin(\pi t)\right)\exp(-t/5) $ and compute the projection by taking the  $10$ dominant basis vectors.

In order to compare the quality of these reduced-order systems with respect to the original system, we first simulate them using the same training input used to construct the POD basis, i.e., $u_1(t) = \left(1+\sin(\pi t )\right)\exp(-t/5) $. We plot the transient responses and relative output errors  for the input in \Cref{fig:chafee_input1}. As expected,
since we are comparing the reduced models for the same forcing term used for POD,
\Cref{fig:chafee_input1} shows that the POD approximation outperforms the other  methods for the input $u_1$. However, the interpolatory methods also provide adequate reduced-order systems for $u_1(t)$ even though the reduction is performed without any knowledge of $u_1(t)$.
\begin{figure}[!tb]
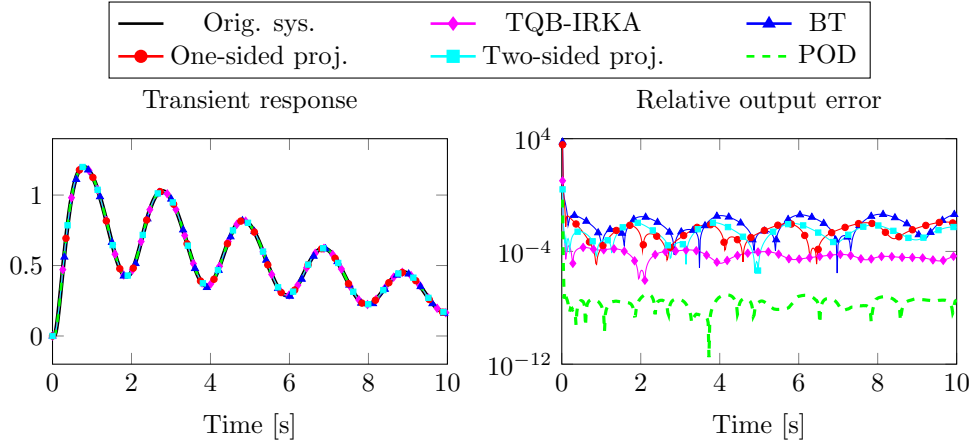

  \centering
  \begin{tikzpicture}
    \begin{customlegend}[legend columns=3, legend style={/tikz/every even column/.append style={column sep=1.0cm}} , legend entries={Orig. sys.,TQB-IRKA,BT, One-sided proj. , Two-sided proj., POD}, ]
      \addlegendimage{black,solid,line width = 1pt}
       \addlegendimage{mycolor1,solid,line width = 1pt, mark = diamond*}      \addlegendimage{blue,line width = 1pt, mark = triangle*}      \addlegendimage{red,line width = 1pt, mark = *}      \addlegendimage{mycolor2,line width = 1pt,mark =square* } \addlegendimage{green,line width = 1pt, dashed}
    \end{customlegend}
  \end{tikzpicture}
  \centering
  \setlength\fheight{3cm}  \setlength\fwidth{5.25cm}
	\tikzsetnextfilename{Figures_Revised/Chafee_Input1_response}%
	\input{Figures_Revised/Chafee_Input1_response.tikz}%
	\tikzsetnextfilename{Figures_Revised/Chafee_Input1_error}%
	\input{Figures_Revised/Chafee_Input1_error.tikz}%

  \caption{Chafee-Infante: comparison of responses for the boundary control input $u_1(t) = \left(1+\sin(\pi t )\right)\exp(-t/5) $.}
  \label{fig:chafee_input1}
 \end{figure}

To test the robustness of the reduced systems, we compare the time-domain simulations of the reduced systems with the original one in \Cref{fig:chafee_input2} for a slightly different input, namely $u_2(t) = 25\left(1+\sin(\pi t )\right) $. First, observe that  the POD approximation fails to reproduce the system's dynamics for the input $u_2$ accurately as POD is mainly an input-dependent algorithm. Moreover, the one-sided interpolatory projection method also performs worse for the input $u_2$. On the other hand, TQB-IRKA, BT, and the two-sided interpolatory projection method, all yield  very accurate reduced-order systems of comparable qualities;  TQB-IRKA produces marginally better reduced systems.  Once again it is important to emphasize that neither $u_1(t)$ nor $u_2(t)$ have entered the model reduction procedure in TQB-IRKA.  To give a quantitative comparison of the reduced systems for both inputs, $u_1$ and $u_2$, we report the  mean relative errors in \Cref{tab:chafee_RE} as well, which also provides us a similar information.

\begin{table}[tb!] \label{tab:chafee_RE}
	\begin{center}
		\begin{tabular}{ |c| c |c |c |c| c|}\hline
			Input & TQB-IRKA& BT & One-sided  & Two-sided  & POD \\ \hline
			$u_1(t)$&   $6.54\cdot 10^{-5}$ & $1.40\cdot 10^{-2}$ & $4.30\cdot 10^{-3}$ & $3.51\cdot 10^{-3}$ & $2.87\cdot 10^{-8}$\\ \hline
			$u_2(t)$ &  $1.63\cdot 10^{-3}$ & $1.43\cdot 10^{-2}$ & $4.59\cdot 10^{-1}$&$6.65\cdot 10^{-3}$ & $6.70\cdot 10^{-2}$\\ \hline
		\end{tabular}
	\end{center}
	\caption{Chafee-Infante: the mean relative errors of the output.}
\end{table}

 \begin{figure}[!tb]
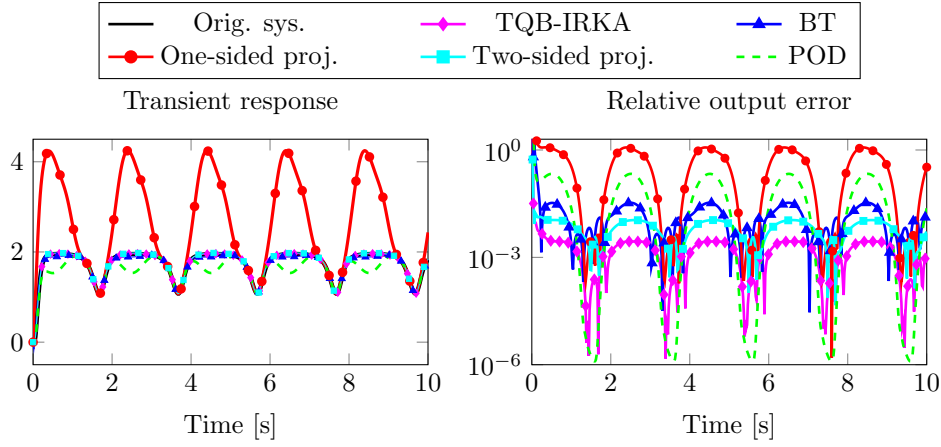

    \centering
  \begin{tikzpicture}
    \begin{customlegend}[legend columns=3, legend style={/tikz/every even column/.append style={column sep=1.0cm}} , legend entries={Orig. sys.,TQB-IRKA,BT, One-sided proj. , Two-sided proj., POD}, ]
    \addlegendimage{black,solid,line width = 1pt}
       \addlegendimage{mycolor1,solid,line width = 1pt, mark = diamond*}      \addlegendimage{blue,line width = 1pt, mark = triangle*}      \addlegendimage{red,line width = 1pt, mark = *}      \addlegendimage{mycolor2,line width = 1pt,mark =square* } \addlegendimage{green,line width = 1pt, dashed}
    \end{customlegend}
  \end{tikzpicture}
  \centering
  \setlength\fheight{3cm}  \setlength\fwidth{5.25cm}
	\tikzsetnextfilename{Figures_Revised/Chafee_Input2_response}%
	\input{Figures_Revised/Chafee_Input2_response.tikz}%
	\tikzsetnextfilename{Figures_Revised/Chafee_Input2_error}%
	\input{Figures_Revised/Chafee_Input2_error.tikz}%

 \caption{Chafee-Infante:  comparison of responses for the boundary control  input $u_2(t) = 25 (1+\sin(\pi t) ) $.}
\label{fig:chafee_input2}
 \end{figure}

In \Cref{thm:opt_rom},  we have  presented the quantities, denoted by $\epsilon_X$, where $X = \{C,B,N,H,\lambda\}$ which measure how  far the reduced-order system upon convergence of TQB-IRKA is from  satisfying the optimality conditions~\eqref{eq:optimalityconditions}. These quantities can be computed as shown in \eqref{eq:perturb_def}, and  are listed in \Cref{table:per_opt}, which shows a very small  magnitude of these perturbations.
\begin{table}[!tbp]
	\begin{center}
	\begin{tabular}{| c| c| c| c| c|}
		\hline
		 $\cE_C$&$\cE_B$&$\cE_N$&$\cE_H$&$\cE_\lambda$ \\ \hline
		  $1.35\times 10^{-8}$ & $8.85\times 10^{-12}$ & $8.84\times 10^{-16}$   & $1.77\times 10^{-13}$ & $1.44\times 10^{-11}$\\
		  \hline
	\end{tabular}
		\end{center}
				\caption{Chafee-Infante: perturbations to the optimality conditions.}
				\label{table:per_opt}
\end{table}
In  \Cref{rem:epsilon}, we have argued that for a weakly nonlinear QB system, we expect these quantities to be small. However, even for this  example with strong nonlinearity, i.e., $\|H\|$ and $\|N_k\|$ are not small at all,  the reduced-order system computed by TQB-IRKA  satisfies the optimality conditions~\eqref{eq:optimalityconditions} very accurately. This result also strongly supports the discussion of \Cref{rem:truncindex} that a small truncation index is expected to be enough in many cases.

Furthermore, since TQB-IRKA approximately minimizes the truncated $\cH_2$-norm of the error system, i.e. $\|\Sigma-\hat{\Sigma}\|_{\cH^{(\cT)}_2}$, we also  compare the truncated $\cH_2$-norm of the error system in \Cref{fig:chafee_h2norm}, where the $\hat{\Sigma}$ are constructed by various methods of different orders. As mentioned before, the reduced-order systems obtained via POD  preserve the structure of the original nonlinearities; therefore, the truncated $\cH_2$-norm definition, given in \Cref{lemma:H2_norm}, does not apply.

\begin{figure}[!tb]
	\centering
	\begin{tikzpicture}
	\begin{customlegend}[legend columns=-1, legend style={/tikz/every even column/.append style={column sep=0.75cm}} , legend entries={TQB-IRKA,BT, One-sided proj. , Two-sided proj.}, ]
	 \addlegendimage{mycolor1,solid,line width = 1pt, mark = diamond*}      \addlegendimage{blue,line width = 1pt, mark = triangle*}      \addlegendimage{red,line width = 1pt, mark = *}      \addlegendimage{mycolor2,line width = 1pt,mark =square* }
	\end{customlegend}
	\end{tikzpicture}
	\centering
	\setlength\fheight{3cm}  \setlength\fwidth{5.25cm}
	\tikzsetnextfilename{Figures_Revised/Chafee_RelativeH2}%
	\input{Figures_Revised/Chafee_RelativeH2.tikz}%

	\caption{Chafee-Infante:  comparison of the truncated $\cH_2$-norm of the error system, having obtained reduced systems of different orders via different methods. }
	\label{fig:chafee_h2norm}
\end{figure}
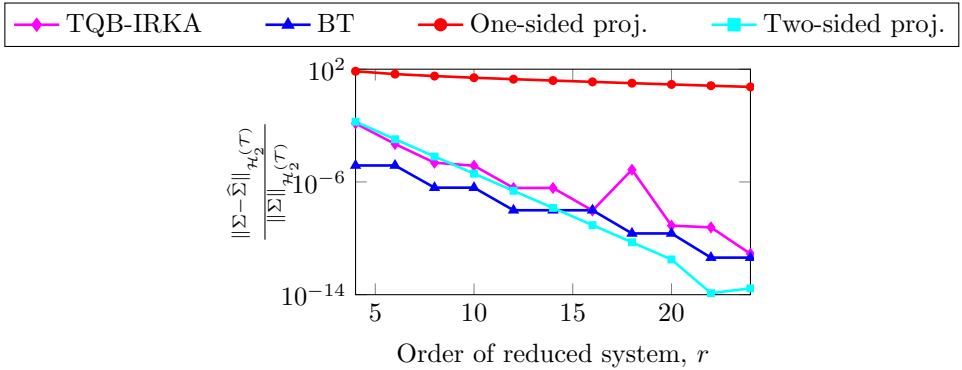

\Cref{fig:chafee_h2norm} indicates that the reduced-order systems obtained via one-sided interpolatory projection perform worst in the truncated $\cH_2$-norm, whereas the quality of the reduced-order systems obtained by TQB-IRKA, BT, and the two-sided interpolatory projection method are  comparable with respect to the truncated $\cH_2$-norm.  This also potentially explains why the reduced systems obtained via TQB-IRKA, BT and two-sided interpolatory projection are of the same quality for both control inputs $u_1$ and $u_2$, and the one-sided interpolatory projection  method provides the worst reduced systems, see~\Cref{fig:chafee_input2}. However,  it is important to emphasize that unlike for  linear dynamical systems,  the $\cH_2$-norm  and the $L^\infty$-norm of the output for nonlinear systems, including  QB systems, are not directly connected. This can be seen in  \Cref{fig:chafee_h2norm} as for reduced systems of order $r=10$, even though BT yields the smallest truncated $\cH_2$-norm of the error
system, but in time-domain simulations for inputs $u_1$ and $u_2$, it is not the best in terms of the $L^\infty$-norm of the output error. Nevertheless, the truncated $\cH_2$-norm of the error system is still a robust indicator for the quality of the reduced system.

 \subsection{A nonlinear RC ladder}\label{exa:RC}
We consider a nonlinear RC ladder, which consists of capacitors and nonlinear  I-V diodes. The characteristics of the  I-V diode are governed by exponential nonlinearities, which can also be rewritten in  QB form. For a  detailed description of the dynamics of this electronic circuit, we refer to~\cite{bai2002krylov,morGu09,li2005compact,phillips2000projection,morPhi03}. We set the number of capacitors in the ladder to $k = 500$, resulting in a QB system of order $n = 1000$. Note that the matrix $A$ of the obtained QB system has eigenvalues at zero; therefore, the  truncated $\cH_2$-norm does not exist. Moreover,  BT also cannot be employed as we need to solve Lyapunov equations that require a stable $A$ matrix. Thus, we shift the matrix $A$ to $A_s := A -0.01I_n$  to determine the projection matrices for TQB-IRKA and BT but we project the original system matrices.

We construct  reduced-order systems of order $r =10$ using all five different methods. In this example as well, we initialize TQB-IRKA randomly and it converges after $27$ iterations. We choose the scaling factor $\gamma = 0.01$. For this example, we determine the reduced Hessian by exploiting the particular structure of the Hessian. In order to compute reduced-order systems via POD, we first obtain $500$ snapshots of the true solution for the training input $u_1 = e^{-t}$ and then use the $10$ dominant modes to determine the projection.  

%
%

We first compare the accuracy of these reduced systems for the same training input $u_1(t) = e^{-t}$ that is also used to compute the POD basis. \Cref{fig:RC_input1} shows the transient responses and relative errors of the output for the input $u_1$. As one would expect, POD outperforms all other methods since the control input $u_1$ is the same as the training input for POD. Nonetheless, TQB-IRKA, BT, and two-sided projection also yield very good reduced-order systems, considering they are obtained without any prior knowledge of the input.

We also test the reduced-order systems for a different input than the training input, precisely, $u_2 = 2.5\left(\sin(\pi t/5)+1\right)$.  \Cref{fig:RC_input2} shows the transient responses and relative errors of the output for the input $u_2$.  We observe that POD does not perform as good as the other methods, such as TQB-IRKA, BT and two-sided projection methods, and the one-sided projection method completely fails to capture the system dynamics for the input $u_2$.  This can also be observed from \Cref{tab:RC_RE}, where  the mean relative errors of the outputs are reported.

\begin{figure}[!tb]
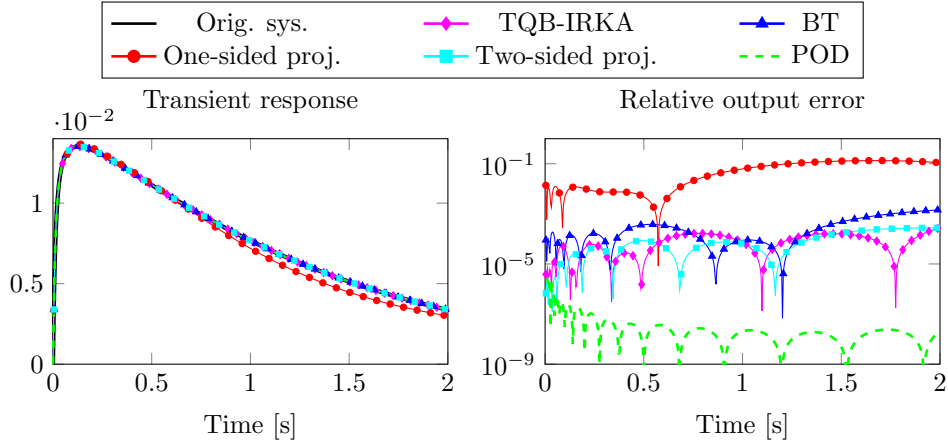

        \centering
        \begin{tikzpicture}
     \begin{customlegend}[legend columns=3, legend style={/tikz/every even column/.append style={column sep=1.0cm}} , legend entries={Orig. sys.,TQB-IRKA,BT, One-sided proj. , Two-sided proj., POD}, ]
    \addlegendimage{black,solid,line width = 1pt}
       \addlegendimage{mycolor1,solid,line width = 1pt, mark = diamond*}      \addlegendimage{blue,line width = 1pt, mark = triangle*}      \addlegendimage{red,line width = 1pt, mark = *}      \addlegendimage{mycolor2,line width = 1pt,mark =square* } \addlegendimage{green,line width = 1pt, dashed}
     \end{customlegend}
 \end{tikzpicture}

  \centering
	\setlength\fheight{3cm}
	\setlength\fwidth{5.25cm}
	\tikzsetnextfilename{Figures_Revised/RC_Input1_response}%
	\input{Figures_Revised/RC_Input1_response.tikz}%

	\tikzsetnextfilename{Figures_Revised/RC_Input1_error}%
	\input{Figures_Revised/RC_Input1_error.tikz}%

 \caption{ An RC circuit: comparison of responses for the input $u_1(t) = e^{-t}$.}
 \label{fig:RC_input1}
 \end{figure}
 \begin{figure}[!tb]
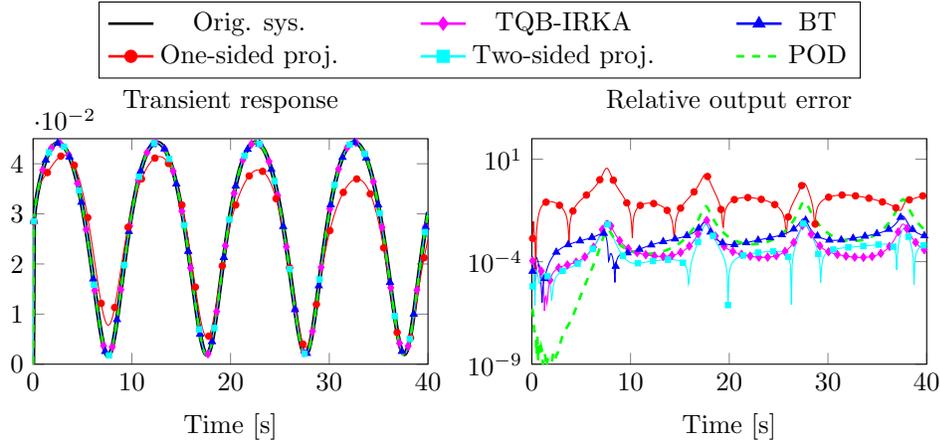

        \centering
        \begin{tikzpicture}
     \begin{customlegend}[legend columns=3, legend style={/tikz/every even column/.append style={column sep=1.0cm}} , legend entries={Orig. sys.,TQB-IRKA,BT, One-sided proj. , Two-sided proj.,POD}, ]
    \addlegendimage{black,solid,line width = 1pt}
       \addlegendimage{mycolor1,solid,line width = 1pt, mark = diamond*}      \addlegendimage{blue,line width = 1pt, mark = triangle*}      \addlegendimage{red,line width = 1pt, mark = *}      \addlegendimage{mycolor2,line width = 1pt,mark =square* } \addlegendimage{green,line width = 1pt, dashed}
     \end{customlegend}
 \end{tikzpicture}
   \centering
	\setlength\fheight{3cm}
	\setlength\fwidth{5.25cm}
	\tikzsetnextfilename{Figures_Revised/RC_Input2_response}%
	\input{Figures_Revised/RC_Input2_response.tikz}%

	\tikzsetnextfilename{Figures_Revised/RC_Input2_error}%
	\input{Figures_Revised/RC_Input2_error.tikz}%

 \caption{An RC circuit:  comparison of responses for the input $u_2 = 2.5\left(\sin(\pi t/5)+1\right)$.}
\label{fig:RC_input2}
 \end{figure}

 \begin{table}[!tb]\label{tab:RC_RE}
\begin{center}
 \begin{tabular}{ |c| c |c |c |c| c|}\hline
  Input & TQB-IRKA& BT & One-sided  & Two-sided  & POD \\ \hline
  $u_1(t)$&   $8.82\cdot 10^{-5}$ & $3.67\cdot 10^{-4}$ & $6.50\cdot 10^{-2}$ & $1.01\cdot 10^{-4}$ & $7.24\cdot 10^{-8}$\\ \hline
  $u_2(t)$ &  $1.12\cdot 10^{-3}$ & $2.15\cdot 10^{-3}$ & $2.32\cdot 10^{-1}$&$7.80\cdot 10^{-4}$ & $7.8\cdot 10^{-3}$\\ \hline
\end{tabular}
	\end{center}
\caption{An RC circuit: the mean absolute errors of the output.}
\end{table}

Further,  we compute the quantities as defined in \eqref{eq:perturb_def} using the reduced system of order $r=10$ obtained upon convergence of TQB-IRKA and  list them in \Cref{tab:per_opt_RC}. This also indicates that the obtained reduced-order system using TQB-IRKA satisfies all the optimality conditions~\eqref{eq:optimalityconditions} very accurately even though the nonlinear part of the system plays a significant role in the system dynamics.

\begin{table}[!tbp]\label{tab:per_opt_RC}
	\begin{center}
	\begin{tabular}{| c| c| c| c| c|}
		\hline
		 $\cE_C$&$\cE_B$&$\cE_N$&$\cE_H$&$\cE\lambda$ \\ \hline
		  $3.99\times 10^{-10}$ & $4.68\times 10^{-8}$ & $3.91\times 10^{-7}$   & $3.37\times 10^{-8}$ & $3.91\times 10^{-8}$\\
		  \hline
	\end{tabular}
		\end{center}
				\caption{An RC circuit: perturbations to the optimality conditions.}

\end{table}

Next, we also  compare the truncated $\cH_2$-norm of the error system, i.e., $\|\Sigma-\hat{\Sigma}\|_{\cH^{(\cT)}_2}$, in \Cref{fig:RC_h2norm}, where the $\hat{\Sigma}$ are constructed by various methods of different orders. The figure explains that TQB-IRKA yields better reduced-order systems with respect to the truncated $\cH_2$-norm.

At last, we mention here again that the reduced system  obtained via POD retains  the original exponential  nonlinearities; therefore, we cannot use the same definition of the truncated $\cH_2$-norm as in \Cref{lemma:H2_norm} for such nonlinear systems. Hence, POD is omitted in~\Cref{fig:RC_h2norm}.

\begin{figure}[!tb]
	\centering
	\begin{tikzpicture}
	\begin{customlegend}[legend columns=-1, legend style={/tikz/every even column/.append style={column sep=0.25cm}} , legend entries={TQB-IRKA,BT, One-sided proj. , Two-sided proj.}, ]
	 \addlegendimage{mycolor1,solid,line width = 1pt, mark = diamond*}      \addlegendimage{blue,line width = 1pt, mark = triangle*}      \addlegendimage{red,line width = 1pt, mark = *}      \addlegendimage{mycolor2,line width = 1pt,mark =square* }
	\end{customlegend}
	\end{tikzpicture}
	\centering
	\setlength\fheight{3cm}  \setlength\fwidth{5.25cm}
	\tikzsetnextfilename{Figures_Revised/RC_RelativeH2}%
	\input{Figures_Revised/RC_RelativeH2.tikz}%

	\caption{An RC circuit:  comparison of the truncated $\cH_2$-norm of the error system obtained via different methods of various orders. }
	\label{fig:RC_h2norm}
\end{figure}
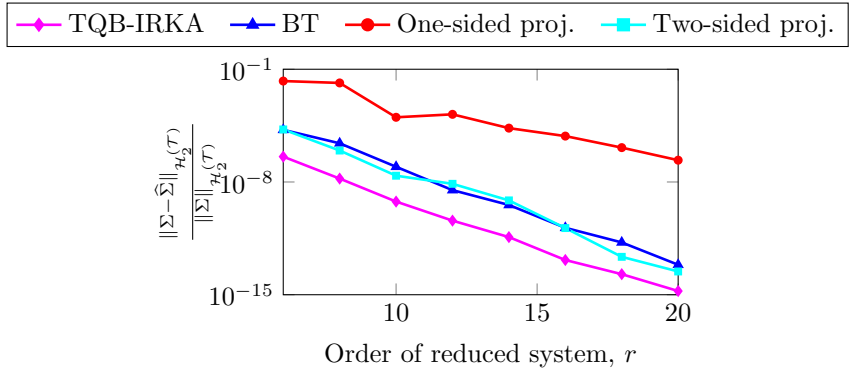

\subsection{The FitzHugh-Nagumo (F-N) system}\label{exa:FN}
This example considers the F-N system,  describing  activation and deactivation dynamics of  spiking neurons. This model is a  simplification of the Hodgkin-Huxley neuron model.
The dynamics of the system is governed by the following nonlinear coupled PDEs:
\begin{equation*}
\begin{aligned}
\epsilon v_t(x,t) & =\epsilon^2v_{xx}(x,t) + f(v(x,t)) -w(x,t) + q,\\
w_t(x,t) &= hv(x,t) -\gamma w(x,t) + q
\end{aligned}
\end{equation*}
with a nonlinear function $f(v(x,t)) = v(v-0.1)(1-v)$  and initial and boundary conditions as follows:
\begin{equation*}
\begin{aligned}
v(x,0) &=0, &\qquad\qquad w(x,0)&=0,\qquad \qquad & &x\in (0,L),\\
v_x(0,t) &= i_0(t), & v_x(1,t) &= 0, & &t\geq 0,
\end{aligned}
\end{equation*}
where $\epsilon = 0.015,~h=0.5,~\gamma = 2,~q = 0.05$, $L = 0.3$, and $i_0(t)$ is an actuator, acting as control input. The voltage and recovery voltage are denoted by $v$ and $w$, respectively. MOR for this model has been considered in~\cite{morBenB12a,morBTQBgoyal,morChaS10}. Furthermore,  we also consider the same output as considered in~\cite{morBenB12a,morBTQBgoyal}, which is the limit-cycle at the left boundary, i.e., $x = 0$. The system can be considered as having two inputs, namely $q$ and $i_0(t)$;  it has also two outputs, which are $v(0,t)$ and $w(0,t)$. This means that the system is a multi-input multi-output system as opposed to the two previous examples.  We discretize the governing equations using  a finite difference scheme. This leads to an ODE system, having cubic nonlinearity, which can then be transformed into the QB form. We consider $k=300$ grid points, resulting in a QB system of order $3  k = 900$.

We next determine reduced systems of order $r = 35$ using TQB-IRKA, BT, and POD. We choose the scaling factor $\gamma = 1$ in TQB-IRKA and it requires  $26$ iterations to converge. For this example, we also utilize the Kronecker structure of the Hessian to perform an efficient computation of the reduced Hessian.   In order to apply POD, we first collect $500$ snapshots of the original system for the time interval $(0,10s]$ using $i_0(t) = 50(\sin(2\pi t) -1)$ and then determine the projection based on the $35$ dominant modes.  The one-sided and two-sided subsystem-based interpolatory projection methods have major disadvantages in the MIMO QB case. The one-sided interpolatory projection approach of \cite{morGu09} can be applied to MIMO QB systems, however the dimension of the subspace $V$, and thus the dimension of the reduced model, increases quadratically due to the $V \otimes V$ term.
As we mentioned in \Cref{sec:intro}, two-sided interpolatory projection is only applicable to SISO QB systems. When the number of inputs and outputs are the same, which is the case in this example, one can still  employ~\cite[Algo. 1]{morBenB15}  to construct a reduced system. This is exactly what we did here. However, it is important to note that even though the method can be  applied numerically, it no longer ensures the theoretical subsystem interpolation property. Despite these drawbacks, for completeness of the comparison, we still construct reduced models using both one-sided and two-sided subsystem-based interpolatory projections.


Since the FHN system has two inputs and two outputs,  each interpolation point yields $6$ vectors in projection matrices $V$ and $W$. Thus, in order to apply the two-sided projection, we use $6$ linear $\cH_2$-optimal points and determine the reduced system of order $35$ by taking the  $35$ dominant vectors. We do the same for the one-sided interpolatory projection method to compute the reduced-order system.

%
%
Next, we compare the quality of the reduced-order systems and plot the transient responses and the absolute errors of the outputs in \Cref{fig:Fitz} for the training input $i_0(t) = 50(\sin(2\pi t) -1)$. 

\begin{figure}[!tb]
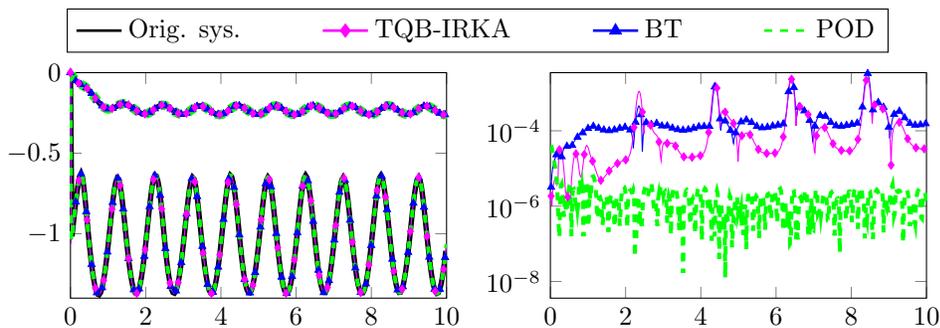

	\centering
	\begin{tikzpicture}
	\begin{customlegend}[legend columns=4, legend style={/tikz/every even column/.append style={column sep=1.0cm}} , legend entries={Orig. sys.,TQB-IRKA,BT,POD}, ]
	\addlegendimage{black,solid,line width = 1pt}
	\addlegendimage{mycolor1,solid,line width = 1pt, mark = diamond*}      \addlegendimage{blue,line width = 1pt, mark = triangle*}     \addlegendimage{green,line width = 1pt, dashed}
	\end{customlegend}
	\end{tikzpicture}
	
	\centering
	\setlength\fheight{3cm}
	\setlength\fwidth{5.0cm}
	\tikzsetnextfilename{Figures_Revised/Fitz_response1}%
	\input{Figures_Revised/Fitz_response1.tikz}%

	\tikzsetnextfilename{Figures_Revised/Fitz_error1}%
	\input{Figures_Revised/Fitz_error1.tikz}%

	\caption{The FitzHugh-Nagumo  system: comparison of the limit-cycle at the left boundary, $x = 0$ for $i_0(t) = 50(\sin(2\pi t)-1)$.}
	\label{fig:Fitz}
\end{figure}

\begin{figure}[!tb]
	\centering
	\begin{tikzpicture}
	\begin{customlegend}[legend columns=4, legend style={/tikz/every even column/.append style={column sep=1.0cm}} , legend entries={Orig. sys.,TQB-IRKA,BT,POD}, ]
	\addlegendimage{black,solid,line width = 1pt}
	\addlegendimage{mycolor1,solid,line width = 1pt, mark = diamond*}      \addlegendimage{blue,line width = 1pt, mark = triangle*}     \addlegendimage{green,line width = 1pt, dashed}
	\end{customlegend}
	\end{tikzpicture}
	
	\centering
	\setlength\fheight{3cm}
	\setlength\fwidth{5.0cm}
	\tikzsetnextfilename{Figures_Revised/Fitz_response2}%
	\input{Figures_Revised/Fitz_response2.tikz}%

	\tikzsetnextfilename{Figures_Revised/Fitz_error2}%
	\input{Figures_Revised/Fitz_error2.tikz}%

	\caption{The FitzHugh-Nagumo  system: comparison of the limit-cycle at the left boundary, $x = 0$ for $i_0(t) = 5\cdot 10^4t^3 \exp(-15t)$.}
	\label{fig:Fitz2}
\end{figure}

\begin{figure}[!tb]
	\centering
	\begin{tikzpicture}
	\begin{customlegend}[legend columns=4, legend style={/tikz/every even column/.append style={column sep=1.0cm}} , legend entries={Orig. sys.,TQB-IRKA,BT,POD}, ]
	\addlegendimage{black,solid,line width = 1pt}
	\addlegendimage{mycolor1,solid,line width = 1pt, mark = diamond*}      \addlegendimage{blue,line width = 1pt, mark = triangle*}     \addlegendimage{green,line width = 1pt, dashed}
	\end{customlegend}
	\end{tikzpicture}
	\centering
	\setlength\fheight{3cm}
	\setlength\fwidth{5.25cm}
	\tikzsetnextfilename{Figures_Revised/Fitz_3D_limit}%
	\input{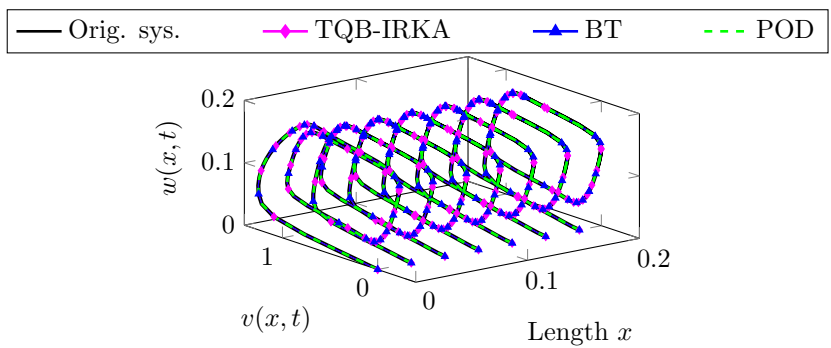}%

	\caption{The FitzHugh-Nagumo system: limit-cycle behavior of the original and reduced-order systems in the spatial domain.}
	\label{fig:Fitz_limitcycle}
\end{figure}


As anticipated, POD provides a very good reduced-order system since the POD basis is constructed by using the same trajectory.  Note that despite not reporting CPU times for the offline phases in this
paper, due to the very different levels of the implementations used for
the various methods, we would like to mention that in this example
the construction of the POD basis with the fairly sophisticated \matlab~integrator {\tt ode15s}  takes roughly 1.5
more CPU time than constructing the TQB-IRKA reduced-order model with
our vanilla implementation.

Between TQB-IRKA and BT, TQB-IRKA gives a marginally better reduced-order system as compared to BT for $i_0(t) = 50(\sin(2\pi t)-1)$, but still both are very competitive. In contrast, the one-sided and two-sided interpolatory projection methods produce unstable reduced-order systems {and are therefore omitted from the figures.} 

To test the robustness of the obtained reduced-order systems, we choose a different $i_0(t) = 5\cdot 10^4t^3 \exp(-15t)$ than the training one, and compare the transient responses in \Cref{fig:Fitz2}. In the figure, we observe that BT performs the best among all methods for  $i_0(t) = 5\cdot 10^4t^3 \exp(-15t)$, and POD and TQB-IRKA produce reduced-order systems of almost the same quality. One-sided and two-sided projection result in  unstable reduced-order systems for $i_0(t) = 5\cdot 10^4t^3 \exp(-15t)$ as well.  Furthermore, we also show the limit-cycles on the full space  obtained from the original and reduced-order systems in \Cref{fig:Fitz_limitcycle} for $i_0(t) = 5\cdot 10^4t^3 \exp(-15t)$, and observe that the reduced-order systems obtained using POD, TQB-IRKA, and BT, enable us to reproduce the limit-cycles, which is a typical neuronal dynamics as shown in \Cref{fig:Fitz,fig:Fitz_limitcycle}

As shown in~\cite{morBenB12a}, for particular interpolation points and higher-order moments, it might be possible to construct reduced-order systems via one-sided and two-sided interpolatory projection methods, which can reconstruct the limit-cycles. But as  discussed in~\cite{morBenB12a}, stability of the reduced-order systems is highly sensitive to these specific choices and even  slight modifications may lead to unstable systems. For the $\cH_2$ linear optimal interpolation points selection we made here, the one-sided and two-sided approaches were not able to reproduce the limit-cycles; thus motivating the usage of TQB-IRKA and BT once again, especially for the MIMO case.

Moreover, we report how far the reduced system of order $r = 35$ due to TQB-IRKA is from satisfying the optimality conditions~\eqref{eq:optimalityconditions}. For this, we compute the perturbations~\eqref{eq:perturb_def} and list them in \Cref{tab:per_opt_fitz}. This clearly indicates that the reduced-order system almost satisfies all optimality conditions.

\begin{table}[!tbp]\label{tab:per_opt_fitz}
	\begin{center}
		\begin{tabular}{| c| c| c| c| c|}
			\hline
			$\cE_C$&$\cE_B$&$\cE_N$&$\cE_H$&$\cE_\lambda$ \\ \hline
			$8.76\times 10^{-8}$ & $7.35\times 10^{-9}$ & $1.78\times 10^{-11}$   & $4.27\times 10^{-9}$ & $9.14\times 10^{-10}$\\
			\hline
		\end{tabular}
	\end{center}
	\caption{The FitzHugh-Nagumo system: perturbations to the optimality conditions.}
\end{table}

Lastly, we  measure the truncated $\cH_2$-norm of the error systems, using the reduced-order systems obtained via different methods of various orders.   We plot the relative truncated $\cH_2$-norm of the error systems in \Cref{fig:fitz_h2norm}. We observe that TQB-IRKA produces better reduced-order systems with respect to a truncated $\cH_2$-norm as compared to BT and one-sided projection. Furthermore, since we require  stability of the matrix $\hA$ in the reduced QB system~\eqref{sys:QBsysRed} to be able to compute a truncated $\cH_2$-norm of the error systems, we could not achieve this in the case of two-sided projection. For POD, we preserve the cubic nonlinearity in the reduced-order system; hence, the truncated $\cH_2$-norm definition in \Cref{lemma:H2_norm} does not apply for such systems. Thus, we cannot compute a truncated $\cH_2$-norm of the error system in the cases of   the two-sided projection and POD, thereby these methods are not included in \Cref{fig:fitz_h2norm}.

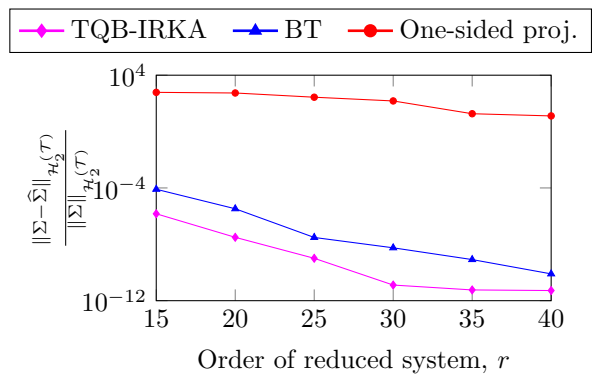
\begin{figure}[!tb]
	\centering
	\begin{tikzpicture}
	\begin{customlegend}[legend columns=-1, legend style={/tikz/every even column/.append style={column sep=0.25cm}} , legend entries={TQB-IRKA,BT, One-sided proj. }, ]
	\addlegendimage{mycolor1,solid,line width = 1pt, mark = diamond*}      \addlegendimage{blue,line width = 1pt, mark = triangle*}      \addlegendimage{red,line width = 1pt, mark = *}      \addlegendimage{mycolor2,line width = 1pt,mark =square* }
	\end{customlegend}
	\end{tikzpicture}
	\centering
	\setlength\fheight{3cm}  \setlength\fwidth{5.25cm}
	\tikzsetnextfilename{Figures_Revised/Fitz_RelativeH2_New}%
	\input{Figures_Revised/Fitz_RelativeH2_New.tikz}%

	\caption{The FitzHugh-Nagumo system: comparison of the truncated $\cH_2$-norm of the error system, having obtained reduced systems of different orders using various methods. }
	\label{fig:fitz_h2norm}
\end{figure}

%% file: Figures_Revised/Chafee_RelativeH2.tikz
%
%
%

%
\begin{tikzpicture}

\begin{axis}[%
width=\fwidth,
height=\fheight,
scale only axis,
xmin=4,
xmax=24,
ymode=log,
ymin=1e-14,
ymax=100,
yminorticks=true,
ylabel = { $\tfrac{\|\Sigma-\hat\Sigma\|_{\cH_2^{(\cT)}}}{\|\Sigma\|_{\cH_2^{(\cT)}}}$},
xlabel = {Order of reduced system, $r$}
]
\addplot [color=mycolor1,solid,forget plot,  mark size=1.5pt,  mark = diamond*, line width = 1pt]
  table[row sep=crcr]{4	0.0142790534457948\\
6	0.000494360771202061\\
8	2.32206238369817e-05\\
10	1.41154826610375e-05\\
12	3.74762866345234e-07\\
14	3.74772673699468e-07\\
16	9.35270918460571e-09\\
18	7.37445428262208e-06\\
20	7.98724220692279e-10\\
22	5.96569307638348e-10\\
24	7.94478120466763e-12\\
};
\addplot [color=blue,solid,forget plot, mark size=1.5pt, mark = triangle*, line width = 1pt]
  table[row sep=crcr]{4	1.4718240579928e-05\\
6	1.47182405799391e-05\\
8	3.92120211052812e-07\\
10	3.92120211053404e-07\\
12	9.79200164536944e-09\\
14	9.79200164539788e-09\\
16	9.79200164514418e-09\\
18	2.21501036076859e-10\\
20	2.21501093477783e-10\\
22	4.25726564425091e-12\\
24	4.25744086464647e-12\\
};
\addplot [color=red,solid,forget plot,  mark size=1.2pt, mark = *, line width = 1pt]
  table[row sep=crcr]{4	71.4134724891061\\
6	44.967580555587\\
8	32.5523233760255\\
10	24.8721618622746\\
12	19.5499172200231\\
14	15.6189735896484\\
16	12.5963592904816\\
18	10.2078856343231\\
20	8.28569722935076\\
22	6.74445939263716\\
24	5.52648450040963\\
};
\addplot [color=mycolor2,solid, mark size=1.0pt, mark = square*, line width = 1pt]
  table[row sep=crcr]{4	0.0187179264757037\\
6	0.00107592285493758\\
8	6.41162167078835e-05\\
10	3.88156655858044e-06\\
12	2.35611544296551e-07\\
14	1.42964252809808e-08\\
16	8.66084746591953e-10\\
18	5.22018091577773e-11\\
20	3.15437903276381e-12\\
22	1.28587612212392e-14\\
24	2.82093500787423e-14\\
};
\end{axis}
\end{tikzpicture}%

%% file: Figures_Revised/RC_RelativeH2.tikz
%
%
%
%
\begin{tikzpicture}

\begin{axis}[%
width=\fwidth,
height=\fheight,
scale only axis,
xmin=6,
xmax=20,
ymode=log,
ymin=1e-15,
ymax=.1,
yminorticks=true,
ylabel = { $\tfrac{\|\Sigma-\hat\Sigma\|_{\cH_2^{(\cT)}}}{\|\Sigma\|_{\cH_2^{(\cT)}}}$},
xlabel = {Order of reduced system, $r$}
]
\addplot [color=mycolor1,solid,forget plot,  mark size=1.5pt,  mark = diamond*, line width = 1pt]
  table[row sep=crcr]{6	3.74121635367296e-07\\
8	1.62618126623308e-08\\
10	6.07255045047599e-10\\
12	3.98548431507971e-11\\
14	3.78701387038945e-12\\
16	1.43124791883027e-13\\
18	1.88750952246869e-14\\
20	1.69589972349558e-15\\
};
\addplot [color=blue,solid,forget plot, mark size=1.5pt, mark = triangle*, line width = 1pt]
  table[row sep=crcr]{6	1.729692439564e-05\\
8	2.52049118410186e-06\\
10	8.78600234086319e-08\\
12	3.12201471213223e-09\\
14	3.73204040718169e-10\\
16	1.42772087541737e-11\\
18	1.76651271954882e-12\\
20	7.24804950595815e-14\\
};
\addplot [color=red,solid,forget plot,  mark size=1.2pt, mark = *, line width = 1pt]
  table[row sep=crcr]{6	0.0184707044970682\\
8	0.0141236738706747\\
10	0.000103911533342528\\
12	0.000158707499471461\\
14	2.20773762020505e-05\\
16	7.07382352660754e-06\\
18	1.36110842252915e-06\\
20	2.25842970196528e-07\\
};
\addplot [color=mycolor2,solid, mark size=1.0pt, mark = square*, line width = 1pt]
  table[row sep=crcr]{6	1.8258405594116e-05\\
8	8.96979037181765e-07\\
10	2.45267090638633e-08\\
12	7.69327628731782e-09\\
14	7.18852201957828e-10\\
16	1.37609131425467e-11\\
18	2.25168735427191e-13\\
20	2.78115769876711e-14\\
};
\end{axis}
\end{tikzpicture}%

%% file: Figures_Revised/Fitz_RelativeH2_New.tikz
%
%
%
%
\begin{tikzpicture}

\begin{axis}[%
width=\fwidth,
height=\fheight,
scale only axis,
xmin=15,
xmax=40,
ymode=log,
ymin=1e-12,
ymax=10000,
yminorticks=true,
ylabel = { $\tfrac{\|\Sigma-\hat\Sigma\|_{\cH_2^{(\cT)}}}{\|\Sigma\|_{\cH_2^{(\cT)}}}$},
xlabel = {Order of reduced system, $r$}
]
\addplot [color=mycolor1,solid,forget plot,  mark size=1.5pt,  mark = diamond*]
  table[row sep=crcr]{15	1.46804713646712e-06\\
20	3.13146496876534e-08\\
25	1.01659612775136e-09\\
30	1.28643385628777e-11\\
35	5.78933559191141e-12\\
40	5.23896039888556e-12\\
};
 \addplot [color=blue,solid,forget plot, mark size=1.5pt, mark = triangle*]
  table[row sep=crcr]{15	8.05459419391671e-05\\
20	3.32979263924172e-06\\
25	3.06567590061616e-08\\
30	5.64627803911906e-09\\
35	8.25173882476676e-10\\
40	7.97840480896847e-11\\
};
\addplot [color=red,solid,forget plot,  mark size=1.2pt, mark = *]
  table[row sep=crcr]{15	606.278715687453\\
20	546.267387445325\\
25	274.016142411122\\
30	149.01357910024\\
35	18.4423763496779\\
40	12.9005315470788\\
};
\end{axis}
\end{tikzpicture}%

%% file: conclusion.tex
In this paper, we have investigated the optimal model reduction problem for quadratic-bilinear control systems.   We first have defined the  $\cH_2$-norm for quadratic-bilinear systems based on the  kernels of the underlying Volterra series and also proposed a tru

ncated $\cH_2$-norm for the latter system class. We have then derived  first-order necessary conditions, minimizing a truncated $\cH_2$-norm of the error system. These optimality conditions lead to the proposed model reduction algorithm (TQB-IRKA), which iteratively constructs reduced order models  that \emph{approximately} satisfy the optimality conditions. We have also discussed the efficient computation of the reduced Hessian, utilizing the Kronecker structure of the Hessian of the QB system. Via several numerical examples, we have shown that 
TQB-IRKA outperforms the one-sided (including POD) and two-sided projection methods  in most cases, and is comparable to balanced truncation.   Especially for MIMO QB systems, TQB-IRKA and BT are the preferred methods of choice since the current framework of  two-sided subspace interpolatory projection method is only applicable to SISO systems and the extension of the one-sided interpolatory projection method to MIMO QB systems yields reduced models whose  dimension  increases quadratically with the number of inputs.  
Moreover, our numerical experiments reveal that the reduced systems via TQB-IRKA  and BT are more robust as compared to the one-sided and two-sided interpolatory projection methods in terms of stability of the reduced models although we do not have any theoretical justification of this observation yet.

As a future research topic, it would be interesting to use  sophisticated tools from  tensor theory to perform the computations related to Kronecker products efficiently and faster, thus  accelerating the iteration steps in TQB-IRKA.  Secondly, even though a stable random initialization has performed well in all of our numerical examples, a more educated initial guess could further improve the convergence  of TQB-IRKA.

%% file: Acknowledgements.tex
The authors would like to thank Dr.\ Tobias Breiten for providing the numerical examples, and \matlab ~implementations for one-sided and two-sided interpolatory projection methods.  We would also like to thank Dr.\ Patrick K{\"u}rschner for providing us \matlab~codes for solving  Lyapunov equations via ADI.  This research is supported by a research grant of the ``International Max Planck Research School (IMPRS) for Advanced Methods in Process and System Engineering (Magdeburg)".  The work of S. Gugercin was supported in part by NSF through Grant DMS-1522616 and by the Alexander von Humboldt Foundation. 

%% file: Appendix1.tex
\appendix
\begin{appendices}
 \renewcommand\thetable{\thesection\arabic{table}}
  \renewcommand\thefigure{\thesection\arabic{figure}}
  \section{Important relations of the Kronecker products}

  In this section, we provide some relations between Kronecker products, which will remarkably simplify the optimality conditions in \Cref{appen:proof}.
  \begin{lemma}{\cite[Lemma A.1]{morBenB12b}}\label{lemma:appenA1}
   Consider $f(x) \in \R^{s\times n}$, $A(y) \in \Rnn$, $G \in \R^{n\times q}$ with $x,y \in \R$, and let $\cL(y)$ be defined as
   $$\cL(y) =  -A(y)\otimes I_n - I_n\otimes A(y).$$
If the functions $f$ and $A$ are differentiable with respect to $x$ and $y$, respectively, then
\begin{multline*}
\dfrac{\partial}{\partial x} \left[ (\cI_s)^T \left(f(x)\otimes f(x)\right) \cL^{-1}(y) (G\otimes G ) \cI_q \right] \\
  =2(\cI_s)^T \left( \left(\dfrac{\partial}{\partial x}f(x)\right)\otimes f(x)\right) \cL^{-1}(y) (G\otimes G ) \cI_q.
\end{multline*}
Moreover, let $X,Y\in\Rnn$ be  symmetric matrices. Then,
\begin{align*}
 \dfrac{\partial}{\partial y}  \left[ \vecop{X}^T\cL^{-1}(y)\vecop{Y} \right] = 2\cdot \vecop{X}^T\cL^{-1}(y)\left(\dfrac{\partial}{\partial y} A(y)\otimes I_n  \right)\cL^{-1}(y)\vecop{Y}.
\end{align*}
 \end{lemma}

\begin{lemma}\label{lemma:rel_FPM}
 Let $\cF$, $\hat{\cF}$ be defined as follows:
 \begin{equation*}
  \begin{aligned}
   \cF &= \bbm I_n & \0\ebm \otimes \bbm I_n & \0\ebm\quad\text{and}\quad
   \hat{\cF} = \bbm \0 & I_{r}\ebm \otimes \bbm \0 & I_{r}\ebm,
   \end{aligned}
 \end{equation*}
 and consider a permutation matrix
  \begin{equation}\label{eq:Per_M1}
  \begin{aligned}
   M& = \bbm M_{nnr} & \0 \\ \0 & M_{rnr} \ebm,
  \end{aligned}
 \end{equation}
where $M_{pqr}$ is defined in~\eqref{eq:per_M}. Moreover, let the  two column vectors $x$ and $y$ be  partitioned as
\begin{equation*}
 x = \bbm x_1^T & x_2^T & x_3^T & x_4^T \ebm^T\quad \text{and}\quad y = \bbm y_1^T & y_2^T & y_3^T & y_4^T \ebm^T,
\end{equation*}
where $x_1,y_1\in \R^{n^2}$, $x_{\{2,3\}},y_{\{2,3\}} \in \R^{nr}$, and $x_4,y_4 \in \R^{r^2}$.
Then, the following relations hold:
 \begin{align}
  (\hat{\cF}\otimes \cF)T_{(n+r,n+r)}(M\otimes M)(x\otimes y) &=  T_{(n,r)}(x_3\otimes y_3),~~\mbox{and} \label{eq:hatFandFrel}\\
    (\hat{\cF}\otimes \hat{\cF})T_{(n+r,n+r)}(M\otimes M)(x\otimes y) &= T_{(r,r)}(x_4\otimes y_4),\label{eq:hatFhatFrel}
 \end{align}
where $T_{(n,m)}$ is also a permutation matrix  given by
\begin{align*}
  T_{(n,m)} &= I_m\otimes  \bbm I_m \otimes e^n_1,\ldots,I_m \otimes e^n_n \ebm \otimes I_n.
 \end{align*}
\end{lemma}
\begin{proof}
 Let us begin by considering the following equation:
  \begin{equation*}
 \begin{aligned}
  (\hat{\cF}\otimes {\cF})T_{(n+r,n+r)} &= \begin{bmatrix} \0& I_r \otimes \bbm \0& I_r\ebm \otimes \cF \end{bmatrix} \left(I_{n+r}\otimes \cG\right),
  \end{aligned}
 \end{equation*}
 where $\cG =  \bbm I_{n+r} \otimes e^{n+r}_1,\ldots,I_{n+r}\otimes e^{n+r}_{n+r} \ebm \otimes I_{n+r} $. Next, we split $I_{n+r}$ as $I_{n+r} = \bbm I_n & \0 \\ \0 & I_r \ebm$, leading to
  \begin{equation}\label{eq:F2f1P_prod}
  \begin{aligned}
   (\hat{\cF}\otimes {\cF})T_{(n+r,n+r)} &= \begin{bmatrix} \0& I_r \otimes \bbm \0& I_r\ebm \otimes \cF \end{bmatrix} \bbm I_n \otimes \cG & \0 \\ \0 & I_r\otimes \cG \ebm \\
   &= \bbm \0&  \begin{pmatrix}I_r \otimes \bbm \0 & I_r\ebm \otimes \cF \end{pmatrix} \begin{pmatrix} I_r\otimes \cG \end{pmatrix}\ebm \\
  &= \bbm \0&  I_r \otimes \Big( \left(\bbm \0 & I_r\ebm \otimes \cF \right)\cG \Big) \ebm.
   \end{aligned}
  \end{equation}
 Now, we  investigate the following equation (a component of the previous equation):
 \begin{equation*}
  \begin{aligned}
   \begin{pmatrix} \bbm \0 & I_r\ebm \otimes \cF \end{pmatrix}\cG_i =: \cL_i,
  \end{aligned}
 \end{equation*}
where $\cG_i$ is $i$th  block column of the matrix $\cG$ given by $\cG_i = I_{n+r} \otimes e^{n+r}_i\otimes I_{n+r}$. This  yields
 \begin{equation*}
  \begin{aligned}
   \cL_i &= \left(\bbm \0 & I_r\ebm \otimes \cF \right) \bpm I_{n+r}\otimes  e^{n+r}_i\otimes I_{n+r}\epm\\
   &= \bpm \bbm \0 & I_r\ebm I_{n+r} \epm \otimes \bpm \cF (e^{n+r}_i\otimes I_{n+r})\epm
   = \bbm  \0 & I_r\ebm  \otimes \bpm \cF (e^{n+r}_i\otimes I_{n+r})\epm.
  \end{aligned}
 \end{equation*}
 Assuming that $1\leq i\leq n$, we can write $\cL_i$ as
  \begin{equation*}
  \begin{aligned}
   \cL_i    &= \bbm  \0 & I_r\ebm  \otimes \bpm \cF \left(\bbm e^{n}_i \\ \0 \ebm \otimes I_{n+r}\right)\epm
   = \bbm  \0 & I_r\ebm  \otimes \bpm \bbm I_n\otimes \bbm I_n & \0  \ebm & \0 \ebm \bbm e^{n}_i\otimes I_{n+r} \\ \0 \ebm \epm\\
   &= \bbm  \0 & I_r\ebm  \otimes  \bbm e^n_i\otimes \bbm I_n & \0  \ebm\ebm
   = \bbm  \0 & I_r \otimes  \bpm e^n_i\otimes \bbm I_n & \0  \ebm \epm \ebm.
  \end{aligned}
 \end{equation*}
  Subsequently, we assume $n+r\geq i> n$, which leads to
  \begin{equation*}
  \begin{aligned}
   \cL_i    &= \bbm  \0 & I_r\ebm  \otimes \bpm \cF \left(\bbm \0 \\ e^{r}_{i-n}\ebm \otimes I_{n+r}\right)\epm\\
   &= \bbm  \0 & I_r\ebm  \otimes \bpm \bbm I_n\otimes \bbm I_n & \0  \ebm & \0 \ebm \bbm \0 \\ e^{r}_{i-n}\otimes I_{n+r} \ebm \epm = \0.
  \end{aligned}
 \end{equation*}
 Thus,
 \begin{equation*}
  \bpm \bbm \0 & I_r\ebm \otimes \cF \epm\cG = \bbm \cL_1 ,\cL_2,\ldots, \cL_n ,\0  \ebm =: \cL .
 \end{equation*}
Inserting the above expression in~\eqref{eq:F2f1P_prod} yields
\begin{equation*}
 (\hat{\cF}\otimes \cF)T_{(n+r,n+r)} = \bbm \0& I_r\otimes \cL  \ebm.
\end{equation*}
Now, we are ready to investigate the following term:
\begin{equation*}
\begin{aligned}
(\hat{\cF}\otimes \cF) T_{(n+r,n+r)} (M\otimes M) & = \bbm \0 &  I_r \otimes \cL \ebm\bbm M_{nnr}\otimes M & \0 \\ \0 & M_{rnr}\otimes M \ebm \\
&=  \bbm \0 &  I_r \otimes \cL \ebm \bbm M_{nnr}\otimes M & \0 \\ \0 & M_{rnr}\otimes M \ebm \\
&=  \bbm \0 &  \left(I_r \otimes\cL  \right)\left(M_{rnr}\otimes M \right) \ebm.
\end{aligned}
\end{equation*}
Further, we consider the second block column of the above relation and substitute for $M_{nnr}$ and $M_{rnr}$ using~\eqref{eq:per_M} to get
 \begin{equation}\label{eq:1}
 \begin{aligned}
 &\left(I_r \otimes \cL  \right)\left(M_{rnr}\otimes M \right)  =\left(I_r \otimes \cL  \right)\bbm I_r \otimes \bbm I_n \\ \0  \ebm \otimes M & I_r \otimes \bbm \0 \\I_r \ebm \otimes M \ebm \\ \allowdisplaybreaks
  &\qquad\qquad\quad = \bbm  \left(I_r \otimes \cL  \right) \left(I_r \otimes \bbm I_n \\ \0 \ebm \otimes M \right)&  \left(I_r \otimes\cL  \right)\left( I_r \otimes \bbm \0 \\I_r \ebm \otimes M\right) \ebm.
 \end{aligned}
 \end{equation}
Our following task is to examine each  block column of~\eqref{eq:1}.   We begin with the first block; this is
 \begin{equation*}
 \begin{aligned}
 \left(I_r \otimes \cL  \right) \left(I_r \otimes \bbm I_n \\ \0  \ebm  \otimes M \right) & = I_r \otimes \left(\cL  \left(  \bbm I_n \\ \0  \ebm \otimes M \right)\right) = I_r \otimes \left(\cL  \bbm I_n \otimes M\\ \0 \ebm  \right)\\
 &= I_r \otimes \bbm \cL_1M, \ldots ,\cL_nM \ebm.
   \end{aligned}
 \end{equation*}
We next aim to simplify the term $\cL_iM$, which appears in the previous equation:
\begin{align}
\cL_i M &= \left[ \0~~ I_r\otimes \bbm  e^{n}_j\otimes \bbm I_n~~\0  \ebm  \ebm \right]  \bbm M_{nnr} & \0 \\ \0 &M_{rnr} \ebm  \nonumber \\
&= \left[ \0~~ \left(I_r\otimes \bbm  e^{n}_j\otimes \bbm I_n~~\0  \ebm  \ebm \right)M_{rnr}\right]  \nonumber \\
&= \left[ \0~~ \left(I_r\otimes   e^{n}_j\otimes \bbm I_n~~\0  \ebm   \right)\bbm I_r \otimes \bbm I_n \\~~ \0 \ebm  & I_r \otimes \bbm \0\\I_r \ebm \ebm \right]  \nonumber\\
&= \left[ \0~~ \left(I_r\otimes   e^{n}_j\otimes I_n   \right) ~~\0 \right]  := \cX_i.\label{eq:xexpression}
\end{align}
The second  block column of~\eqref{eq:1} can be studied in a similar fashion, and it can be shown that
 \begin{equation*}
 \begin{aligned}
 \left(I_r \otimes \cL \right) \left(I_r \otimes \bbm \0 \\I_r \ebm \otimes M \right)& = \0.
   \end{aligned}
 \end{equation*}
Summing up all these expressions, we obtain
 \begin{equation*}
 \begin{aligned}
 (\hat{\cF}\otimes \cF) T_{(n+r,n+r)} (M\otimes M) & =  \bbm \0 &  \left(I_r \otimes \bbm \cX_1,\ldots,\cX_n \ebm   \right)  & \0 \ebm,
 \end{aligned}
 \end{equation*}
 where $\cX_i$ is defined in \eqref{eq:xexpression}. This gives
 \begin{equation}
 \begin{aligned}
   (\hat{\cF}\otimes \cF) T_{(n+r,n+r)}(M\otimes M)(x\otimes y)
&= \bbm \0 & I_r\otimes \bbm \cX_1,\ldots,\cX_n \ebm & \0 \ebm (x\otimes y)\\
   & = \left(I_r\otimes \bbm \cX_1,\ldots,\cX_n \ebm \right) (x_3\otimes y).\label{eq:F2F1PM_relation}
 \end{aligned}
\end{equation}
  Next, we define another permutation
  \begin{equation*}
 \cQ = \bbm \underbrace{I_{r}\otimes I_n\otimes \bbm I_{n^2} \\ \0\\\0\\\0\ebm}_{\cQ_1} & \underbrace{I_{r}\otimes I_n\otimes \bbm \0\\ I_{nr}\\ \0 \\ \0\ebm}_{\cQ_2} &\underbrace{I_{r}\otimes I_n\otimes \bbm \0 \\ \0\\I_{nr}\\0\ebm}_{\cQ_3} & \underbrace{I_{r}\otimes I_n\otimes \bbm \0 \\ \0\\ \0\\I_{r^2}\ebm}_{\cQ_4} \ebm, 
 \end{equation*}
which allows us to write
$$(x_3\otimes y) = \cQ\bbm x_3\otimes y_1\\x_3\otimes y_2\\x_3\otimes y_3\\x_3\otimes 	y_4\ebm.$$
Substituting this into \eqref{eq:F2F1PM_relation} results in
\begin{equation*}
 \begin{aligned}
   &(\hat{\cF}\otimes \cF) T_{(n+r,n+r)}(M\otimes M)(x\otimes y) \\
   & \qquad\qquad\qquad= \bpm  I_r\otimes \bbm \cX_1,\ldots,\cX_n \ebm  \epm \bbm \cQ_1&\cQ_2&\cQ_3&\cQ_4\ebm \bbm x_3\otimes y_1\\x_3\otimes y_2\\x_3\otimes y_3\\x_3\otimes 	y_4\ebm.
  \end{aligned}
  \end{equation*}
  Now, it can be easily verified that $\bpm  I_r\otimes \bbm \cX_1,\ldots,\cX_n \ebm  \epm \bbm \cQ_1&\cQ_2&\cQ_4\ebm = 0$.  Thus, we obtain
\begin{equation*}
\begin{aligned}   
   &(\hat{\cF}\otimes \cF) T_{(n+r,n+r)}(M\otimes M)(x\otimes y)  = \bpm  I_r\otimes \bbm \cX_1,\ldots,\cX_n \ebm  \epm \cQ_3 (x_3\otimes y_3)\\ 
   & \qquad\qquad\qquad= \bpm  I_r\otimes \bbm \cX_1,\ldots,\cX_n \ebm  \epm \bpm I_{r}\otimes I_n\otimes \bbm \0 \\ \0\\I_{nr}\\ \0\ebm \epm (x_3\otimes y_3)\\
   & \qquad\qquad\qquad= \left(I_r \otimes \bbm I_r\otimes   e^{n}_1\otimes I_n   , \ldots , I_r\otimes   e^{n}_1\otimes I_n \ebm     \right) \left(x_3\otimes y_3\right)
 = T_{(n,r)}(x_3\otimes y_3).
 \end{aligned}
\end{equation*}
One can prove the relation~\eqref{eq:hatFandFrel} in a similar manner. However, for the brevity of the paper, we omit it. This concludes the proof.
\end{proof}

We will find similar expressions as~\eqref{eq:hatFandFrel} and~\eqref{eq:hatFhatFrel} in Appendix B, where we  then make use of \Cref{lemma:rel_FPM} to simplify them.

 \end{appendices}

%% file: Appendix.tex
\begin{appendices}\label{appem:proofthm}
\renewcommand\thetable{\thesection\arabic{table}}
\renewcommand\thefigure{\thesection\arabic{figure}}
\section{Proof of Theorem $\text{\ref{thm:optimality_conditions}}$}\label{appen:proof}
\subsection*{Optimality conditions with respect to \texorpdfstring{\boldmath{$\tC$}}{\text{$\tC$}}}
We start with deriving the optimality conditions by taking the derivative of the error functional $\cE$~\eqref{eq:errexp2} with respect to $\tC$. By using \Cref{lemma:appenA1}, we obtain
\begin{equation*}
 \begin{aligned}
  \frac{\partial \cE^2}{\partial \tC_{ij}} & = 2(\cI_p)^T \left( \bbm \0 ~ -e^p_i(e^r_j)^T \ebm \otimes \tC^e\right)
   \left(-\tA^e \otimes I_{n+r}   - I_{n+r}\otimes  \tA^e \right)^{-1} \\
   &\Big( \left(\tB^e \otimes \tB^e \right)\cI_m
      + \sum_{k=1}^m\left( \tN^e_k \otimes  \tN^e_k\right) \cP_l   + \left(\tH^e \otimes \tH^e \right)T_{(n+r,n+r)} ( \cP_l\otimes  \cP_l) \Big),
  \end{aligned}
 \end{equation*}
 where $\cP_l$ is defined in~\eqref{eq:p1_exp}.
On simplification, we get
  \begin{align}
    \frac{\partial \cE^2}{\partial \tC_{ij}}     & = 2(\cI_p)^T \left(   -e^p_i(e^r_j)^T  \otimes \tC^e \right)   \left( -\Lambda  \otimes I_{n+r}   - I_{r}\otimes \tA^e \right)^{-1}\Big(   \left(\tB  \otimes \tB^e \right)\cI_m \nonumber\\
    &\quad +  \sum_{k=1}^m \left(\tN_k  \otimes \tN^e_k\right) \cP^{(2)}_l +   \left(\tH \hat{ \cF}  \otimes \tH^e \right) T_{(n+r,n+r)} (\cP_l \otimes \cP_l)\Big),\nonumber\\
   & = 2(\cI_p)^T \left(   -e^p_i(e^r_j)^T  \otimes \tC^e \right)\nonumber
    \left( M_{rnr} \left( - \cJ_\Lambda  - \cJ_A \right) M_{rnr}^T \right)^{-1}\Big(  \left( \tB  \otimes \tB^e \right)\cI_m \nonumber\\
    &\quad +  \sum_{k=1}^m \left(\tN_k  \otimes \tN^e_k\right) \cP^{(2)}_l +  \left( \tH \hat{\cF}  \otimes \tH^e \right)T_{(n+r,n+r)}(\cP_l \otimes \cP_l)\Big), \label{eq:derC}
 \end{align}
where
\begin{align*}
 \cJ_\Lambda = \bbm \Lambda\otimes I_n & \0  \\ \0  &  \Lambda\otimes I_{r} \ebm,\qquad \cJ_A = \bbm I_{r}\otimes A & \0  \\ \0  &I_{r}\otimes \Lambda \ebm, ~~\text{and}
\end{align*}
$\cP_l^{(2)}$ is the lower  block row of $\cP_l$ as shown in~\eqref{eq:p1_exp}. Furthermore, since $M_{rnr}$ is a permutation matrix, this implies $M_{rnr}M_{rnr}^T = I$. Using this relation in~\eqref{eq:derC}, we obtain
\begin{align}
    \frac{\partial \cE^2}{\partial \tC_{ij}}     &  = 2(\cI_p)^T \left(   -e^p_i(e^r_j)^T  \otimes \bbm C&-\tC\ebm \right)M_{rnr}
    \left(  - \cJ_\Lambda  - \cJ_A \right)^{-1}\Big(  M_{rnr}^T \left(\tB  \otimes \tB^e \right)\cI_m \nonumber\\
    &\quad + M_{rnr}^T \sum_{k=1}^m \left(\tN_k  \otimes \tN^e_k\right) \cP^{(2)}_1 +  M_{rnr}^T\left( \tH \hat{\cF}  \otimes \tH^e \right)T_{(n+r,n+r)}(\cP_l \otimes \cP_l)\Big) \nonumber\\
    &= 2(\cI_p)^T \left(   \bbm  -e^p_i(e^r_j)^T \otimes C & e_ie_j^T  \otimes\tC\ebm \right)
    \left(  - \cJ_\Lambda  - \cJ_A \right)^{-1}\left(   \bbm \tB  \otimes B \\ \tB\otimes \tB \ebm \cI_m \nonumber\right.\nonumber\\
    &\left.\quad +  \sum_{k=1}^m  \bbm \tN_k  \otimes N_k & \0  \\ \0 & \tN_k\otimes \tN_k  \ebm  M_{rnr}^T\cP^{(2)}_1 \right.\nonumber\\ \allowdisplaybreaks
    &\left. \quad +  \bbm \left(\tH \hat{\cF} \otimes H\cF \right)T_{(n+r,n+r)}(M\otimes M)(M^T\otimes M^T)(\cP_l \otimes \cP_l) \\ \left(\tH\hat{\cF} \otimes \tH\hat{\cF}\right)T_{(n+r,n+r)}(M\otimes M) (M^T\otimes M^T)(\cP_l \otimes \cP_l) \ebm\right) ,\label{eq:secondlasteq_C}
 \end{align}
where $M$ is the permutation matrix  defined in \eqref{eq:Per_M1}. The multiplication of $M^T$  and $\cP_l$ yields
\begin{equation*}
 \begin{aligned}
  M^T\cP_l =  \bbm M_{nnr} \cP_l^{(1)} \\ M_{rnr}\cP^{(2)}_1 \ebm =   \bbm p_1^T&p_2^T& p_3^T& p_4^T \ebm ^T =: \tilde{\cP_l} ,
 \end{aligned}
\end{equation*}
where
\begin{equation}\label{eq:def_P1234}
 \begin{aligned}
 p_1& = \left(  -A \otimes I_n  - I_n \otimes A  \right)^{-1}\left(  B\otimes B\right) \cI_m,&
 p_2& = \left(  -A \otimes I_r  - I_n \otimes \Lambda  \right)^{-1}\left(  B\otimes \tB\right) \cI_m, \\
  p_3& = \left(  -\Lambda \otimes I_n  - I_r \otimes A  \right)^{-1}\left(  \tB\otimes B\right) \cI_m, &
 p_4& = \left(  -\Lambda \otimes I_r  - I_r \otimes \Lambda  \right)^{-1}\left(  \tB\otimes \tB\right) \cI_m. \\
 \end{aligned}
\end{equation}
Moreover, note that $p_3 = \vecop{V_1}$, where $V_1$ solves \eqref{eq:solveV1}.  Applying the result of \Cref{lemma:rel_FPM} in \eqref{eq:secondlasteq_C} yields
 \begin{align}
   \frac{\partial \cE^2}{\partial \tC_{ij}}   &=    2(\cI_p)^T \left(   e^p_i(e^r_j)^T  \otimes  C  \right) \left(  -  \Lambda\otimes I_n    - I_{r}\otimes A    \right)^{-1} \Big(     (\tB  \otimes B )\cI_m +    \sum_{k=1}^m(\tN_k  \otimes N_k)  p_3 \nonumber \\
   &\qquad +    (\tH  \otimes H) T_{(n,r)} p_3\otimes p_3\Big)-
      2(\cI_p)^T \left(   e^p_i(e^r_j)^T  \otimes  \tC  \right) \left(   - \Lambda\otimes I_n    - I_{r}\otimes \tA    \right)^{-1} \nonumber\\
      &\qquad \times \Big(    (\tB  \otimes \tB )\cI_m  +   \sum_{k=1}^m( \tN_k  \otimes \tN_k ) p_4  + (\tH   \otimes \tH) T_{(r,r)}  (p_4 \otimes p_4 \Big) \nonumber\\
      &=    2(\cI_p)^T \left(   e^p_i(e^r_j)^T  \otimes  C  \right) \left(  -  \Lambda\otimes I_n    - I_{r}\otimes A    \right)^{-1} \Big(     (\tB  \otimes B )\cI_m +    \sum_{k=1}^m (\tN_k  \otimes N_k)  p_3  \nonumber\\
   &\qquad +    (\tH  \otimes H) T_{(n,r)} p_3\otimes p_3 \Big)- 2
      (\cI_p)^T \left(   e^p_i(e^r_j)^T \otimes  \hC  \right) \left(   - \Lambda\otimes I_n    - I_{r}\otimes \hA    \right)^{-1} \nonumber\\
      &\qquad \times \Big(    (\tB  \otimes \hB )\cI_m  +  \sum_{k=1}^m ( \tN_k  \otimes \hN_k ) \hp_4  + (\tH   \otimes \hH) T_{(r,r)}  (\hp_4 \otimes \hp_4 \Big),\label{eq:lasteq_C}
 \end{align}
where $ \hp_4 = \left(  -\Lambda \otimes I_r  - I_r \otimes \hA  \right)^{-1}\left(  \tB\otimes \hB\right) \cI_m = \vecop{\hV_1}$, where $\hV_1$ is as defined in \eqref{eq:def_VWhat}. Setting \eqref{eq:lasteq_C} equal to zero results in a necessary condition with respect to $\tC$ as follows:
   \begin{equation}\label{eq:first_condition}
   \begin{aligned}
   & (\cI_p)^T \left(   e^p_i(e^r_j)^T \otimes  C  \right) \left(  -  \Lambda\otimes I_n    - I_{r}\otimes A    \right)^{-1} \Big(     (\tB  \otimes B )\cI_m +    \sum_{k=1}^m (\tN_k  \otimes N_k)  p_3  \\
   &\qquad +    (\tH  \otimes H) T_{(n,r)} (p_3\otimes p_3) \Big)\\
     & = (\cI_p)^T \left(   e^p_i(e^r_j)^T \otimes  \hC  \right) \left(   - \Lambda\otimes I_n    - I_{r}\otimes \hA    \right)^{-1} \\
      &\qquad \times \Big(    (\tB  \otimes \hB )\cI_m  +    \sum_{k=1}^m( \tN_k  \otimes \hN_k ) \hp_4  + (\tH   \otimes \hH) T_{(r,r)}  (\hp_4 \otimes \hp_4 \Big).
\end{aligned}
      \end{equation}
Now, we first manipulate the left-side of the above  equation~\eqref{eq:first_condition}.  Using \Cref{lemma:change_kron} and \eqref{eq:tracepro}, we get
\begin{align*}
&(\cI_p)^T \left(   e^p_i(e^r_j)^T  \otimes  C  \right) \left(  -  \Lambda\otimes I_n    - I_{r}\otimes A    \right)^{-1} \Big(     (\tB  \otimes B )\cI_m +    \sum_{k=1}^m (\tN_k  \otimes N_k)  p_3  \\
&\qquad +    (\tH  \otimes H) T_{(n,r)} (p_3\otimes p_3) \Big) \\
&= (\cI_p)^T \left(   e^p_i(e^r_j)^T  \otimes  C  \right) \left(-\Lambda\otimes I_n    - I_{r}\otimes A    \right)^{-1} \Big(  \vecop{B\tB^T} +    \sum_{k=1}^m \vecop{N_kV_1\tN_k^T}  \\
&\qquad +    (\tH  \otimes H) \vecop{V_1\otimes V_1} \Big) \displaybreak\\
&= (\cI_p)^T \left(   e^p_i(e^r_j)^T  \otimes  C  \right) \left(-\Lambda\otimes I_n    - I_{r}\otimes A    \right)^{-1}  \Big(\vecop{B\tB^T+    \sum_{k=1}^m N_kV_1\tN_k^T}    \\
&\qquad  +   \vecop{ H (V_1\otimes V_1)\tH^T }\Big) \\ 
&= (\cI_p)^T \left(   e^p_i(e^r_j)^T \otimes  C  \right) \left(\vecop{V_1}+\vecop{V_2}\right) = \trace{C(V_1+V_2)e^r_j(e^p_i)^T} \\&= \trace{CVe^r_j(e^p_i)^T},
\end{align*}
where $V_2$ solves \eqref{eq:solveV2} and $V = V_1+V_2$. Using the similar steps, we can show that the right-side of~\eqref{eq:first_condition} is equal to $\trace{\hC\hV e^r_j(e^p_i)^T}$, where $\hV$ is defined in~\eqref{eq:def_VWhat}. Therefore, \cref{eq:first_condition} is the same as  \eqref{eq:cond_C}.

 \subsection*{ Necessary conditions with respect to \boldmath{$\Lambda$}}
 By utilizing \Cref{lemma:appenA1}, we aim at deriving the necessary condition with respect to the $i$th diagonal entry of $\Lambda$. We differentiate $\cE$ w.r.t. $\lambda_i$ to obtain
\begin{equation*}
 \begin{aligned}
  \frac{\partial \cE^2}{\partial \lambda_{i}} & =  2(\cI_p)^T \left(\tC^e  \otimes \tC^e\right)  \cL_e^{-1}\E \cL_e^{-1}   \Big(  \left(\tB^e \otimes \tB^e \right)\cI_m  + \sum_{k=1}^m\left(\tN^e_k \otimes \tN^e_k\right) \cP_l   \\
  &\quad + \left(\tH^e \otimes \tH^e \right) T_{(n+r,n+r)} ( \cP_l \otimes  \cP_l ) \Big)   +  (\cI_p)^T \left( \tC^e\otimes \tC^e \right) \cL_e^{-1} \\
   &\quad \times \left( 2\sum_{k=1}^m\left(\tN^e_k \otimes \tN^e_k \right)\cL_e^{-1}\E\cP_l + 4  \left(\tH^e \otimes \tH^e\right) T_{(n+r,n+r)} \left( \left(\cL_e^{-1} \E\cP_l\right)\otimes \cP_l\right) \right),
 \end{aligned}
\end{equation*}
where
\begin{equation*}
 \cL_{e} = - \left(\tA^e\otimes I_{n+r}   + I_{n+r}\otimes \tA^e \right) \quad\text{and}\quad \quad \E = \bbm \0 & \0\\ \0 & e^r_i (e^r_i)^T \ebm\otimes I_{n+r}.
\end{equation*}
Performing some algebraic  calculations  gives rise to the following expression:
\begin{align*}
 \frac{\partial \cE^2}{\partial \lambda_{i}} & = 2(\cI_p)^T \left( -\tC \otimes \tC^e  \right) \cZ_e^{-1} \Xi_{n+r} \cZ_e^{-1}  \Big(   \left(\tB  \otimes \tB^e \right)\cI_m  +  \sum_{k=1}^m \left(\tN_k  \otimes \tN^e_k\right) \cP^{(2)}_1 \\
 &+   \left(\tH \hat{ \cF}  \otimes \tH^e \right) T_{(n+r,n+r)} (\cP_l \otimes \cP_l)\Big)
     + 2(\cI_p)^T \left( -\tC \otimes \tC^e  \right) \cZ_e^{-1} \\
     &\times\Big(  \sum_{k=1}^m \left(\tN_k  \otimes \tN^e_k\right)
    \cZ_e^{-1} \Xi_{n+r}\cP^{(2)}_1 +  2\left(\tH \hat{ \cF}  \otimes \tH^e \right) T_{(n+r,n+r)} (\cL^{-1}_e\E\cP_l \otimes \cP_l)\Big),
\end{align*}
where $\cZ_e := -\left( \Lambda  \otimes I_{n+r}   + I_{r}\otimes A^e \right)$ and $\Xi_m := (e^r_i(e^r_i)^T\otimes I_m)$. Next, we utilize \Cref{lemma:rel_FPM} and use the permutation matrix $M$ (as done while deriving the necessary conditions with respect to $\tC$) to obtain
{\small
\begin{align*}
&\frac{\partial \cE^2}{\partial \lambda_{i}} \\
& ~~=   2(\cI_p)^T \cS  \Big( (\tB  \otimes B )\cI_m +    \sum\nolimits_{k=1}^m(\tN_k  \otimes N_k)  p_3   +    (\tH  \otimes H) T_{(n,r)} (p_3\otimes p_3) \Big) \\ 
&~~~-2(\cI_p)^T \tilde{\cS}\Big( (\tB  \otimes \tB )\cI_m +    \sum\nolimits_{k=1}^m(\tN_k  \otimes \tN_k)  p_4   +    (\tH  \otimes \tH) T_{(r,r)} (p_4\otimes p_4) \Big) \\
&~~~+2(\cI_p)^T \left(  \tC  \otimes  C  \right) L ^{-1} \Big(    \sum\nolimits_{k=1}^m(\tN_k  \otimes N_k)L ^{-1} \Xi_n p_3   +   2(\tH  \otimes H) T_{(n,r)} (L ^{-1} \Xi_np_3\otimes p_3) \Big)\\
&~~~-2(\cI_p)^T \left(  \tC  \otimes  \tC  \right) \tL ^{-1} \Big(   \sum\nolimits_{k=1}^m (\tN_k  \otimes \tN_k)\tL ^{-1} \Xi_r p_4   +  2 (\tH  \otimes \tH) T_{(r,r)} (\tL ^{-1} \Xi_r p_4\otimes p_4) \Big)\displaybreak\\[5pt]
& ~~=   2(\cI_p)^T \cS \Big( (\tB  \otimes B )\cI_m +    \sum\nolimits_{k=1}^m(\tN_k  \otimes N_k)  p_3   +    (\tH  \otimes H) T_{(n,r)} (p_3\otimes p_3) \Big) \\
&~~~-2(\cI_p)^T \hat{\cS}\Big( (\tB  \otimes \hB )\cI_m +    \sum\nolimits_{k=1}^m(\tN_k  \otimes \hN_k)  \hp_4   +    (\tH  \otimes \hH) T_{r,r} (\hp_4\otimes \hp_4) \Big)  \\
&~~~+2(\cI_p)^T \left(  \tC  \otimes  C  \right) L ^{-1} \Big(    \sum\nolimits_{k=1}^m(\tN_k  \otimes N_k)L ^{-1} \Xi_n p_3   +   2(\tH  \otimes H) T_{(n,r)} (L ^{-1} \Xi_n p_3\otimes p_3) \Big)\\
&~~~-2(\cI_p)^T \left(  \tC  \otimes  \hC  \right) \hL ^{-1} \Big(    \sum\nolimits_{k=1}^m(\tN_k  \otimes \hN_k)\hL ^{-1} \Xi_r \hp_4   +  2 (\tH  \otimes \hH) T_{(r,r)} (\hL ^{-1} \Xi_r(\hp_4\otimes \hp_4) \Big),
\end{align*}
where $p_3$ and $p_4$ are the same as defined in~\eqref{eq:def_P1234}, and
\begin{equation*}
\begin{aligned}
\cS &: = \left(  \tC  \otimes  C  \right) L ^{-1} (e^r_i(e^r_i)^T\otimes I_n) L^{-1},&
\tilde{\cS} &:= \left(  \tC  \otimes  \tC  \right) \tL ^{-1} (e^r_i(e^r_i)^T\otimes I_r) \tL^{-1}, \\
\hat{\cS} &:= \left(  \tC  \otimes  \hC  \right) \hL ^{-1} (e^r_i(e^r_i)^T\otimes I_r) \hL^{-1},&
 L&:= -\left(  \Lambda \otimes I_n  + I_r \otimes A  \right),\\
 \tL &:= -\left(  \Lambda \otimes I_r  + I_r \otimes \Lambda  \right),&
  \hL &:= -\left(  \Lambda \otimes I_r  + I_r \otimes \hA  \right).
 \end{aligned}
\end{equation*}
By using properties derived in \Cref{lemma:trace_property}, we can  simplify the above equation
\begin{align*}
    \frac{\partial \cE^2}{\partial \lambda_{i}} & =   2(\cI_p)^T \cS \Big( (\tB  \otimes B )\cI_m +    \sum\nolimits_{k=1}^m(\tN_k  \otimes N_k)  p_3   +    (\tH  \otimes H) T_{(n,r)} (p_3\otimes p_3) \Big) \\
    &~~-2(\cI_p)^T\hat{\cS} \Big( (\tB  \otimes \hB )\cI_m +   \sum\nolimits_{k=1}^m (\tN_k  \otimes \hN_k)  p_4   +    (\tH  \otimes \hH) T_{(r,r)} (\hp_4\otimes \hp_4) \Big) \\
     &~~+2(\cI_m)^T \left(  \tB  \otimes  B  \right) L ^{-T} \Xi_n L^{-T}\Big(    \sum_{k=1}^m(\tN_k  \otimes N_k)^Tq_3   +   2(\tilde{\cH}^{(2)}  \otimes \cH^{(2)}) T_{(n,r)} (p_3\otimes q_3) \Big)\\
     &~~-2(\cI_m)^T \left(  \tB  \otimes  \hB  \right) \hL ^{-T} \Xi_r \hL^{-T} \Big(  \sum_{k=1}^m(\tN_k  \otimes \hN_k)^T\hq_4   + 2  (\tilde{\cH}^{(2)}  \otimes \hat{\cH}^{(2)}) T_{(r,r)} ( \hp_4\otimes \hq_4) \Big),
\end{align*}
where
\begin{equation*}
 \begin{aligned}
  q_3& = \left(  -\Lambda \otimes I_n  - I_r \otimes A  \right)^{-T}\left(  \tC\otimes C\right) \cI_p &\text{and}\quad
 \hq_4& = \left(  -\Lambda \otimes I_r  - I_r \otimes A_r  \right)^{-T}\left(  \tC\otimes \hC\right) \cI_p. \\
 \end{aligned}
\end{equation*}
  }
\noindent
Once again, we determine an interpolation-based necessary condition with respect to $\Lambda_i$ by setting the last equation equal to zero:
{\small
\begin{equation}\label{eq:cond2}
\begin{aligned}
 &(\cI_p)^T \cS \Big( (\tB  \otimes B )\cI_m +   \sum_{k=1}^m (\tN_k  \otimes N_k)  p_3   +    (\tH  \otimes H) T_{(n,r)} (p_3\otimes p_3) \Big) \\
     &~~+ (\cI_m)^T \left(  \tB  \otimes  B  \right) L ^{-T} \Xi_n L^{-T}\Big(   \sum_{k=1}^m (\tN_k  \otimes N_k)^Tq_3   +   2(\tilde{\cH}^{(2)}  \otimes \cH^{(2)}) T_{(n,r)} (p_3\otimes q_3) \Big)\\
     &~~ =  (\cI_p)^T\hat{\cS} \Big( (\tB  \otimes \hB )\cI_m +    \sum_{k=1}^m(\tN_k  \otimes \hN_k)  p_4   +    (\tH  \otimes \hH) T_{(r,r)} (\hp_4\otimes \hp_4) \Big) ,\\
     &~~+ (\cI_m)^T \left(  \tB  \otimes  \hB  \right) \hL ^{-T} \Xi_r\hL^{-T} \Big(  \sum_{k=1}^m(\tN_k  \otimes \hN_k)^T\hq_4   +  2 (\tilde{\cH}^{(2)}  \otimes \hat{\cH}^{(2)}) T_{(r,r)} ( \hp_4\otimes \hq_4) \Big).
\end{aligned}
\end{equation}
}
Now, we first simplify the left-side of the above equation using \Cref{lemma:change_kron} and \eqref{eq:tracepro}.  We first focus of the first part of the left-side of \eqref{eq:cond2}. This yields

\begin{align*}
&(\cI_p)^T \cS \Big( (\tB  \otimes B )\cI_m +   \sum_{k=1}^m (\tN_k  \otimes N_k)  p_3   +    (\tH  \otimes H) T_{(n,r)} (p_3\otimes p_3) \Big) \\
& \qquad = (\cI_p)^T \left(  \tC  \otimes  C  \right) L ^{-1} (e^r_i(e^r_i)^T\otimes I_n) \\
&\quad\qquad \times L^{-1}\Big( (\tB  \otimes B )\cI_m +   \sum_{k=1}^m (\tN_k  \otimes N_k)  p_3   +    (\tH  \otimes H) T_{(n,r)} (p_3\otimes p_3) \Big)  \displaybreak \\
&\qquad =  \underbrace{(\cI_p)^T\left(  \tC  \otimes  C  \right) L ^{-1}}_{(\vecop{W_1})^T} (e^r_i(e^r_i)^T\otimes I_n) \vecop{V} = \trace{Ve^r_i(e^r_i)^TW^T_1} \\
&\qquad = \left(V_1(:,i)\right)^TW(:,i), =  \left(W_1(:,i)\right)^TV(:,i),
\end{align*}
where $W_1$ solves \eqref{eq:solveW1}. Analogously, we can show that 
\begin{equation}
\begin{aligned}
&(\cI_m)^T \left(  \tB  \otimes  B  \right) L ^{-T} \Xi_n L^{-T}\Big(   \sum_{k=1}^m (\tN_k  \otimes N_k)^Tq_3   +   2(\tilde{\cH}^{(2)}  \otimes \cH^{(2)}) T_{(n,r)} (p_3\otimes q_3) \Big) \\
&\qquad\qquad = \left(W(:,i)\right)^TV_1(:,i).
\end{aligned}
	\end{equation}
	Thus, the left-side of \eqref{eq:cond2} is equal to $\left(W(:,i)\right)^TV_1(:,i) + \left(W_1(:,i)\right)^TV(:,i)$. Using the similar steps, we can also show that the right-side of \eqref{eq:cond2} is equivalent to $\left(\hW(:,i)\right)^T\hV_1(:,i) + \left(\hW_1(:,i)\right)^T\hV(:,i)$. Thus, we obtain the optimality conditions with respect to $\Lambda$ as \eqref{eq:cond_lambda}.

The necessary conditions with respect to $\tB$, $\tN$ and $\tH$ can also be determined in a similar manner as for $\tC$ and $\lambda_i$. For  brevity of the paper, we skip detailed derivations; however, we state  final optimality conditions.  A necessary condition for optimality with respect to the $(i,j)$ entry of $\tN_k$ is
\begin{align*}
   &(\cI_p)^T \left(   \tC  \otimes  C  \right) L^{-1} \left(    (e^r_i(e_j^r)^T  \otimes N_k)  p_3   \right) = (\cI_p)^T \left( \tC  \otimes  \hC  \right) \hL^{-1} \left(        ( e^r_i(e^r_j)^T  \otimes \hN_k ) \hp_4     \right),
\end{align*}
which then yields~\eqref{eq:cond_N} in  the Sylvester equation form. A similar optimality condition with respect to the $(i,j)$ entry of $\tH$ is given by
\begin{align*}
   &(\cI_p)^T \left(   \tC  \otimes  C  \right) L^{-1} \left(    (e_i^r(e^{r^2}_j)^T  \otimes H)T_{(n,r)}  (p_3\otimes p_3)   \right)\\
      & \qquad\qquad= (\cI_p)^T \left( \tC  \otimes  \hC  \right) \hL^{-1} \left(        ( e_i^r(e^{r^2}_j)^T  \otimes \hH) T_{(r,r)}(\hp_4\otimes \hp_4)     \right),
\end{align*}
which can be equivalently described as~\eqref{eq:cond_H}. Finally, the necessary condition appearing with respect to the $(i,j)$ entry of $\tB$ is

\begin{align*}
   &(\cI_m)^T \left(   e^r_i(e^m_j)^T  \otimes  B  \right) L^{-T} \Big(     (\tC  \otimes C )\cI_p +   \sum_{k=1}^m (\tN_k  \otimes N_k)^T  q_3  \\
   &\qquad\qquad +   2(\tilde{\cH}^{(2)}  \otimes \cH^{(2)}) T_{(n,r)} (p_3\otimes q_3) \Big),\\
     & = (\cI_m)^T \left(   e^r_i(e^m_j)^T  \otimes  \hB  \right) \hL^{-T} \Big(    (\tC  \otimes \hC )\cI_p  +  \sum_{k=1}^m ( \tN_k  \otimes \hN_k )^T \hq_4  \\
     &\qquad\qquad + 2(\tilde{\cH}^{(2)}   \otimes \hat{\cH}^{(2)}) T_{(r,r)}  (\hp_4 \otimes \hq_4 \Big),
\end{align*}
which gives rise to~\eqref{eq:cond_B}.

 \end{appendices}

%% file: proof_red_optimality_rel.tex
\begin{appendices}
	\renewcommand\thetable{\thesection\arabic{table}}
	\renewcommand\thefigure{\thesection\arabic{figure}}
	
	\section{Proof of Theorem $\text{\ref{thm:opt_rom}}$}\label{appen:proof_opt}
	
	We begin by establishing a relationship between $V_1\in \Rnr,\hV_1 \in\Rrr$ and $V\in\Rnr$.  For this,  consider the Sylvester equation related to $V_1$
	\begin{equation}\label{eq:sys_v1}
	-V_1\Lambda - AV_1 = B\tB^T, 
	\end{equation}
	and an oblique projector $\Pi_{v} := V_1(W^TV_1)^{-1}W^T$. Then, we apply the projector $\Pi_{v}$ to the Sylvester equation~\eqref{eq:sys_v1} from the left to obtain
	\begin{align}
	&-V_1\Lambda - \Pi_{v}AV_1 = \Pi_{v} B\tB^T,~\mbox{and} \nonumber\\
	&-V_1\Lambda - \Pi AV_1 = (\Pi_v - \Pi)AV_1 + \Pi_{v} B\tB^T,\label{eq:full_v1}
	\end{align}
	where $\Pi := V(W^TV)^{-1}W^T$. Now, we recall that $\hV_1$ satisfies the Sylvester equation
	\begin{equation*}
	-\hV_1\Lambda - \hA\hV_1 = \hB\tB^T.
	\end{equation*}
	We next multiply it by $V$ from the left and substitute  for $\hA$ and $\hB$  to obtain
	\begin{equation}\label{eq:v1_red}
	-V\hV_1\Lambda - \Pi AV\hV_1 = \Pi B\tB^T.
	\end{equation}
Subtracting \eqref{eq:full_v1} from \eqref{eq:v1_red} yields
	\begin{equation*}
	\begin{aligned}
	&(V_1 - V\hV_1)\Lambda + \Pi A(V_1-V\hV_1) = (\Pi - \Pi_v)\left( AV_1 + B\tB^T \right).\\
	\end{aligned}
	\end{equation*}
	Since it is assumed that $\sigma(\hA) \cap \sigma(-\Pi A)  = \emptyset$, this implies  that $\Lambda\otimes I_n + I_r\otimes (\Pi A)$ is invertible. Therefore, we can write
	\begin{equation}\label{eq:rel_vv1}
	V_1 = V\hV_1 + \epsilon_{v},
	\end{equation}
	where  $\epsilon_{v}$ solves the Sylvester equation
	\begin{equation}\label{eq:epsilonv1}
	\epsilon_{v} \Lambda  + \Pi_v A \epsilon_{v} =  (\Pi - \Pi_v)\left( AV_1 + B\tB^T \right).
	\end{equation}
	Similarly, one can show that
	\begin{equation}\label{eq:rel_ww1}
	W_1 = W(W^TV)^{-T}\hW_1 + \epsilon_{w},
	\end{equation}
	where  $\epsilon_w$ solves
$$			\epsilon_w \Lambda + \Pi^T A^T \epsilon_w = (\Pi^T-\Pi_{w})(A^TW_1 + C^T\tC),$$
in which $\Pi_w := W_1(V^TW)V^T$.	Using~\eqref{eq:rel_vv1} and \eqref{eq:rel_ww1}, we obtain
	\begin{equation*}
	\begin{aligned}
	\hW_1(:,i)^T\hN_k\hV_1(:,j) &= \hW_1(:,i)^T (W^TV)^{-1}W^TN_kV\hV_1(:,j)\\
	&= \left(W_1(:,i) - \epsilon_{w}(:,i)\right)^T N_k \left(V_1(:,j) -\epsilon_v(:,j)\right) \\
	&= W_1(:,i)^T N_k V_1(:,j) -  \left(\epsilon_w(:,i)\right)^TN_k(  V_1(:,j)-\epsilon_v(:,j)) \\
	&\qquad 	-   \left(W_1(:,i)\right)^TN_k(\epsilon(:,j)),\\
	\end{aligned}
	\end{equation*}
	which is \eqref{eq:cond_N1} in \Cref{thm:optimality_conditions}. Similarly, one can prove \eqref{eq:cond_H1}. To prove \eqref{eq:cond_C1}, we consider the following Sylvester equation  for $V$:
	\begin{equation}
	\begin{aligned}
	V(-\Lambda) - AV &= B\tB^T + \sum_{k=1}^mN_kV_1\tN_k^T + H(V_1\otimes V_1)\tH^T.
	\end{aligned}
	\end{equation}
	Applying $\Pi$ to both sides of the above Sylvester equation yields
	\begin{equation}
	\begin{aligned}
	V \Big(I_r(-\Lambda) - \hA I_r\Big) &= V\Big(\hB\tB^T +  \cY  \Big),
	&\qquad\qquad
	\end{aligned}
	\end{equation}
	where $\cY = (W^TV)^{-1}W^T\left( \sum_{k=1}^m N_kV_1\tN_k^T + H(V_1\otimes V_1)\tH^T\right)$.
	This implies that
	\begin{equation}\label{eq:V_full}
	\begin{aligned}
	I_r(-\Lambda) - \hA I_r &= \hB\tB^T + \cY .
	&\qquad\qquad
	\end{aligned}
	\end{equation}
	Next, we consider the Sylvester equation for $\hV$,
	\begin{equation}\label{eq:V_red}
	\begin{aligned}
	\hV(-\Lambda) - \hA\hV &= \hB\tB^T + \sum_{k=1}^m\hN_k\hV_1\tN_k^T + \hH(\hV_1\otimes \hV_1)\tH^{T}.
	\end{aligned}
	\end{equation}
	We then subtract \eqref{eq:V_red} and \eqref{eq:V_full} to obtain
	\begin{equation*}
	\begin{aligned}
	(I_r-\hV)(-\Lambda) -\hA(I_r - \hV) &=  \sum_{k=1}^m(W^TV)^{-1}W^TN_k\left( V_1 - V\hV_1\right)\tN_k^T \\
	&\qquad + (W^TV)^{-1}W^T H \left(V_1\otimes V_1 - (V\hV_1 \otimes V\hV_1) \right)\tH^T.
	\end{aligned}
	\end{equation*}
	Substituting $V\hV_1$ from~\eqref{eq:rel_vv1} gives
	\begin{equation*}
	\begin{aligned}
	(I_r-\hV)(-\Lambda) -\hA(I_r - \hV) &=  \sum_{k=1}^m (W^TV)^{-1}W^TN_k\epsilon_v\tN_k^T \\
	&\qquad + (W^TV)^{-1}W^T H \left( \epsilon_v\otimes V_1+ V_1\otimes  \epsilon_v + \epsilon_v\otimes \epsilon_v \right)\tH^T.
	\end{aligned}
	\end{equation*}
Since $\Lambda$ contains the eigenvalues of $\hA$ and $\hA$ is stable, $\Lambda$ and $-\hA$ cannot have any common eigenvalues. Hence, the matrix  $\Lambda\otimes I_r + I_r\otimes\hA$ is invertible. Therefore the above Sylvester equations in $\Gamma:= \hV - I_r $ exists and solves

	\begin{align*}
	\Gamma_v\Lambda + \hA\Gamma_v &=  \sum_{k=1}^m (W^TV)^{-1}W^TN_k\epsilon_v\tN_k^T \\
	&\qquad + (W^TV)^{-1}W^T H \left( \epsilon_v\otimes V_1+ V_1\otimes  \epsilon_v + \epsilon_v\otimes \epsilon_v \right)\tH^T.
\end{align*}
To prove \eqref{eq:cond_C1}, we  observe that
	\begin{align*}
	\trace{\hC\hV e_i^r\left(e^p_j\right)^T } &= \trace{C V(I_r + \Gamma_{v}) e_i^r\left(e^p_j\right)^T} \\
	&= \trace{CVe_i^r\left(e^p_j\right)^T} + \trace{CV\Gamma_ve_i^r\left(e^p_j\right)^T}.
	\end{align*}
	Thus,
	\begin{equation*}
	\trace{CV e_i^r\left(e^p_j\right)^T} = \trace{\hC \hV e_i^r\left(e^p_j\right)^T} + \epsilon_C^{(i,j)}.
	\end{equation*}
	Analogously, we can prove that there exists $\Gamma_w$ such that
	$\hW = (W^TV)^{T} + \Gamma_{w}$ and it satisfies
\begin{align*}
 \Gamma_w \Lambda + \hA^T \Gamma_w &= V^T \left(\sum_{k=1}^mN^T_k\epsilon_w\tN_k + \cH^{(2)}(\epsilon_v\otimes (W_1+\epsilon_w) + V_1\otimes \epsilon_w)\left(\cH^{(2)}\right)^T \right).
\end{align*}
To prove \eqref{eq:cond_B1}, we  observe that
	\begin{equation*}
	\trace{\hB^T\hW e_i^r\left(e^m_j\right)^T } = \trace{B^TW(W^TV)^{-T} ((W^TV)^{T} + \Gamma_{v}) e_i^r\left(e^p_j\right)^T}.
	\end{equation*}
	Thus,
	\begin{equation*}
	\trace{\hB^T\hW e_i^r\left(e^m_j\right)^T } = \trace{B^TW^T + B^TW  (W^TV)^{-T} \Gamma_{w}) e_i^r\left(e^p_j\right)^T}.
	\end{equation*}
Since we now  know that $\hV = I_r + \Gamma_{v}$ and $\hW = (W^TV)^{T} + \Gamma_{w}$, hence, we get
\begin{equation}\label{eq:VW_fin}
 V\hV = V + V\Gamma_{v} \quad \text{and}\quad  W(W^TV)^{-T}\hW = W + W(W^TV)^{-T}\Gamma_{w}.
\end{equation}
We make use of \eqref{eq:VW_fin} to prove \eqref{eq:cond_lambda1} in the following:
\begin{align*}
 &(W_1(:,i))^T  V(:,i) +  \left(W_2(:,i)\right)^TV_1(:,i)\\
 &\qquad\qquad  = (W(:,i))^TV(:,i) -  (W_2(:,i))^T  V_2(:,i) \\
 &\qquad\qquad  = \left( W(W^TV)^{-T}\left(\hW(:,i) - \Gamma_{w}(:.i)\right) \right)^TV\left( \hV(:,i) - \Gamma_{v}(:,i)\right) \\
 &\qquad\qquad\qquad-  (W_2(:,i))^T  V_2(:,i) \\
 &\qquad\qquad  =   \left( \hW(:,i) - \Gamma_{w}(:.i) \right)^T\left( \hV(:,i) - \Gamma_{v}(:,i)\right)  -  (W_2(:,i))^T  V_2(:,i) \\
&\qquad\qquad= \left( \hW(:,i)\right)^T\hV(:,i) - \left( \hW(:,i) \right)^T\Gamma_v(:,i) - \left(\Gamma_{w}(:.i)\right)^T \left( \hV(:,i) - \Gamma_{v}(:,i)\right) \\
&\qquad\qquad\qquad  -  (W_2(:,i))^T  V_2(:,i) \\
&\qquad\qquad = (\hW_1(:,i))^T  \hV(:,i) +  \left(\hW_2(:,i)\right)^T\hV_1(:,i)  + \epsilon_\lambda^{(i)},
 \end{align*}
where
\begin{align*}
 \epsilon_\lambda^{(i)} &= -\left( \hW(:,i) \right)^T\Gamma_v(:,i) -  \left(\Gamma_{w}(:.i)\right)^T \left( \hV(:,i)- \Gamma_{v}(:,i)\right) \\
 &\qquad  -(W_2(:,i))^T  V_2(:,i)  +  (\hW_2(:,i))^T  \hV_2(:,i).
\end{align*}
This completes the proof.

\end{appendices}